\documentclass[11pt]{amsart}
\usepackage[T1]{fontenc}
\usepackage{txfonts}
\usepackage{amssymb,verbatim}
\usepackage[all]{xy}
\usepackage{subfig}
\usepackage{tikz}
\usepackage{color}
\usepackage{a4wide}
\usepackage{mathrsfs}
\usepackage{hyperref}
\usepackage{wrapfig}
\usetikzlibrary{arrows}
\DeclareMathAlphabet{\mathpzc}{OT1}{pzc}{m}{it}
\newcommand{\R}{\mathbb{R}}
\newcommand{\C}{\mathbb{C}}
\renewcommand\P{\mathbf{P}}
\newcommand{\Hex}{\mathbf{H}}
\newcommand\Z{\mathbb{Z}}

\newcommand{\N}{\mathbb{N}}
\newcommand{\Q}{\mathbb{Q}}
\renewcommand{\H}{\mathcal{H}}
\newcommand{\T}{\mathrm{T}}
\newcommand{\Tb}{\mathbf{T}}

\newcommand{\SL}{{\rm SL}}
\newcommand{\GL}{{\rm GL}}

\newcommand{\MCG}{\mathrm{Mod}}
\newcommand{\MF}{\mathcal{MF}}
\newcommand{\PMF}{\mathcal{PMF}}
\newcommand{\UE}{\mathcal{UE}}
\newcommand{\Cur}{\mathscr{C}}
\newcommand{\Aa}{\textrm{Area}}
\renewcommand{\S}{\mathbb{S}}
\newcommand{\Cc}{\mathcal{C}}

\newcommand{\Mm}{{\mathcal M}}

\newcommand{\Aff}{\mathrm{Aff}}

\newcommand{\RP}{\mathbb{RP}}
\newcommand{\CP}{\mathbb{CP}}
\newcommand{\eps}{\epsilon}

\newcommand{\inter}{\mathrm{int}}
\newcommand{\Lc}{\mathcal{L}}
\newcommand{\diam}{\mathrm{diam}}

\newcommand{\Ecal}{\mathcal{E}}
\newcommand{\Xcal}{\mathcal{X}}
\newcommand{\Gp}{\mathpzc{G}}
\newcommand{\Vp}{\mathpzc{V}}
\newcommand{\Ep}{\mathpzc{E}}
\newcommand{\Pb}{\mathbf{P}}
\newcommand{\Ap}{\mathpzc{A}}
\newcommand{\hor}{\mathbf{h}}

\newcommand{\ol}{\overline}

\newcommand{\lra}{\longrightarrow}
\newcommand{\ra}{\rightarrow}
\newcommand{\dist}{\mathbf{d}}

\newcommand{\Tcal}{\mathcal{T}}
\newcommand{\Fcal}{\mathcal{F}}

\newcommand{\id}{\mathrm{id}}
\newcommand{\Id}{\mathrm{Id}}
\newcommand{\vide}{\varnothing}

\newcommand{\lbd}{\lambdaup}

\newcommand{\del}{\delta}
\newcommand{\Del}{\Delta}
\newcommand{\hinv}{\tauup}
\newcommand{\Extm}{\mathrm{Ext}}
\newcommand{\distc}{\mathbf{d}^{\mathcal{C}}}

\newtheorem{Theorem}{Theorem}[section]
\newtheorem{Corollary}[Theorem]{Corollary}
\newtheorem{Lemma}[Theorem]{Lemma}
\newtheorem{Proposition}[Theorem]{Proposition}
\newtheorem{Remark}[Theorem]{Remark}
\newtheorem{Definition}[Theorem]{Definition}

\newtheorem{Claim}[Theorem]{Claim}

\newtheorem{MainTheorem}{Theorem}

\begin{document}

\title{Translation surfaces and the curve graph in genus two}
\author{Duc-Manh Nguyen}

\address{IMB Bordeaux-Universit\'e de Bordeaux \newline
351, Cours de la Lib\'eration \newline
33405 Talence Cedex \newline FRANCE}

\email{duc-manh.nguyen@math.u-bordeaux1.fr}
\date{\today}

\begin{abstract}
Let $S$ be a (topological) compact closed surface of genus two. We associate to each translation surface $(X,\omega) \in \Omega\Mm_2=\H(2)\sqcup\H(1,1)$ a subgraph $\hat{\Cc}_{\rm cyl}$ of the curve graph  of $S$. The vertices of this subgraph are free homotopy classes of curves which can be represented either by a simple closed geodesic, or by a concatenation of two parallel saddle connections (satisfying some additional properties) on $X$. The subgraph $\hat{\Cc}_{\rm cyl}$ is by definition  $\GL^+(2,\R)$-invariant. Hence it may be seen as the image of the corresponding Teichm\"uller disk in the curve graph.  We will show that $\hat{\Cc}_{\rm cyl}$ is always connected and has infinite diameter. The group ${\rm Aff}^+(X,\omega)$ of affine automorphisms  of $(X,\omega)$ preserves naturally $\hat{\Cc}_{\rm cyl}$, we show that ${\rm Aff}^+(X,\omega)$ is precisely the stabilizer of $\hat{\Cc}_{\rm cyl}$ in ${\rm Mod}(S)$. We also prove that $\hat{\Cc}_{\rm cyl}$ is Gromov-hyperbolic if  $(X,\omega)$ is
completely periodic in the sense of Calta.

It turns out that the quotient of $\hat{\Cc}_{\rm cyl}$ by $\Aff^+(X,\omega)$ is closely related to  McMullen's prototypes  in the case  $(X,\omega)$ is a Veech surface in $\H(2)$. We finally show that this quotient graph has finitely many vertices if and only if $(X,\omega)$ is a Veech surface for $(X,\omega)$ in both strata $\H(2)$ and $\H(1,1)$.
\end{abstract}

\maketitle

\section{Introduction}\label{sec:intro}

\subsection{Curve complex}
Let $S$ be compact surface. The {\em curve complex} of $S$ is a simplicial complex whose vertices are free homotopy classes of essential simple closed curves on $S$, and $k$-simplices are defined to be the sets of (homotopy classes of) $k+1$  curves that can be realized pairwise disjointly on $S$.  This complex was introduced by Harvey~\cite{Harvey81} in order to use its combinatorial structure to encode the asymptotic geometry of the Teichm\"uller space. It turns out that its geometry is intimately related to the geometry and topology of Teichm\"uller space (see {\em e.g.} \cite{Rafi14}). The curve complex has now become a central subject in Teichm\"uller Theory, Low Dimensional Topology, and Geometric Group Theory.  Note that this complex is quasi-isometric to its $1$-skeleton which is referred to as the {\em curve graph} of $S$. In  this paper we will denote the curve graph by $\Cc(S)$.\medskip

The Mapping Class Group $\MCG(S)$ naturally acts on the curve complex by isomorphisms. In most cases, all automorphisms of the curve complex are induced by elements of $\MCG(S)$ (see~\cite{Iva97,Luo00}). Based on this relation, topological and combinatorial properties of the curve complex have been used to study the mapping class group (\cite{Harer86, BestFuj02}). In~\cite{MasMin99}, Masur and Minsky showed that the curve graph  (and the curve complex) is Gromov-hyperbolic (see also~\cite{Bow06}). A stronger result, that is the hyperbolicity constant is independent of the surface $S$, has been proved recently simultaneously by several people~\cite{Aoug13,Bow14,CRS14,HPW15}. Its boundary at infinity has been studied by Klarreich~\cite{Kla99} and Hamenst\"adt~\cite{Ham06}. Those results have led to numerous applications and a fast growing literature on the subject. In particular, the hyperbolicity of the curve graph has been exploited in the resolution of the Ending Lamination Conjecture by Brock-Canary-
Minsky~\cite{BCM12}. For a nice survey on the curve complex and its applications we refer to~\cite{Bow-survey}.

\subsection{Teichm\"uller disk and translation surface}
Another important notion in Teichm\"uller theory  is the Teichm\"uller disks. These are  isometric embeddings of the hyperbolic disk $\mathbb{H}$ in the Teichm\"uller space. Such a disk can be viewed as a complex geodesic generated  by a quadratic differential $q$ on a Riemann surface $X$. This quadratic differential defines a flat metric structure on $X$ with conical singularities such that the holonomy of any closed curve on $X$ belongs to the subgroup $\{\pm\Id\}\times\R^2$ of $\mathrm{Isom}(\R^2)$.  If this quadratic differential is the square of a holomorphic one form $\omega$ on $X$, then the holonomy of any closed curve is a translation of $\R^2$, and we have a translation surface $(X,\omega)$. \medskip

Using the flat metric viewpoint, one can easily define  a natural action  of $\GL^+(2,\R)$ on the space of translation surfaces as follows: given a matrix $A \in \GL^+(2,\R)$ and an atlas $\{\phi_i, \, i\in I\}$ defining a translation surface structure, we get an atlas for another translation surface structure defined by $\{A\circ\phi_i, \ i \in I\}$. The Teichm\"uller disk generated by a holomorphic one-form $(X,\omega)$ is precisely the projection into the Teichm\"uller space of its $\GL^+(2,\R)$-orbit. Translations surfaces and their $\GL^+(2,\R)$-orbit also arise in different contexts such as dynamics of billiards in rational polygons, interval exchange transformations, pseudo-Anosov homeomorphisms...

The importance of the $\GL^+(2,\R)$-action is related to the fact that the $\GL^+(2,\R)$-orbit closure of a translation surface  encodes information on its geometric and dynamical properties. A remarkable illustration of this phenomenon is the famous Veech's Dichotomy, which states that if the stabilizer of $(X,\omega)$ for the action of $\GL^+(2,\R)$ is a lattice in $\SL(2,\R)$, then the linear flow in any direction on $X$ is either periodic or uniquely ergodic. Following a work of  Smillie, the stabilizer of  $(X,\omega)$, denoted by $\SL(X,\omega)$, is a lattice in $\SL(2,\R)$ if and only if the $\GL^+(2,\R)$-orbit of $(X,\omega)$ is a closed subset of the moduli space. For more details on translation surfaces and related problems we refer to the excellent surveys~\cite{MasTab02, Zor-survey}.\medskip

The group $\SL(X,\omega)$ is closely related to the subgroup of the mapping class group that stabilizes the Teichm\"uller disk generated by $(X,\omega)$. This subgroup consists of elements of $\MCG(S)$ that are realized by homeomorphisms of $X$ preserving the set of singularities (for the flat metric), and given by affine maps in local charts of the flat metric structure. This subgroup is denoted by $\Aff^+(X,\omega)$. There is a natural homomorphism from $\Aff^+(X,\omega)$ to $\SL(X,\omega)$ which associates to each element of $\Aff^+(X,\omega)$ its derivative. It is not difficult to see that this homomorphism is surjective and has finite kernel.   The study of $\Aff^+(X,\omega)$ and $\SL(X,\omega)$ is a recurrent theme in the theory of dynamics in Teichm\"uller space (see {\em e.g.} \cite{McM_flux, HubSch04, HubLan06, Mol09, Leh14}).

\subsection{Flat metric and curve complex}
Consider now the flat metric defined by a holormorphic one-form $\omega$ on a (compact) Riemann surface $X$. By compactness, there exists a curve of minimal length in the free homotopy class of any essential simple closed curve. In general this curve of minimal length may not be a geodesic as it may contain some singularity in its interior. Nevertheless, following a result by Masur~\cite{Mas86}, we know that there are infinitely many curves that can be realized as simple closed geodesics for $\omega$. Thus $(X,\omega)$ specifies a subset of vertices of $\Cc(S)$. Note that unlike the situation of hyperbolic surface, closed geodesics of minimal length are not unique in their homotopy class. They actually arise in family, that is simple closed geodesics in the same homotopy class fill out a subset of $X$ which is isometric to $(\R/{c\Z})\times(0,h)$. We will call such a subset a {\em geometric cylinder}, and the corresponding simple closed geodesics its {\em core curves}.

Mimicking the construction of the curve graph, we can add an edge between two vertices representing two cylinders if there exist two curves, one in each  homotopy class, that can be realized  disjointly (this condition is equivalent to requiring that the corresponding geodesics for the flat metric are disjoint). Thus, for each translation surface, we have a subgraph $\Cc_{\rm cyl}$ of the curve graph.\medskip

Let $A$ be a matrix in $\GL^+(2,\R)$, and consider the surface $(X',\omega'):=A\cdot(X,\omega)$. Since the action of $A$  preserves the affine structure, a geodesic on $X$ corresponds to a geodesic on $X'$ and vice-versa. Therefore, the subgraphs associated to $(X',\omega')$ and to $(X,\omega)$ are the same. This subgraph is actually associated to the Teichm\"uller disk generated by $(X,\omega)$. As $\Cc(S)$ can be viewed as the combinatorial model for the Teichm\"uller space, $\Cc_{\rm cyl}$ can be viewed as the counterpart of a Teichm\"uller disk in this setting. By definition, elements of $\Aff^+(X,\omega)$ preserve $\Cc_{\rm cyl}$ and act on $\Cc_{\rm cyl}$ by isomorphisms. As properties of the Mapping Class Group can be studied via its action on the curve complex, one can expect the knowledge about the combinatorial and geometric structure of $\Cc_{\rm cyl}$ to be useful for the study of $\Aff^+(X,\omega)$.

\subsection{Statement of results}
The main purpose of this paper is to investigate $\Cc_{\rm cyl}$ when $X$ is a surface of genus two.  The reason for this restriction is the technical difficulties for the general cases. Hopefully, the results and techniques used in this situation may inspire further results in higher genera.\medskip

Recall that the moduli space of translation surfaces is naturally stratified by the zero orders of the one-form $\omega$ (or equivalently, the cone angles at the singularities). In genus two, we have two strata: $\H(2)$ which contains pairs $(X,\omega)$ such that $\omega$ has a unique double zero, and $\H(1,1)$ which contains pairs $(X,\omega)$ such that $\omega$ has two simple zeros.
Our first result shows that the geometry of $\Cc_{\rm cyl}$ does depend on the stratum of $(X,\omega)$. \medskip

%Recall that every Riemann surface of genus two is hyperelliptic. Let $\hinv$ denote the hyperelliptic involution of $X$.  If $\omega\in \H(2)$ then the unique zero of $\omega$ is fixed by $\hinv$ hence it % must be  Weierstrass point, if $\omega \in \H(1,1)$ then the two zeros of $\omega$ are permuted by $\hinv$. \medskip

\begin{MainTheorem}[Theorem~\ref{thm:exist:mult:curv}]\label{MThm:A}
If $(X,\omega) \in \H(2)$ then $\Cc_{\rm cyl}$ contains no triangles, but if $(X,\omega) \in \H(1,1)$ then $\Cc_{\rm cyl}$ always contains some triangles.
\end{MainTheorem}

Note that a triangle in $\Cc_{\rm cyl}$ is a triple of simple closed curves pairwise disjoint that are simultaneously realized as core curves of three cylinders in $(X,\omega)$. \medskip

From its  definition, the geometric structure of the subgraph $\Cc_{\rm cyl}$ depends very much on the flat metric of $(X,\omega)$. It is not difficult to see that $\Cc_{\rm cyl}$ is not  connected in general (see Section~\ref{sec:cyl:graph:def}). To get a ``nicer'' subgraph of $\Cc(S)$, we enlarge $\Cc_{\rm cyl}$ by adjoin to it the vertices of $\Cc(S)$  representing {\em degenerate cylinders}. Roughly speaking, a degenerate cylinder on $X$ is a union of two saddle connections in the same direction such that there are deformations of $(X,\omega)$, on which this union is freely homotopic to the core curves of a geometric cylinder. We refer to Section~\ref{sec:cyl:graph:def} for a more precise definition. In particular, any degenerate cylinder is freely homotopic a simple closed curve. Thus it corresponds to a vertex of $\Cc(S)$.

We define $\hat{\Cc}^{(0)}_{\rm cyl}$ to be the set of vertices of $\Cc(S)$ that correspond to geometric cylinders and degenerate cylinders in $(X,\omega)$.  We then define $\hat{\Cc}^{(1)}_{\rm cyl}$ to be the set of the edges of $\Cc(S)$ whose both endpoints belong to $\hat{\Cc}^{(0)}_{\rm cyl}$.  We thus get a subgraph $\hat{\Cc}_{\rm cyl}$ of $\Cc(S)$. By a slight abuse of notation, we will also call $\hat{\Cc}_{\rm cyl}$ the {\em cylinder graph} of $(X,\omega)$. Subsequently, this subgraph is the main object of our investigation. We resume the results  concerning $\hat{\Cc}_{\rm cyl}$ in the following \medskip

\begin{MainTheorem}~\label{MThm:B}
 For any $(X,\omega) \in \H(1,1)\sqcup\H(2)$, the subgraph $\hat{\Cc}_{\rm cyl}$ is connected and has infinite diameter. The subgroup of $\MCG(S)$ that stabilizes $\hat{\Cc}_{\rm cyl}$ is precisely $\Aff^+(X,\omega)$.  Moreover, if $(X,\omega)$ is completely periodic in the sense of Calta then $\hat{\Cc}_{\rm cyl}$ is Gromov-hyperbolic.
\end{MainTheorem}

Theorem~\ref{MThm:B} actually comprises several statements which are proved in Corollary~\ref{cor:Cess:connect}, Proposition~\ref{prop:infty:diam}, Proposition~\ref{prop:aut:preserve:cyl}, and Theorem~\ref{thm:cper:hyperbolic}. The contexts and precise statements will be given in the corresponding sections. %\medskip

We finally consider the quotient of $\hat{\Cc}_{\rm cyl}$ by the action of $\Aff^+(X,\omega)$ in the case $(X,\omega)$ is a Veech surface (that is $\SL(X,\omega)$ is a lattice of $\SL(2,\R)$).\medskip

\begin{MainTheorem}\label{MThm:C}
 Let $\Gp$ be the quotient of $\hat{\Cc}_{\rm cyl}$ by the group of affine automorphisms. Then  $(X,\omega) \in \H(2)\sqcup\H(1,1)$ is a Veech surface if and only if $\Gp$ has finitely many vertices. For any Veech surface in $\H(2)$ the set of edges of $\Gp$ is also finite. There exist Veech surfaces in $\H(1,1)$ such that $\Gp$ has infinitely many edges.
\end{MainTheorem}

The statements of Theorem~\ref{MThm:C} are proved in Theorem~\ref{thm:quotient:fin:ver} and Proposition~\ref{prop:Veech:q:infin:H11}.

%\subsection{Perspectives}

%\bigskip

\subsection{Outline} The paper is organized as follows:

In Section~\ref{sec:preliminary} we recall standard notions concerning translation surfaces. We show some geometric and topological features of translation surfaces of genus two. We end this section by the proof of Theorem~\ref{MThm:A}.

%In Section~\ref{sec:triangulation}, we give some constructions of surfaces in $\H(2)\sqcup\H(1,1)$ from triangles and polygons in $\R^2$. Those constructions are used in Section~\ref{sec:reduce:inters:nb}.

In Section~\ref{sec:cyl:graph:def}, we introduce the notion of degenerate cylinders and define the cylinder graphs $\Cc_{\rm cyl}$ and $\hat{\Cc}_{\rm cyl}$. We show that $\hat{\Cc}_{\rm cyl}$ is connected  and has infinite diameter in Section~\ref{sec:reduce:inters:nb} and Section~\ref{sec:inft:diam}. Those results follow from Theorem~\ref{thm:dist:n:inter} which gives a bound on the distance in $\hat{\Cc}_{\rm cyl}$ using the intersection number.

Section~\ref{sec:automorphisms} is devoted to the proof of the fact that the stabilizer subgroup of $\hat{\Cc}_{\rm cyl}$ in $\MCG(S)$ is precisely the group  of affine automorphisms.

In Section~\ref{sec:hyperbolic} we show that if $(X,\omega)$ is completely periodic in the sense of Calta then $\hat{\Cc}_{\rm cyl}$ is Gromov-hyperbolic. Our proof follows a strategy  of Bowditch and uses a hyperbolicity criterion by Masur-Schleimer.

We give the proof of Theorem~\ref{MThm:C} in Section~\ref{sec:quotient}. Finally, in Section~\ref{sec:prototype}, we give the connection between the quotient graph $\Gp=\hat{\Cc}_{\rm cyl}/\Aff^+$ and the set of prototypes for Veech surfaces in $\H(2)$, which were introduced by McMullen~\cite{McM_spin}.

\subsection{Acknowledgements} The author warmly thanks Arnaud Hilion for the very helpful and stimulating discussions.

\bigskip

\section{Preliminaries}\label{sec:preliminary}
In this section we will prove some topological properties of saddle connections and cylinders on translation surfaces in genus two. The main result of this section  is Theorem~\ref{thm:exist:mult:curv}. %\medskip

Let $(X,\omega)$ be a translation surface. A saddle connection on $X$ is a geodesic segments whose endpoints are singularities which contains no singularities in its interior. A (geometric) cylinder of $X$ is a subset $C$ isometric to $(\R/c\Z)\times(0,h)$, with $c,h \in \R_{>0}$, which is not properly contained in another subset with the same property. The parameter $c$ is called the {\em circumference} and $h$  the {\em width} or {\em height} of this cylinder.

The isometry from $(\R/c\Z)\times(0,h)$ to $C$ can be extended by continuity to a map from $(\R/c\Z)\times[0,h]$ to $X$.  We will call the images of $(\R/c\Z)\times\{0\}$ and $(\R/c\Z)\times\{h\}$ the boundary components of $C$. Each boundary component is a concatenation of some saddle connections. It may happen that the two boundary components coincide as subsets of $X$. We say that $C$ is {\em simple cylinder} if each of its boundary component is a single saddle connection. It is worth noticing that on a translation surface of genus two, every cylinder is invariant by  the hyperelliptic involution. Therefore the two boundary components of any cylinder contain the same number of saddle connections. \medskip

Throughout this paper, for any cycle $c \in H_1(X,\{\text{zeros of } \omega\};\Z)$, we will use the notation $\omega(c):=\int_c\omega$,  and for any saddle connection $s$, its euclidean length will be denoted by $|s|$. Let us start by the following elementary lemma.

\begin{Lemma}\label{lm:embd:par}
 Let $(X,\omega)$ be a translation surface in one of the hyperelliptic components $\H^{\rm hyp}(2g-2)$ or $\H^{\rm hyp}(g-1,g-1)$, and $s$ be a saddle connection invariant by the hyperelliptic involution $\hinv$ of $X$. We assume that $s$ is not vertical. Then there exist a parallelogram $\P=(P_1P_2P_3P_4)$ in $\R^2$, and a locally isometric mapping $\varphi: \P \ra X$ such that

 \begin{itemize}
  \item[a)]   The vertical lines through the vertices $P_3$ and $P_4$  intersect the diagonal $\ol{P_1P_2}$.

  \item[b)]   The vertices of $\P$ are mapped to the singularities of $X$, and $\ol{P_1P_2}$ is mapped isometrically to $s$.

  \item[c)]   The restriction of $\varphi$ into $\inter(\P)$ is an embedding.

  \item[d)]   Let $\eta>0$ be the length of the vertical segment from $P_3$ or $P_4$  to a point in $\ol{P_1P_2}$. Then for any vertical segment $u$ in $X$ from a singular point to a point in $s$, we have $|u| \geq \eta$, where $|u|$ is the euclidian length of $u$.
\end{itemize}
We will call $\P$ the {\em embedded parallelogram} associated to $s$.
\end{Lemma}

\begin{Remark}\hfill
 \begin{itemize}
  \item[$\bullet$] Since $s$ in invariant by $\hinv$, we must have $\hinv(\varphi(\P))=\varphi(\P)$.

  \item[$\bullet$] The sides of $\P$ are mapped to saddle connections on $X$. Even though the restriction of $\varphi$ into $\inter(\P)$ is  one-to-one, it may happen that $\varphi$ maps the opposite sides of $\P$ to the same saddle connection.

  \item[$\bullet$] This lemma is also valid for  translation surfaces in $\H(0)$ and $\H(0,0)$.
  \end{itemize}
\end{Remark}

\begin{proof}[Proof of Lemma~\ref{lm:embd:par}]
We will only give the proof for the case $(X,\omega)\in \H^{\rm hyp}(2g-2)$ as the proof for $\H^{\rm hyp}(g-1,g-1)$ is the same. Using $U_-=\{\left(\begin{smallmatrix} 1 & 0 \\ t & 1 \end{smallmatrix}\right), \, t \in \R\}$, we can assume that $s$ is horizontal. Let $\Psi_t$ be the vertical flow on $X$ generated by the vertical vector field $(0,1)$, this flow moves regular points of $X$ vertically, upward if $t>0$.

Consider the vertical geodesic rays emanating from the unique zero $P_0$ of $\omega$ in direction $(0,-1)$. We claim that one of the rays in this  direction  must meet $s$. Indeed, if this is not the case then for any $t \in \R_{>0}$, $\Psi_t(s)$ does not contain $P_0$, and it follows that one can embed a rectangle of infinite area into $X$. Let $u^+$ be a vertical geodesic segment of minimal length from $P_0$ to a point in $s$ which is included in a ray in direction $(0,-1)$. Since $s$ is invariant by $\hinv$,  $u^-:=\hinv(u^+)$ is a  vertical segment of minimal length from $P_0$ to a point in $s$ which is included in a ray in direction $(0,1)$. Using the developing map, we can realize $s$ as a horizontal segment $\ol{P_1P_2} \subset \R^2$, $u^+$ (resp. $u^-$) as a vertical segment $\ol{P_3{P'}_3}$ (resp. $\ol{P_4{P'}_4}$) where ${P'}_3, {P'}_4 \in \ol{P_1P_2}$. Remark that the central symmetry fixing the midpoint of $\ol{P_1P_2}$ exchanges $\ol{P_3{P'}_3}$ and $\ol{P_4{P'}_4}$.

Let $\P$ denote the parallelogram $(P_1P_3P_2P_4)$. We define a map $\varphi: \P \ra X$ as follows: for any point $M\in \P$, let $M'$ be the orthogonal projection of $M$ in  $\ol{P_1P_2}$, and $t$ be the length of $\ol{MM'}$. Let $\hat{M}'$ be the point in $s$ corresponding to $M'$ by the identification between $\ol{P_1P_2}$ and $s$. We  then  define $\varphi(M):=\Psi_t(\hat{M}')$ if $M$ is above $\ol{P_1P_2}$, and $\varphi(M)=\Psi_{-t}(\hat{M}')$ if $M$ is below $\ol{P_1P_2}$.  By definition, $\varphi$ is a local isometry and maps the vertices of $\P$ to $P_0$.

Note that we have $|\ol{MM'}| \leq |\ol{P_3{P'}_3}|=|\ol{P_4{P'}_4}|$ and the equality only occurs when $M=P_3$, or $M=P_4$. Thus, for all $M \in \P \setminus\{P_1,P_2,P_3,P_4\}, \; \varphi(M)$ is a regular point in $X$ (otherwise we would have a vertical segment from $P_0$ to a point in $s$ of length smaller than $|u^+|$).

We now claim that $\varphi|_{\inter(\P)}$ is an embedding. Assume that there exist two points $M_1,M_2\in\inter(\P)$ such that $\varphi(M_1)=\varphi(M_2)$. Set $\overrightarrow{v}:=\overrightarrow{M_1M_2}$, then for any $M,M' \in \P$ such that $\overrightarrow{MM'}=\overrightarrow{v}$, we have $\varphi(M)=\varphi(M')$. Since $\P$ is a parallelogram, there exists a vertex $P_i \in \{P_1,P_2,P_3,P_4\}$ and a point $M'\in \P\setminus\{P_1,P_2,P_3,P_4\}$ such that $\overrightarrow{P_iM'}=\overrightarrow{v}$, which implies that $\varphi(M')=P_0$, and we have a contradiction to the observation above.

It is now straightforward to verify  that $\P$ and $\varphi$ satisfy all the required properties.
\end{proof}

In what follows, by a {\em slit torus} we will mean a triple $(X,\omega,s)$ where $X$ is an elliptic curve, $\omega$ a non-zero holomorphic one-form, and $s$ an embedded geodesic segment (with respect to the flat metric defined by $\omega$) on $X$.  The following lemma is useful for us in the sequel.

% Let $\Lambda:=\{\int_c \omega, \, c \in H_1(X,\Z)\}$. Then $\Lambda$ is a lattice isomorphic to $\Z^2$ in $\C$. There exists $z \in \C$ such that the segment $I:=[0,z]$ satisfy $I\cap \Lambda =\{0\}$, and we have  $(X,\omega,s) \simeq (\C/\Lambda, dz, p(I))$, where $p: \C \ra \C/\Lambda$ is the canonical projection.

\begin{Lemma}\label{lm:slit:tor:decomp}
Let $(X,\omega,s)$ be a slit torus and $P_1,P_2$ be the endpoints of $s$. Assume that the segment (slit) $s$ is not vertical, that is $|{\rm Re}\omega(s)| >0$.  Then there exist a pair of parallel simple closed geodesics $c_1,c_2$ with $c_i$ passing through $P_i$ such that $c_i\cap \inter(s) =\vide$, and $0\leq |\mathrm{Re}\omega(c_i)| \leq |s|$. In particular,  the geodesics $c_1,c_2$ cut $X$ into two cylinders, one of which contains $\inter(s)$. Moreover, any leaf of the vertical foliation intersecting $c_i$ must intersect $s$, and if every leaf of the vertical foliation meets $s$, then we have $|\mathrm{Re}\omega(c_i)|>0$.
\end{Lemma}
\begin{proof}
 Remark that  a flat torus with two marked points can be considered as hyperelliptic translation surface. Here, the hyperelliptic involution is the unique one that acts by $-\Id$ on $H_1(X,\Z)$ and exchanges the marked points. Therefore, this lemma is a particular case of Lemma~\ref{lm:embd:par}. 
\end{proof}

We now turn into translation surfaces in genus two.  Let $(X,\omega)$ be a translation surface in $\H(2)\sqcup\H(1,1)$.  We denote by $\hinv$ the hyperelliptic involution of $X$.

\begin{Lemma}\label{lm:sc:no:inv:types}%\label{lm:sc:H2:2types}
 Let $s_1, s_2$ be a pair of saddle connections in $X$ which are permuted by $\hinv$. If $(X,\omega) \in \H(2)$, then  $s_1$ and $s_2$ bound a simple cylinder. If $(X,\omega) \in \H(1,1)$ then we have two cases:
 
 \begin{itemize}
 \item[$\bullet$]  if $s_i$ joins a zero of $\omega$ to itself then $s_1$ and $s_2$ bound a simple cylinder,
 
 \item[$\bullet$] if $s_i$ joins two different zeros of $\omega$ then $s_1\cup s_2$ decomposes $X$ as a connected sum of two slit tori.
 \end{itemize}
\end{Lemma}
 
\begin{proof}
Since $\hinv$ acts by $-\Id$ on  $H_1(X,\Z)$, $s_1$ and $s_2$ must be homologous. This lemma follows from an inspection on the configurations of rays originating from the zero(s) of $\omega$ in the same direction. 
\end{proof}

\begin{Lemma}\label{lm:H2:inv:sc:2cyl}
Let $(X,\omega)$ be a surface in $\H(2)$ and $s$ be a saddle connection in $X$ invariant by the hyperelliptic involution $\hinv$. Then there exist two disjoint cylinders $C_1,C_2$ that do not intersect $s$ (that is, $C_1\cap C_2=\vide$, and the core curves of $C_1$ and $C_2$ do not meet $s$). Remark that $s$ may be contained in the boundary of $C_1$ or $C_2$. The possible configurations of $C_1$ and $C_2$ with respect to $s$ are shown in Figure~\ref{fig:s:inv:disj:cyls}.
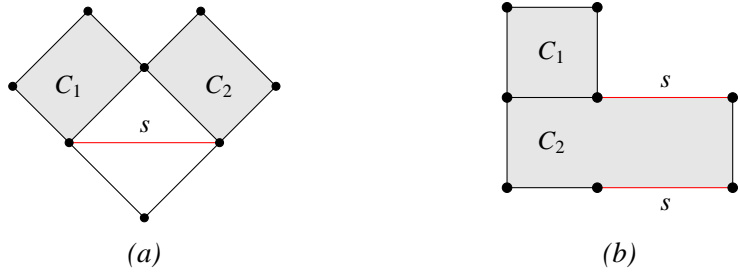
\begin{figure}[htb]
\begin{minipage}[t]{0.4\linewidth}
\centering
 \begin{tikzpicture}[scale=0.5]
 \fill[blue!50!yellow!20] (-3.5,1.5) -- (-2,0) -- (0,2) -- (-1.5,3.5) -- cycle;
 \fill[blue!50!yellow!20] (3.5,1.5) -- (2,0) -- (0,2) -- (1.5,3.5) -- cycle;
  \draw (0,2) -- (1.5,3.5) -- (3.5,1.5) -- (0,-2) -- (-3.5,1.5) -- (-1.5,3.5) -- cycle;
  \draw (-2,0) -- (0,2) -- (2,0);
  \draw[red] (-2,0) -- (2,0);
  \draw (0,0) node[above] {\small $s$} (-2,1.5) node {\small $C_1$} (2,1.5) node {\small $C_2$};
  \foreach \x in {(-3.5,1.5), (-2,0), (-1.5,3.5), (0,2), (0,-2),(1.5,3.5),(2,0),(3.5,1.5)}\filldraw[fill=black] \x circle (3pt);
  \draw (0,-3) node {(a)};
 \end{tikzpicture}
\end{minipage}
\begin{minipage}[t]{0.4\linewidth}
\centering
\begin{tikzpicture}[scale=0.6]
 \fill[blue!50!yellow!20] (0,0) -- (2,0) -- (2,2) -- (0,2) -- cycle;
 \fill[blue!50!yellow!20] (0,-2) -- (5,-2) -- (5,0) -- (0,0) --cycle;
 \draw (0,-2) -- (0,2) (2,0) -- (2,2) (5,-2) -- (5,0);
 \draw (0,2) -- (2,2) (0,0) -- (2,0) (0,-2) -- (2,-2);
 \draw[red] (2,0) -- (5,0) (2,-2) --  (5,-2);

 \foreach \x in {(0,2), (0,0), (0,-2), (2,2), (2,0), (2,-2), (5,0), (5,-2)} \filldraw[fill=black] \x circle (3pt);

 \draw (1,1) node {\small $C_1$} (1,-1) node {\small $C_2$} (3.5,0) node[above] {\small $s$} (3.5,-2) node[below] {\small $s$};
 \draw (2.5,-3.5) node {(b)};
\end{tikzpicture}
\end{minipage}
 \caption{Configurations of $C_1,C_2$ with respect to $s$: (a) none of $C_1,C_2$ contains $s$ in its boundary, (b) $s$ is contained in the boundary of $C_2$.}
 \label{fig:s:inv:disj:cyls}
\end{figure}
\end{Lemma}

\begin{proof}
Without loss of generality, we can assume that $s$ is horizontal. Let $\P=(P_1P_3P_2P_4)$ be the embedded parallelogram associated to $s$, and $\varphi: \P \ra X$ be the embedding map such that $s=\varphi(\ol{P_1P_2})$ (see Lemma~\ref{lm:embd:par}). We choose the labeling of the vertices of $\P$ such that $P_3$ is the highest vertex, and $P_4$ is the lowest one. Throughout the proof, we will refer to Figure~\ref{fig:inv:hex}.

Let $d_1^+=\varphi(\ol{P_3P_1}), d_2^+=\varphi(\ol{P_3P_2}), d^-_1=\varphi(\ol{P_4P_2}), d^-_2=\varphi(\ol{P_4P_1})$. We have $d_i^-=\hinv(d_i^+)$. By Lemma~\ref{lm:sc:no:inv:types}, either $d_i^+=d_i^-$ as subsets of $X$, or $d_i^\pm$ bound a simple cylinder. Remark that $d_1^+$ and $d_2^+$ cannot be both invariant by $\hinv$, otherwise we would have $X=\varphi(\P)$, and $X$ must be a torus. Thus we only have to consider two cases:
\begin{itemize}
\item[i)] Both $d_1^\pm$ and $d_2^\pm$ are respectively boundaries of two simple cylinders $C_1,C_2$ in $X$. In this case, it is not difficult to see that both $C_1$ and $C_2$ are disjoint from $\varphi(\P)$, and $C_1\cap C_2 =\vide$. We then get the configuration (a).

\item[ii)] One of  $d_1^+,d_2^+$ is invariant by $\hinv$, the other  bounds a simple cylinder. In this case, $\varphi(\P)$ is actually a simple cylinder. Without loss of generality, we can assume that $d_1^\pm$ bound the cylinder $C=\varphi(\P)$, and $\hinv(d_2^+)=d_2^-$.

Let $P_5$ be the point in $\R^2$ such that the triangle $(P_3P_5P_2)$ is the image of $(P_1P_2P_4)$ by the translation by $\overrightarrow{P_1P_3}$. Using the assumption that $\varphi(\ol{P_3P_2})=\varphi(\ol{P_1P_4})$, we see that $\varphi$ extends to a local isometric map from $\P'=(P_1P_2P_5P_3)$ to $X$ such that $\varphi(\P')=C$ and $\varphi_{|\inter(\P')}$ is an embedding (see Figure~\ref{fig:inv:hex}).

Consider the horizontal rays emanating from the unique zero $x_0$ of $\omega$ to the outside of $C$. By the same argument as in Lemma~\ref{lm:embd:par},  we see that one of the rays in direction $(1,0)$ reaches $d_1^+=\varphi(\ol{P_3P_1})$ from the outside of $C$. It follows that we can then extend $\varphi$ to a convex hexagon $\Hex:=(P_1P_2Q_2P_5P_3Q_1)$, which is the union of  $\P'$ and two triangles $(P_2Q_2P_5)$ and $(P_3Q_1P_1)$. Note that $(P_2Q_2P_5)$ and $(P_3Q_1P_1)$ are exchanged by the central symmetry fixing the midpoint of $\ol{P_2P_3}$, and all the vertices of $\Hex$ are mapped to $x_0$.

\begin{figure}[htb]
     \begin{tikzpicture}[scale=0.5]
     \fill[blue!50!yellow!20] (0,0) -- (10,2) -- (9,4) -- (-1,2) -- cycle;
      \draw (0,0) -- (7,0) -- (10,2) -- (9,4) -- (2,4) -- (-1,2) -- cycle;
      \draw (0,0) -- (10,2) (-1,2) -- (9,4);
      \draw (0,0) -- (5,-4) -- (7,0);
      \draw[dashed] (0,0) -- (2,4) -- (7,0) -- (9,4);
      \foreach \x in {(0,0), (7,0), (10,2), (9,4), (2,4), (-1,2), (5,-4)} \filldraw[fill=black] \x circle (2pt);

      \draw (0,0) node[below] {\small $P_1$} (7.2,0) node[below] {\small $P_2$} (10,2) node[right] {\small $Q_2$} (9,4) node[above] {\small $P_5$} (2,4) node[above] {\small $P_3$} (-1,2) node[left] {\small $Q_1$} (5,-4) node[below] {\small $P_4$};
      \draw (4,0) node[below] {\tiny $s$} (5,4) node[above] {\tiny $s$};
      \draw (-1,1) node {\tiny $d^+_4$} (10,3) node {\tiny $d^-_4$} (9,0.8) node {\tiny $d_3^-$} (0.5,3.5) node {\tiny $d^+_3$} (2.5,-2.5) node {\tiny $d_2^-$} (6.5,-2) node {\tiny $d_1^-$};
      \draw (5.5,2.5) node {$D$};
     \end{tikzpicture}
     \caption{Finding a cylinder disjoint from $s$.}
     \label{fig:inv:hex}
\end{figure}
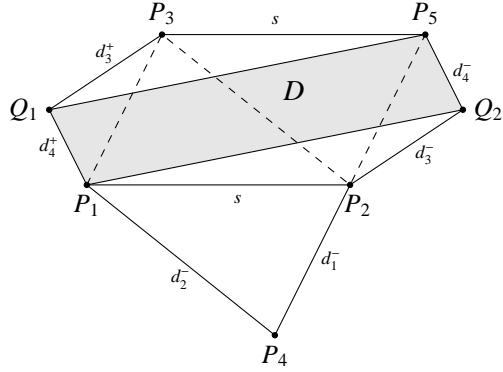

Let $d_3^+=\varphi(\ol{P_3Q_1}), d_4^+=\varphi(\ol{Q_1P_1}), d_3^-=\varphi(\ol{P_2Q_2}), d_4^-=\varphi(\ol{Q_2P_5})$. Again, for $i=3,4$, we have either $d_i^+=d_i^-$, or $d_i^\pm$ bound a simple cylinder.   If $d_i^+=d_i^-$ for both $i=3,4$, then $X=\varphi(\Hex)$ and $X$ must be a flat torus, so we have a contradiction. If both $d_i^\pm$ are the boundaries of simple cylinders, then these cylinders are disjoint, and also disjoint from $\varphi(\Hex)$. It follows that the total angle at $x_0$ is at least  $8\pi$ (the total angle of $\Hex$ plus $4\pi$), thus we have again a contradiction. We can then conclude that one of the pairs $d_3^\pm, d_4^\pm$ consist of a single saddle connection, and the other pair bound a simple cylinder. Without loss of generality, we can assume that $d_3^\pm$ bound a simple cylinder $C_3$, and $d_4^+=d_4^-=d_4$. Note that $C_3$ must be disjoint from $\varphi(\Hex)$, and in particular it is
disjoint from $s$.

Let $d^+=\varphi(\ol{Q_1P_5}), d^-=\varphi(\ol{P_1Q_2})$ then $d^\pm$ is the boundary of a cylinder $D$ whose core curves cross $d_4^\pm$. If $\Hex$ is  strictly convex then $D$ is a simple cylinder, but if $\ol{P_2Q_2}$ is parallel to $\ol{P_1P_2}$ then $D$ is not simple (in this case we actually have $\ol{D}=\varphi(\Hex)$). Nevertheless, in both cases the core curves of $D$ do not intersect $s$. Since $D$ is contained in $\varphi(\Hex)$, we have $C_3\cap D=\vide$. Since both $C_3$ and $D$ are disjoint from $s$, the lemma is proved.
\end{itemize}
\end{proof}

We are now ready to show

\begin{Theorem}\label{thm:exist:mult:curv}\hfill
\begin{itemize}
\item[a)] On any $(X,\omega)\in \H(2)$ there always exist two disjoint simple cylinders. There cannot exist  a triple of pairwise disjoint cylinders in $X$.
%In other words, $\Cc(X,\omega)$ always contains two a adjacent simple vertices.

\item[b)] On any $(X,\omega) \in \H(1,1)$ there always exists a triple of cylinders which are pairwise-disjoint. %In other words $\Cc(X,\omega)$ always contains a triangle.
\end{itemize}
\end{Theorem}

%\medskip

\begin{Remark} \hfill
\begin{itemize}
\item[$\bullet$] The cylinders in Theorem~\ref{thm:exist:mult:curv}  are not necessarily parallel.

\item[$\bullet$] There cannot exist more than 3 simple closed curves pairwise disjoint on $S$. The statement b) means that  given  any holomorphic one-form in $\H(1,1)$, there always exists a family of three disjoint (simple closed) curves that are realized simultaneously as simple closed geodesics for the flat metric induced by this one-form.

\item The statement a) of the theorem is a direct consequence of \cite[Prop. A.1]{Ng14}.
\end{itemize}
\end{Remark}

\medskip

\begin{proof}[Proof of Theorem~\ref{thm:exist:mult:curv}, case $\H(2)$]
Lemma~\ref{lm:H2:inv:sc:2cyl} almost proves the statement for $\H(2)$ except that it does not guarantee that both cylinders are simple.  We will give here a proof  by using \cite[Lem. 2.1]{Ng11}. Let $s$ be a saddle connection on $X$ that is invariant by the hyperelliptic involution $\hinv$ (one can find such a saddle connection by picking a geodesic segment of minimal length $\hat{s}$ joining a regular Weierstrass point of $X$ to the unique zero of $\omega$, then take $s=\hat{s}\cup\hinv(\hat{s})$). By \cite[Lem. 2.1]{Ng11}, there exists a simple cylinder $C_1$ that contains $s$. Cut off $C_1$ from $X$ then identify the two geodesic segments on the boundary of the resulting surface, we obtain a flat torus $(X',\omega')$ with a marked geodesic segment $s'$.

We consider $(X',\omega',s')$ as a slit torus. By Lemma~\ref{lm:slit:tor:decomp}, we know that there exists a cylinder $C'$ in $X'$ that contains  $s'$. The complement of $C'$ in $X'$ is another cylinder $C_2$ whose core curves do not meet $s'$. By construction $C_2$ is a simple cylinder in $X$ and disjoint from $C_1$, hence the first assertion follows. \medskip

For the second assertion, we observe that any triple of pairwise disjoint simple closed curves disconnect $X$ into two thrice-holed spheres. If all the curves in this triple are simple closed geodesics (core curves of cylinders), then we get two flat surfaces with geodesic boundary. Since $X$ has only one singularity, one of the surfaces has no singularities in its interior. But the Euler characteristic of a thrice-holed sphere is $-1$, thus we have a contradiction to the Gauss-Bonnet formula. We can then conclude that  $X$ can not contain three disjoint cylinders.
\end{proof}

\medskip

\begin{proof}[Proof of Theorem~\ref{thm:exist:mult:curv}, case $\H(1,1)$]
By \cite[Lem. 2.1]{Ng11}, we know that there exists a simple cylinder $C_0$ on $(X,\omega)$ that is invariant by $\hinv$. Cut off $C_0$ and glue the two boundary components of the resulting surface, we obtain a surface $(\hat{X},\hat{\omega}) \in \H(2)$ with a marked saddle connection $\hat{s}$. Note that $\hat{s}$ is invariant by the hyperelliptic involution of $\hat{X}$.  By Lemma~\ref{lm:H2:inv:sc:2cyl}, we know that there exist two cylinders $C_1$ and $C_2$ on $\hat{X}$ disjoint from $\hat{s}$ such that $C_1\cap C_2 =\vide$. It follows immediately that $C_1$ and $C_2$ are actually cylinders in $X$ and disjoint from $C_0$, from which we get the desired conclusion.
\end{proof}

\section{Degenerate cylinders and cylinder graph}\label{sec:cyl:graph:def}

\subsection{Cylinder and the curve graph.}
Each cylinder in a translation surface is filled by simple closed geodesics in the same free homotopy class. The following elementary lemma shows that two (freely) homotopic closed geodesics must belong to the same cylinder.

\begin{Lemma}\label{lm:free:hom:same:cyl}
 Let $c_1$ and $c_2$ be two simple closed geodesics in $(X,\omega)$ which are freely homotopic. Then $c_1$ and $c_2$ are contained in the same cylinder.
\end{Lemma}
\begin{proof}
 Since $c_1,c_2$ are freely homotopic, they are homologous, hence $\omega(c_1)=\omega(c_2)$. It follows that $c_1$ and $c_2$ are parallel, thus must be disjoint. The pair $c_1,c_2$ cut $X$ into two components, one of which must be an annulus denoted by $A$ (see~\cite[Prop. A.11]{Bus92}).  We have a flat metric on $A$ induced by the flat metric of $X$. Let $\theta_1,\dots,\theta_k$ be the cone angles at the singularities in $A$. Since the boundary of $A$ is geodesic, the Gauss-Bonnet formula gives
 $$
 \sum_{1\leq i \leq k} (2\pi -\theta_i)=2\pi\chi(A)=0.
 $$
 Since any singularity on a translation surface has cone angle at least $4\pi$, the equation above actually shows that $A$ contains no singularities. Thus $A$ is a flat annulus, which must be contained in a cylinder of $X$. Therefore, $c_1$ and $c_2$ are contained in the same cylinder.
\end{proof}

Let $S$ be a fixed topological compact closed surface of genus two. Let $\Cc(S)$ denote the curve graph of $S$. Let $\Omega\Tcal_2$ be the Abelian differential bundle over the Teichm\"uller space $\Tcal_2$. Elements of $\Omega\Tcal_2$ are equivalence classes of triples $(X,\omega,f)$, where $X$ is a Riemann surface of genus two, $\omega$ is a holomorphic one-form on $X$, and $f$ is a homeomorphism from $S$ to $X$; two triples $(X,\omega,f)$ and $(X',\omega',f')$ are identified if there exists an isomorphism $\varphi: X \ra X'$ such that $\varphi^*\omega'=\omega$ and ${f'}^{-1}\circ \varphi \circ f : S \ra S$ is isotopic to $\id_S$. The equivalence class of $(X,\omega,f)$ will be denoted by $[X,\omega,f]$.

Each element $[X,\omega,f]$ of $\Omega\Tcal_2$ defines naturally a subgraph $\Cc_{\rm cyl}(X,\omega,f)$ of $\Cc(S)$. The vertices of this subgraph are free homotopy classes of the core curves of all cylinders on the translation surface $(X,\omega)$. The set  $\Cc_{\rm cyl}^{(1)}(X,\omega,f)$ consists of the edges in $\Cc^{(1)}(S)$ whose both endpoints belong to $\Cc_{\rm cyl}^{(0)}(X,\omega,f)$.

\subsection{Degenerate Cylinders}
 If $C$ be a cylinder in $X$ that fills $X$ ({\em i.e.}  $\ol{C}=X$), then $C$ represents an isolated vertex in $\Cc_{\rm cyl}(X,\omega,f)$. This is because the core curve of any other cylinder in $X$ must cross $C$. So in general $\Cc_{\rm cyl}(X,\omega,f)$ is not a connected graph. To fix this issue we introduce the notion of {\em degenerate cylinders}. Roughly speaking, a degenerate cylinder in $X$ is a union of parallel saddle connections such that there exist deformations of $(X,\omega)$ where this union is freely homotopic to the core curves of a simple cylinder.

To be more precise, let $x_0$ be a singularity on a translation surface $(X,\omega)$. For any pair $(r_1,r_2)$ of geodesic rays emanating from $x_0$, we will denote the counterclockwise angle from $r_1$ to $r_2$ by $\vartheta(r_1,r_2)$.  If $s$ is an oriented  saddle connection from a singularity $x_1$ to a singularity $x_2$, then we denote by $s^+$ (resp. $s^-$) the intersection of $s$ with a neighborhood of $x_1$ (resp. a neighborhood of $x_2$). This definition also makes sense when $x_1=x_2$, in which case the orientation of $s$ is to start in $s^+$ and end in $s^-$. %Note that we consider $s^+$ as a ray emanating from $x_1$ and $s^-$ as a ray emanating from $x_2$.

\begin{Definition}[Degenerate cylinder]\label{def:degen:cyl}
 We will call the union of two saddle connections $s_1,s_2$ in $(X,\omega) \in \H(2)\sqcup\H(1,1)$ a {\em degenerate cylinder} if they  are both invariant by the hyperelliptic involution, and up to an appropriate choice for the orientations of $s_1$ and $s_2$, we have
 $$
 \vartheta(s^-_1,s^+_2)=\vartheta(s^+_1,s^-_2)=\pi.
 $$
\end{Definition}

In Figure~\ref{fig:degen:cyl:config}, we represent the configurations of a degenerate cylinder at the singularities.

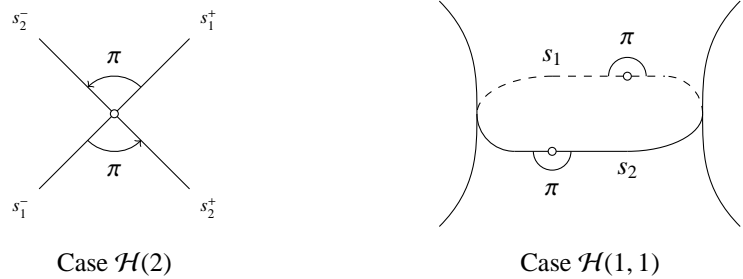
\begin{figure}[htb]
\begin{minipage}[t]{0.4\linewidth}
\centering
\begin{tikzpicture}[scale=0.5]
\draw (-2,2) -- (2,-2) (-2,-2) -- (2,2);

\filldraw[fill=white] (0,0) circle (3pt);

\draw[->, rotate=45] (1,0) arc (0:90:1);
\draw[->, rotate=45] (-1,0) arc (180:270:1);

\draw (0,1.5) node {\small $\pi$} (0,-1.5) node {\small $\pi$};

\draw (2.5,2.5) node {\tiny $s^+_1$} (-2.5,2.5) node {\tiny $s^-_2$} (-2.5,-2.5) node {\tiny $s^-_1$} (2.5,-2.5) node {\tiny $s^+_2$};

\draw (0,-4) node {\small Case $\H(2)$};
\end{tikzpicture}
\end{minipage}
\begin{minipage}[t]{0.4\linewidth}
\centering
\begin{tikzpicture}[scale=0.5]
\draw (-3,0) arc (180:270:1);
\draw (-2,-1) -- (1,-1);
\draw (1,-1) arc (270:360: 2 and 1);

\draw[dashed] (-1,1)  arc (90:180: 2 and 1);
\draw[dashed] (-1,1) -- (2,1);
\draw[dashed] (3,0) arc (0:90:1);

\draw (4,3) .. controls (3,2) and (3,1) .. (3,0);
\draw (4,-3) .. controls (3,-2) and (3,-1) .. (3,0);

\draw (-4,3) .. controls (-3,2) and (-3,1) .. (-3,0);
\draw (-4,-3) .. controls (-3,-2) and (-3,-1) .. (-3,0);

\draw (-1.5,-1) arc (180:360:0.5);
\draw (1.5,1) arc (0:180:0.5);

\filldraw[fill=white] (-1,-1) circle (3pt);
\filldraw[fill=white] (1,1) circle (3pt);

\draw (-1,-2) node {\small $\pi$} (1,2) node {\small $\pi$};
\draw (-1,1.5) node {\small $s_1$} (1,-1.5) node {\small $s_2$};
\draw (0,-4) node {\small Case $\H(1,1)$};
\end{tikzpicture}
\end{minipage}
\caption{Configuration of a degenerate cylinder at the singularities.}
\label{fig:degen:cyl:config}
\end{figure}

\begin{Remark}\hfill
\begin{itemize}
\item[$\bullet$] If $(X,\omega) \in \H(2)$, then a degenerate cylinder is not a simple curve, the zero of $\omega$ is its unique double point.

\item[$\bullet$] If $(X,\omega) \in \H(1,1)$ then the hyperelliptic involution $\hinv$ of $X$ permutes the zeros of $\omega$, thus a saddle connection  invariant by $\hinv$ must connect the two zeros of $\omega$. Therefore a degenerate cylinder must be a simple closed curve.
\end{itemize}
\end{Remark}

%\medskip

\noindent {\bf Examples:} Assume that $(X,\omega) \in \H(2)\sqcup\H(1,1)$ is horizontally periodic, and has a unique (geometric) horizontal cylinder $C$.  If $(X,\omega) \in \H(2)$ then it has $3$ horizontal saddle connections $s_1,s_2,s_3$, which are contained in the boundary of $C$ (see Figure~\ref{fig:degen:cyl:ex}). Note that all of them are invariant by the hyperelliptic involution. By definition $s_1\cup s_2, s_2\cup s_3, s_3\cup s_1$ are three degenerate cylinders.  Similarly, if $(X,\omega) \in \H(1,1)$, then we have $4$ horizontal saddle connections denoted by $s_1,\dots,s_4$ (see Figure~\ref{fig:degen:cyl:ex}) such that  $s_i\cup s_{i+1}$ is a degenerate cylinder, for $i=1,\dots,4$, with the convention $s_5=s_1$.

\begin{figure}[htb]
%\begin{minipage}[t]{0.4\linewidth}
\begin{tikzpicture}[scale=0.5]
\draw (0,0) -- (6,0) -- (6,2) -- (0,2) -- cycle;
\draw (9,0) -- (17,0) -- (17,2) -- (9,2) -- cycle;

\foreach \x in {(0,0), (2,0), (4,0),(6,0), (6,2), (4,2), (2,2), (0,2), (9,0), (11,0), (13,0), (15,0), (17,0), (17,2), (15,2), (13,2), (11,2), (9,2)} \filldraw[fill=black] \x circle (2pt);

\foreach \x in {(1,2.3), (5,-0.3), (10,2.3), (16,-0.3)} \draw \x node {\tiny  $s_1$};

\foreach \x in {(3,2.3), (3,-0.3), (12,2.3), (14,-0.3)} \draw \x node {\tiny $s_2$};

\foreach \x in {(5,2.3), (1,-0.3), (14,2.3), (12,-0.3)} \draw \x node {\tiny $s_3$};

\foreach \x in {(16,2.3), (10,-0.3)} \draw \x node {\tiny $s_4$};

\draw (3,1) node {\small $C$} (13,1) node {\small $C$};

\draw (3,-1.5) node {\small $\omega\in \H(2)$} (13,-1.5) node {\small $\omega  \in \H(1,1)$};

\end{tikzpicture}
%\end{minipage}
\caption{Degenerate cylinders on a horizontally periodic surface with a unique geometric horizontal cylinder.}
\label{fig:degen:cyl:ex}
\end{figure}
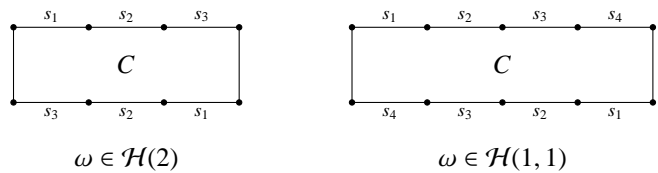

We will now prove some key properties of degenerate cylinders.

\begin{Lemma}\label{lm:degen:cyl:deform}%\label{cor:degen:cyl:H2:H11}
 Let $s:=s_1\cup s_2$ be a horizontal degenerate cylinder in $(X,\omega)  \in \H(2)\sqcup\H(1,1)$. Then there exists in a neighborhood of $(X,\omega)$ a continuous  family of translation surfaces $\{(X_t,\omega_t), \, t \in [0,\eps)\}$ in the same stratum as $(X,\omega)$,  with $\eps \in \R_{>0}$, such that
 \begin{itemize}
  \item $(X_0,\omega_0)=(X,\omega)$,

  \item for any $t\in (0,\eps)$, $(X_t,\omega_t)$ contains two saddle connections $s_{1,t}, s_{2,t}$ corresponding to $s_1, s_2$ and satisfy the following property: $s_{1,t}\cup s_{2,t}$ is freely homotopic to the core curves of a simple cylinder $C_t$ in $X_t$,

  \item as $t \ra 0$, the width of $C_t$ decreases to zero.
 \end{itemize}

 \noindent Moreover, for all $t\in (0,\eps)$, any vertical saddle connection (resp. regular geodesic) in $(X,\omega)$ corresponds to a vertical saddle connection (resp. regular geodesic) in $(X_t,\omega_t)$.
\end{Lemma}
\begin{proof}
 Let us define a {\em half cylinder} to be the quotient  $ (\R\times[0,h])/\Gamma$, where $\Gamma\simeq \Z_2\ltimes\Z$ is generated by $t: (x,y) \mapsto (x+\ell,y)$ and $s: (x,y) \mapsto (-x,h-y)$. We will call $h$ and $\ell$ respectively the {\em width} and {\em circumference} of the half disc. We will refer to the projection of $(0,0)$ as the marked point on its boundary. Equivalently, a half cylinder is a closed disc equipped with a flat metric structure with geodesic boundary and two singularities of angle $\pi$ in the interior.  
 
 Recall that all Riemann surfaces of genus two are hyperelliptic. Let $p: X \ra \CP^1$ be the hyperelliptic double cover of $X$. There exists a meromorphic quadratic differential $\eta$ on $\CP^1$ with at most simple poles such that $\omega^2=p^*\eta$. Note that $\eta$ has one zero, and $k$ poles, where $k=5$ if $\omega\in \H(2)$, and $k=6$ if $\omega\in \H(1,1)$. Let $P_0$ denote the unique zero of $\eta$, and $P_1,\dots,P_k$ its simple poles.  Let $Y$ be the flat surface defined by $\eta$ on $\CP^1$. Observe that the cone angle of $Y$ at $P_0$ is $3\pi$ if $\omega\in \H(2)$, and $4\pi$ if $\omega\in\H(1,1)$. The cone angle at $P_i$ is $\pi$, for $1,\dots,k$. 
 
 Since $s_i$, $i=1,2$, is invariant by $\hinv$, its projection in $Y$ is a geodesic segment $s'_i$ joining $P_0$ to a pole of $\eta$. By the definition of degenerate cylinder, one of the angles at $P_0$ specified by $s'_1$ and $s'_2$ is $\pi$. Let $\hat{Y}$ be the flat surface obtained by slitting open $Y$ along $s'_1$ and $s'_2$. By construction, $\hat{Y}$ is a flat disc with $k-2$ singularities (of cone angle $\pi$) in its interior, and whose boundary is a geodesic loop $c$ based at $P_0$. Note that $P_0$ is also a singular point of $\hat{Y}$.
 
 Let $c$ denote the boundary of $\hat{Y}$, and $\ell$ be the length  of $c$. Fix an $\eps >0$. For any $t \in (0,\eps)$, let $\hat{C}_t$ be the half cylinder of circumference equal to $\ell$, and width equal to $t$.  We can glue $\hat{C}_t$ to $\hat{Y}$ such that the marked point in the boundary of $\hat{C}_t$ is identified with $P_0$. Let $Y'_t$ denote the resulting flat surface. Observe that $Y'_t$ corresponds to a meromorphic differential $\eta'_t$ on $\CP^1$ which has a unique zero at $P_0$ and the same number of simple poles as $\eta$. It follows that the orienting double cover of $(\CP^1,\eta'_t)$ is an Abelian differential $(X_t,\omega_t)$ in the same stratum as $(X,\omega)$. Remark also that the double cover of $\hat{C}_t$ is a simple cylinder of width equal to $t$. We define $(X_0,\omega_0)$ to be  $(X,\omega)$.  It is now straightforward to check that the family $\{(X_t,\omega_t), \; t \in [0,\eps)\}$ satisfies the properties in the statement of the lemma.
\end{proof}

As a by product of Lemma~\ref{lm:degen:cyl:deform}, we also have

\begin{Lemma}\label{lm:degen:cyl:H2:H11}
Let $s:=s_1\cup s_2$ be a degenerate horizontal cylinder in $(X,\omega) \in \H(2)\sqcup \H(1,1)$. 
\begin{itemize}
\item[(i)] If $(X,\omega) \in \H(2)$, then there exist a pair of homologous saddle connections $r^\pm$ that cut out a slit torus  containing  $s$ satisfying the following condition: any vertical leaf crossing $r^\pm$  must intersect $s$.

\item[(ii)] If   $(X,\omega) \in \H(1,1)$, then either
        \begin{itemize}
        \item[a)] there exist a pair of homologous saddle connections $r^\pm$ that cut out a  a slit torus  containing  $s$ such that  any vertical leaf crossing $r^\pm$  must intersect $s$, or
        \item[b)] there are two simple cylinders  $C_1,C_2$ disjoint from $s$ such that any vertical leaf crossing $C_1$ or $C_2$ must intersect $s$.
        \end{itemize}
\end{itemize}
\end{Lemma}
\begin{proof}
 Let us use the same notations as in the proof of Lemma~\ref{lm:degen:cyl:deform}. Recall that by slitting open $Y$ along the projections of $s_1$ and $s_2$, we obtain a flat surface $\hat{Y}$, whose boundary is a geodesic loop $c$ based at $P_0$. One can construct a new flat surface homeomorphic to the sphere $\CP^1$ by ``sewing up'' $c$. This operation produces an extra singular point of angle $\pi$ at the midpoint of $c$.  
 
 Let $Y'$ denote the resulting surface. On $Y'$, we have $k-1$  singularities of cone angles $\pi$ and a singularity at $P_0$ of cone angle $2\pi$ if $\omega\in\H(2)$, or $3\pi$ if  $\omega\in \H(1,1)$. The loop $c$ corresponds to a segment $c'$ on $Y'$ joining $P_0$ to a singularity of angle $\pi$. Let $(X',\omega')$ be the orienting double cover of $Y'$, then either $(X',\omega')\in \H(0,0)$, or $(X',\omega')\in \H(2)$. In both cases, $c'$ gives rise to a saddle connection $s'$ invariant by the hyperelliptic involution of $X'$. Note that by construction, we can identify $X'\setminus s'$ with $X\setminus s$. 
 
 Let $\varphi: \P \ra X'$ be the embedded parallelogram associated to $s'$ introduced in Lemma~\ref{lm:embd:par}. By construction, $\varphi$ maps the sides of $\P$ to saddle connections on $X'$ which do not intersect $s'$ in their interior. Thus those saddle connections correspond to some saddle connections on $X$. It follows that $\varphi(\P)\subset X'$ corresponds to  a subsurface of $X$ containing $s$. The conclusions  of the lemma then follow from a careful inspection on the boundary of $\varphi(\P)$. 
\end{proof}

\subsection{Cylinder graph}
We now define a new subgraph $\hat{\Cc}_{\rm cyl}(X,\omega,f)$ of $\Cc(S)$ as follows: the vertices of $\hat{\Cc}_{\rm cyl}(X,\omega,f)$ are free homotopy classes of core curves of cylinders, or free homotopy classes of degenerate cylinders in $X$. Elements of $\hat{\Cc}^{(1)}_{\rm cyl}(X,\omega,f)$ are the edges of $\Cc(S)$ whose both endpoints are in $\hat{\Cc}^{(0)}_{\rm cyl}(X,\omega,f)$. %\medskip

Let $\dist^{\Cc}$ denote the distance in $\Cc(S)$. Recall that by definition each edge of $\Cc(S)$ has length equal to one. Let $a,b$ be two simple closed curves on $S$, and  $[a],[b]$ be respectively their free homotopy classes considered as vertices of $\Cc(S)$. We have
$$
\dist^{\Cc}([a],[b])=\min\{{\rm leng}(\gamma), \, \gamma \hbox{ path in $\Cc(S)$ from $[a]$ to $[b]$} \}.
$$

\noindent We define a distance $\dist$ in $\hat{\Cc}_{\rm cyl}(X,\omega,f)$ in the same manner, that is, every edge has length equal to one, and  given $[a],[b] \in \hat{\Cc}_{\rm cyl}(X,\omega,f)$,
$$
\dist([a],[b])=\min\{{\rm leng}(\gamma), \, \gamma \hbox{ path in $\hat{\Cc}_{\rm cyl}(X,\omega,f)$ from $[a]$ to $[b]$}\}.
$$

\noindent By convention, if there are no paths in $\hat{\Cc}_{\rm cyl}(X,\omega,f)$ from $[a]$ to $[b]$, then we define $\dist([a],[b])=\infty$.  The subgraph $\hat{\Cc}_{\rm cyl}(X,\omega,f)$ will be the main subject of our investigation in the remaining of this paper. To alleviate the notations, when $(X,\omega)$ and a marking mapping $f: S \ra X$ are fixed, we will write $\Cc_{\rm cyl}$ and $\hat{\Cc}_{\rm cyl}$ instead of $\Cc_{\rm cyl}(X,\omega, f)$ and $\hat{\Cc}_{\rm cyl}(X,\omega,f)$.

\medskip

\noindent \underline{\bf Convention:} In the sequel, a ``cylinder'' could mean a usual geometric cylinder or a degenerate one. We will refer to geometric usual cylinders as {\em non-degenerate} cylinders. The term  ``a core curve'' will have the usual meaning for non-degenerate cylinder, for a degenerate one it just means the cylinder itself.

\subsection{Intersection numbers}
Let $\iota(.,.)$ denote the geometric intersection form on the set of free homotopy classes of simple closed curves on $S$.  Let $a$, $b$ be two simple closed curves in $S$, and $[a]$, $[b]$ their free homotopy classes respectively. Recall that  $[a]$ and $[b]$ are connected by an edge in $\Cc(S)$ if and only if $\iota([a],[b])=0$. %\medskip

Assume now that $a$ and $b$ are simple closed geodesics in $(X,\omega)$.  If $a$ and $b$ are parallel, then they do not have intersection, hence $\iota([a],[b])=0$. If they are not parallel, then they intersect transversally at every intersection point.  By using the bigon criterion (see~\cite[Section 1.2.4]{FM-primer11}), it is not difficult to show that $\iota([a],[b]) =\#\{a\cap b\}$. However, if $a$ or $b$ is a degenerate cylinder then we must be a little more careful since in this case $a$ or $b$ may be not a simple curve ({\em i.e.} in $\H(2)$), and their intersections are not always transversal.

To deal with this complication, if $a$ and $b$ are core curves of two cylinders in $X$ (possibly degenerate), we will fix some  parametrizations $\alpha: \S^1 \ra X$ for $a$, and $\beta: \S^1 \ra X$ for $b$ such that $\alpha$ and $\beta$ are locally homeomorphic onto their images, and the restriction of $\alpha$ (resp. of $\beta$) to $\S^1\setminus\alpha^{-1}(\{\text{singularities of } X\})$ (resp. to $\S^1\setminus\beta^{-1}(\{\text{singularities of } X\})$) is one-to-one.

By an {\em intersection} of $a$ and $b$, we will mean a pair $(t,t')\in \S^1\times\S^1$ such that $\alpha(t)=\beta(t')$. This intersection is said to be {\em transversal} if there exist $\eps>0$ and $\eps'>0$ such that $a_1:=\alpha((t-\eps,t+\eps))$ and $b_1:=\beta((t'-\eps',t'+\eps'))$ are two simple arcs in $X$,  $a_1$ intersects $b_1$ transversally at $p=\alpha(t)=\beta(t')$, and $a_1$ and $b_1$ have no other intersections. We denote by $a\cap b$ the set of intersections of $a$ and $b$ , and by $a\hat{\cap} b$ the subset of transversal intersections.

%Therefore, if $C$ and $D$ are two geometric cylinders in $X$, one can define $

\begin{Lemma}\label{lm:2cyl:inters}
 Let $C$ and $D$ be two cylinders on $(X,\omega)$ (both possibly degenerate) that are not parallel. Let $c$ and $d$ be respectively a core curve of $C$ and a core curve of $D$. We denote by $[c]$ and $[d]$ the free homotopy classes of $c$ and $d$ respectively. Let $c\hat{\cap}d$ denote the set of transversal intersections of $c$ and $d$.  Then we have
 $$
\iota([c],[d])=\#\{c\hat{\cap} d\}.
 $$
Since a non-transversal intersection of $c$ and $d$ can only occur at a singularity, it follows  in particular that $\iota([c],[d])=\#\{c\cap d\}$ if  one of $c$ and $d$ is a regular geodesic.
\end{Lemma}
\begin{proof}
Let $\pi: \Delta=\{z \in \C, \quad |z| <1\} \rightarrow X$ denote the universal cover of $X$. The pull-back $\pi^*\omega$ of $\omega$ is a holomorphic one-form, which defines a flat metric with cone singularities on $\Delta$.

Fix a base point $x$ for $c$ and a base point  $y$ for $d$, which are not the singularities of $X$. Through any point in $\pi^{-1}(\{x\})$ (resp. any point in $\pi^{-1}(\{y\})$),  there is a unique lift of $c$ (resp. a unique lift $d$).  Since $c$ and $d$ are not necessarily simple curves, {\em a priori} each lift of $c$ and $d$ may not be a simple arc. But this actually does not happen.

\begin{Claim}\label{clm:inter:cyl:bigon}\hfill
\begin{itemize}
\item[(i)] Each lift of $c$ (resp. of $d$) is a simple arc in $\Delta$.
\item[(ii)] Two lifts of $c$ (resp. of $d$) can only meet  at at most one point (which is a non-transversal intersection).
\item[(iii)] A lift of $c$ and a lift of $d$ can only meet at at most one point.
\end{itemize}
\end{Claim}
\begin{proof}[Proof of the claim]
Since the argument for the three assertions are the same, we only give the proof of (iii). Let $\tilde{c}_0$ and $\tilde{d}_0$   be a lift of $c$ and a lift of $d$ in $\Delta$ respectively. Let us assume that $\tilde{c}_0$ and $\tilde{d}_0$ intersect at two points. There exists then a disc $B\subset \Delta$  bounded  by a subarc  $c_0 \subset \tilde{c}_0$ and  a subarc $d_0 \subset \tilde{d}_0$. Let $p,q$ be the common endpoints of $c_0$ and $d_0$, and $\alpha$ and $\beta$  be respectively the interior angles of $B$ at $p$ and $q$. Since $c_0$ and $d_0$ are geodesic segments for the flat metric on $\Delta$, we have $\alpha >0$ and $\beta>0$ ($\alpha=0$ or $\beta=0$ means that $c$ and $d$ are parallel).

Let $p_1,\dots,p_r$ be the  points in $\partial B$ that correspond to the zeros  of $\pi^*\omega$ and different from $p,q$. Let $\theta_i$ be the interior angle of $B$ at $p_i$.  By  definition of cylinders, we have $\theta_i\geq \pi$, for all $i=1,\dots,r$. Let $x_1, \dots, x_s$ be the zeros of $\pi^*\omega$ in $\inter(B)$, and $\hat{\theta}_i$ be the angles at $x_i$. The Gauss-Bonnet formula gives (see for instance~\cite[Prop. 1]{Troy91})
$$
\sum_{i=1}^s(2\pi-\hat{\theta}_i) + \sum_{i=1}^r(\pi-\theta_i) +2\pi-(\alpha+\beta) =2\pi\chi(B)=2\pi.
$$
Since $\alpha+\beta >0, \, \pi-\theta_i \leq 0$, and $2\pi-\hat{\theta}_i < 0$, we see that the equality above cannot be realized. Therefore, $B$ cannot exist, which means that $\tilde{c}_0$ and $\tilde{d}_0$ can only meet at at most one point.
\end{proof}

Since non-transversal intersections of $c$ and $d$ can only occur at the singularities of $X$ (zeros of $\omega$), we can deform $c$ and $d$  slightly in a neighborhood of each zero of $\omega$ to get simple closed curves $c'$ and $d'$ in the same free homotopy classes as $c$ and $d$ respectively such that $\#\{c \hat{\cap}d\}= \#\{c'\cap d'\}$.  Claim~\ref{clm:inter:cyl:bigon} then  implies that any lift of $c'$ in $\Delta$ intersects a lift of $d'$ at at most one point and all the intersections are transversal. It follows from the bigon criterion (see {\em e.g.} \cite[Prop. 1.7]{FM-primer11}) that
$$
\iota([c],[d])=\#\{c'\cap d'\}=\#\{c \hat{\cap}d\}.
$$
\noindent The lemma is then proved.
\end{proof}

\begin{Remark}\hfill
\begin{itemize}
\item[$\bullet$] If $C$ and $D$ are not parallel, we can assume that $C$ is horizontal and $D$ is vertical. In the case both $C$ and $D$ are degenerate, to compute their intersection number, one can use Lemma~\ref{lm:degen:cyl:deform} to get a deformation $(X_t,\omega_t)$  of $(X,\omega)$ in which $C$ corresponds to simple (horizontal) cylinder $C_t$. In $X_t$,  $D$ corresponds to a vertical cylinder $D_t$. Consequently, $c$ is freely homotopic to a regular horizontal geodesic $c_t$ in $X_t$, while $d$ is freely homotopic to a core curve $d_t$ of $D_t$.  It follows from Lemma~\ref{lm:2cyl:inters} that $\iota([c],[d])=\iota([c_t],[d_t])=\#\{c_t\cap d_t\}$.

\item[$\bullet$] It may happen that two degenerate cylinders in the same direction have a positive intersection number.
\end{itemize}
\end{Remark}

\section{Reducing  numbers of intersection}\label{sec:reduce:inters:nb}
In what follows, given two cylinders $C,D$ in $X$, by $\iota(C,D)$ we will mean the geometric intersection number $\iota([c],[d])$, where $c$ and $d$ are some core curves of $C$ and $D$ respectively. Our first goal is to estimate the distance in $\hat{\Cc}_{\rm cyl}$ by using intersection numbers. 

\begin{Theorem}\label{thm:dist:n:inter}
There exist two positive constants $K_1,K_2$ such that for any $[X,\omega,f] \in \Omega\Tcal_2$, and any cylinders $C$ and $D$ in $X$ (both possibly degenerate) considered as vertices of $\hat{\Cc}_{\rm cyl}(X,\omega,f)$, we have

\begin{equation}\label{eq:rel:dist:inter}
\dist(C,D) \leq K_1\iota(C,D)+K_2.
\end{equation}
\end{Theorem}

As a direct consequence of inequality \eqref{eq:rel:dist:inter}, we get

\begin{Corollary}\label{cor:Cess:connect}
The subgraph $\hat{\Cc}_{\rm cyl}(X,\omega,f)$ is connected.
\end{Corollary}

\subsection{Reducing to simple cylinders}
In what follows, we will fix a point $[X,\omega,f] \in \Omega\Tcal_2$, and by  cylinders in $X$ we include degenerate ones.  Our first step is to reduce the problem to the case $C$ and $D$ are simple cylinders. 

\begin{Lemma}\label{lm:nsimple:cylinder}
Let $C$ be horizontal cylinder that does not fill $X$, {\em i.e.} $\ol{C}\neq X$, and  $D$ be a vertical cylinder. Assume that $\iota(C,D)>0$. Then there exists a simple cylinder $C'$  such that $\dist(C,C')\leq 1$ and $\iota(C',D) \leq \iota(C,D)$.
\end{Lemma}
\begin{proof}
We first consider the case $C$ is non-degenerate. Let $c$ be a core curve of $C$ and $d$ a core curve of $D$. Since $c$ is a  regular simple closed geodesic, by Lemma~\ref{lm:2cyl:inters}, we have $ \iota(C,D)=\#\{c\cap d\}$.  Obviously, we only need to consider the case $C$ is not simple. 

If $(X,\omega)\in \H(2)$, then the complement of $\ol{C}$  is a simple cylinder $C'$ whose boundary is a pair homologous saddle connections contained in the boundary of $C$. In particular $C'$ is also horizontal, and we have $\iota(C,C')=0$, hence $\dist(C,C')=1$.  Any time $d$ crosses $C'$, it must cross $C$ before returning to $C'$. Therefore, we have $\iota(C',D) \leq \iota(C,D)$.

If $(X,\omega)\in \H(1,1)$ then  the complement of $\ol{C}$ is either: (a) horizontal simple cylinder, (b) two disjoint horizontal simple cylinders, or (c) a  torus  with a horizontal slit. In case (a) and case (b), the boundaries of the horizontal cylinders in the complement are contained in the boundary of $C$. Therefore, it suffices to choose one of them to be $C'$.  In case (c), let $(X',\omega',s')$ be the slit torus which is the complement of $\ol{C}$. Note that the slit  $s'$  corresponds to a pair of homologous saddle connections in  the boundary of $C$.   By Lemma~\ref{lm:slit:tor:decomp} we know that $X'$ contains a simple cylinder $C'$ disjoint from the slit $s'$ such that any vertical line crossing $C'$ must cross $s'$.  Since $C'$ is disjoint from $C$ we have $\dist(C,C')=1$.  Any time $d$ crosses $C'$, it must cross the slit $s'$ and hence $C$.  Therefore, we also have $\iota(C',D)\leq \iota(C,D)$. \medskip

% \begin{Lemma}\label{lm:degen:cyl:to:simp:cyl}
% Let $C$ be a degenerate horizontal cylinder in $X$ and $D$ be a vertical cylinder. Let $d$ be a core curve of $D$. Assume that $\iota(C,D)>0$. Then there exists a simple cylinder $C'$ such that $\dist(C,C')=1$ and $\iota(C',D) \leq \iota(C,D)$.
% \end{Lemma}

We now turn to the case $C$  is degenerate. If $(X,\omega) \in \H(2)$, from Lemma~\ref{lm:degen:cyl:H2:H11}, we know that $C$ is contained in a slit torus cut out by a pair of homologous  saddle connections $r^\pm$ such that every vertical leaf crossing $r^\pm$ intersects $C$.  Since $(X,\omega) \in \H(2)$, the complement of the slit torus is a simple cylinder $C'$ bounded by $r^\pm$. Clearly, we have $\dist(C,C')=1$.  If the core curves of $D$ are  regular geodesics (that is $D$ is non-degenerate), then we can  immediately conclude that $\iota(C',D) \leq \iota(C,D)$. In case $D$ is degenerate, we consider the deformations  $\{(X_t,\omega_t), \;  t \in [0,\eps)\}$ of $(X,\omega)$ given by Lemma~\ref{lm:degen:cyl:deform}. For $t\in (0,\eps)$, in $(X_t,\omega_t)$, $D$ becomes a simple cylinder $D_t$, while the cylinders $C$ and $C'$ persist and have the same properties. Since $\iota(C',D)=\iota(C',D_t)$ and $\iota(C,D)=\iota(C,D_t)$, we also get $\iota(C',D) \leq \iota(C,D)$.

The case $(X,\omega) \in \H(1,1)$ also follows from  similar arguments.
\end{proof}

%Lemma~\ref{lm:degen:cyl:H2:H11}, we have two cases: (a) $C$ is contained in a slit torus $(X',\omega',r')$ cut out by a pair of saddle connections $r^\pm$, and (b) $C$ is contained in the union of two embedded parallelograms, whose complement in $X$ is the union of two simple cylinders $C_1,C_2$.
%
%\begin{itemize}
%  \item[$\bullet$] In case (a), every vertical leaf crossing $r'$ intersects $C$, and the complement of $(X',\omega',r')$ is also a  slit torus $(X'',\omega'',r'')$. If $r^\pm$ is vertical then $r'$ and $r''$ are also vertical. Since $D$ intersects $C$, we derive that $D$ is contained in $X'$ hence disjoint from $X''$ as a vertical geodesic cannot cross $r'$. Therefore, it suffices to choose a simple cylinder in $X''$ disjoint from the slit $r''$.  If $r^\pm$ is not vertical, then using Lemma~\ref{lm:slit:tor:decomp}, we see that there exists in $X''$ a simple cylinder $C''$ disjoint from $r''$  such that every vertical leaf crossing $C''$ crosses also $r''$. Let $d$ be a core curve of $D$, we have $\iota(C'',D) \leq \#\{d\cap r''\} \leq \iota(C,D)$. Thus in this case we can choose $C'$ to be $C''$.
%
%  \item[$\bullet$] In case (b), every vertical leaf crossing $C_1$ or $C_2$ intersects $C$, hence we can choose $C'$ to be either $C_1$ or $C_2$.
%\end{itemize}

\begin{Lemma}\label{lm:fill:hor:cyl}
Assume that $C$ is a horizontal cylinder that fills $X$, and $D$ is a vertical cylinder. Then there exists a  simple cylinder $C'$ such that
$$
\left\{ \begin{array}{ccl}
         \dist(C',C) & = & 2, \\
         \iota(C',D) & \leq & \iota(C,D).
        \end{array}
\right.
$$
\end{Lemma}

\begin{proof}
Let $c$ be a core curve of $C$.  If $(X,\omega) \in \H(2)$ then the complement of $C$ is the union of three horizontal saddle connections $s_1,s_2,s_3$, all are invariant by the hyperelliptic involution.  Remark that the union of any two of these saddle connections is a degenerate cylinder. One can easily find a transverse simple cylinder $C'$ containing $s_1$, disjoint from the union $s_2\cup s_3$,  whose core curves cross $c$ once. Furthermore, we can choose $C'$ such that the horizontal component of its core curves has length smaller than the length of $c$.  Clearly,  we have $\dist(C,C')=2$. Since any vertical geodesic crossing  $C'$ crosses also  $C$, we have $\iota(C',D) \leq \iota(C,D)$. Thus the lemma is proved for this case.

The case $(X,\omega) \in \H(1,1)$ follows from the same arguments.
\end{proof}

In what follows, a geodesic line on $X$ that does not contain any singularity is called {\em regular}.

\begin{Lemma}\label{lm:vleaf:miss:C}
Let $C$ be a horizontal cylinder and $D$ be a vertical cylinder in $X$. If there exists a regular vertical leaf  which does not cross $C$ then $\dist(C,D) \leq 2$.
\end{Lemma}

\begin{proof}
Obviously we only need to consider the case $\iota(C,D)>0$. Assume that there is a regular vertical closed geodesic that does not intersect $C$, then there exists another vertical cylinder $D'$ which is disjoint from both $C$ and $D$. Consequently, we have $\dist(C,D)=2$.

Assume now that there is an infinite regular  vertical leaf that does not intersect $C$. The closure of this  leaf is a subsurface $X'$ of $X$ bounded by some vertical saddle connections. Let $s$ be a saddle connection in the boundary of $X'$. Note that $s$ and $\hinv(s)$ are homologous. Thus they decompose $X$ into two subsurfaces $X_1$ and $X_2$ both invariant by $\hinv$. Since $C$ is invariant by $\hinv$, it must be contained in one of the subsurfaces, say $X_1$. Since $s$ and $\hinv(s)$ are vertical, the core curves of $D$ cannot  cross $s$ and $\hinv(s)$, which means that $D$ is also contained in one subsurface. Since we have assumed that $\iota(C,D)>0$, $D$ must be contained in $X_1$.

The subsurface $X_2$ must be either a slit torus, or a surface in $\H(2)$ with a marked saddle connection. Actually, the latter case does not occur because it would imply that $X_1$ is a vertical simple cylinder containing both $C$ and $D$, which is impossible. Now, by Lemma~\ref{lm:slit:tor:decomp},  one can find in the torus $X_2$ a simple cylinder $C'$ that does not meet the slit. Since $C'$ corresponds to  a simple cylinder of $X$ which is disjoint from both $C$ and $D$, and we have $\dist(C,D)=2$. The lemma is then proved.
\end{proof}

From the Lemmas~\ref{lm:nsimple:cylinder}, \ref{lm:fill:hor:cyl}, we know that if $C$ is not simple then there exists a simple cylinder $C'$ such that $\dist(C,C')\leq 2$ and $\iota(C',D) \leq \iota(C,D)$. Consequently, we can find simple cylinders $C',D'$ such that
$$
\left\{
\begin{array}{lcl}
\dist(D,D') & \leq & 2, \\
\dist(C,C') & \leq & 2,\\
\iota(C',D') &\leq & \iota(C,D).
\end{array}
\right.
$$
It follows in particular that $\dist(C,D) \leq \dist(C',D') +4$. Therefore we only need to prove \eqref{eq:rel:dist:inter} for the case $C$ and $D$ are simple cylinders. Moreover, by Lemma~\ref{lm:vleaf:miss:C}, we can further assume that all the leaves  of the foliation in the direction of $D$ intersect $\ol{C}$. Thus, Theorem~\ref{thm:dist:n:inter} is a consequence of the following

\medskip
\begin{Proposition}~\label{prop:reduce:inters}
Let $C$ and $D$ be two simple cylinders such that all the leaves of the foliation in the direction of $D$ intersect $\ol{C}$.  Then there always exists a simple cylinder $C'$ such that
\begin{equation}\label{eq:s:cyl:induct}
\dist(C',C) \leq 3 \text{ and } \iota(C',D) < \iota(C,D).
\end{equation}
\end{Proposition}

To prove this proposition we will make use of the representation of translation surfaces as polygons in $\R^2$. In Section~\ref{sec:triangulation}, we give a uniform construction from symmetric polygons of translation surfaces in genus two satisfying the hypothesis of Proposition~\ref{prop:reduce:inters}.

\subsection{Proof of Proposition~\ref{prop:reduce:inters}, Case $\H(2)$}
\begin{proof}
By using $\GL^+(2,\R)$, we can assume that $C$ is a horizontal cylinder, and $D$ is vertical. From Proposition~\ref{prop:polygon:construct}(i), we can construct $(X,\omega)$ from a symmetric polygon  $\P:=(P_0\dots P_3 Q_0\dots Q_3)$ in $\R^2$. Note that by construction, the hyperelliptic involution of $X$ lifts to  the central symmetry fixing the midpoint of $\ol{P_0Q_0}$.

\begin{figure}[htb]
\begin{minipage}[t]{0.3\linewidth}

\centering
\begin{tikzpicture}[scale=0.35]
\fill[blue!60!yellow!10] (2,2) -- (7,6) -- (5,8) -- cycle;
\fill[blue!60!yellow!10] (4,-4) -- (6,-6) -- (9,0) -- cycle;
\fill[green!60!yellow!10] (1,0) -- (2,2) -- (4,-4) -- cycle;
\fill[green!60!yellow!10] (10,2) -- (7,6) -- (9,0) -- cycle;

\draw (1,0) -- (4,-4) -- (6,-6) -- (9,0) -- (10,2) -- (7,6) -- (5,8)  -- (2,2) -- cycle;
\draw (1,0) -- (9,0) (2,2) -- (10,2);
\draw (2,2) -- (7,6) -- (9,0) -- (4,-4) -- cycle;

\foreach \x in {(1,0), (4,-4), (6,-6), (9,0), (10,2), (7,6), (5,8), (2,2)} \filldraw[fill=black] \x circle (3pt);
\draw  (5,8) -- (5,2) (7,6) -- (7,2) (9,0) -- (9,2);
\draw (2,2) node[left] {\tiny $P_0$} (5,8) node[above] {\tiny $P_1$} (7,6) node[above] {\tiny $P_2$} (10,2) node[right] {\tiny $P_3$} (9,0) node[right] {\tiny $Q_0$} (6,-6) node[below] {\tiny $Q_1$} (4,-4) node[below] {\tiny $Q_2$} (1,0) node[left] {\tiny $Q_3$};
\draw (5,2) node[below] {\tiny $X_1$} (7,2) node[below] {\tiny $X_2$} (9,2.3) node {\tiny $Y$};
\draw (6,-8) node {\tiny Case $x_2 \leq y < x_3$};
\end{tikzpicture}

\end{minipage}
\begin{minipage}[t]{0.3\linewidth}

\centering
\begin{tikzpicture}[scale=0.35]
\fill[green!60!yellow!10] (0,0) -- (3,2) -- (1,-3) -- cycle;
\fill[green!60!yellow!10] (9,2) -- (8,5) -- (6,0) -- cycle;
\fill[red!60!yellow!10] (3,2) -- (5,8) -- (6,0) -- (4,-6) -- cycle;

\draw (0,0) -- (1,-3) -- (4,-6) -- (6,0) -- (9,2) -- (8,5) -- (5,8) -- (3,2) -- cycle;
\draw (0,0) -- (6,0) (3,2) -- (9,2);
\draw (1,-3) -- (3,2) -- (4,-6) (5,8) -- (6,0) -- (8,5);
\draw (5,8) -- (5,2) (6,0) -- (6,2) (8,5) -- (8,2);
\foreach \x  in {(0,0),(1,-3),(4,-6),( 6,0),(9,2), (8,5),(5,8),(3,2)} \filldraw[fill=black] \x circle (3pt);

\draw (3,2) node[left] {\tiny $P_0$} (5,8) node[above] {\tiny $P_1$} (8,5) node[above] {\tiny $P_2$} (9,2) node[right] {\tiny $P_3$} (6,0) node[right] {\tiny $Q_0$} (4,-6) node[below] {\tiny $Q_1$} (1,-3) node[below] {\tiny $Q_2$} (0,0) node[left] {\tiny $Q_3$};

\draw (5,2) node[below] {\tiny $X_1$} (6,2) node[above] {\tiny $Y$} (7.5,2.3) node {\tiny $X_2$};

\draw (4.5,-8) node {\tiny Case $x_1\leq y < x_2$};
\end{tikzpicture}

\end{minipage}
\begin{minipage}[t]{0.3\linewidth}

\centering
\begin{tikzpicture}[scale=0.35]
\fill[red!60!yellow!10] (4,2) -- (7,8) -- (6,0) -- (3,-6) -- cycle;

 \draw (0,0) -- (1,-4) -- (3,-6) -- (6,0) -- (10,2) -- (9,6) -- (7,8) -- (4,2) -- cycle;

 \draw (0,0) -- (6,0) (4,2) --(10,2) (4,2) -- (3,-6) (7,8) -- (6,0);

 \draw (4,2) -- (4,0) (6,0) -- (6,2) (7,8) -- (7,2) (9,6) -- (9,2);

 \draw[red] (4.25,2.5) -- (4.25,0) (7.25,7.75) -- (7.25,2) (9.25,5) -- (9.25,2);

 \foreach \x in {(0,0),(1,-4),(3,-6),( 6,0), (10,2),(9,6),(7,8),(4,2)} \filldraw[fill=black] \x circle (3pt);

 \draw (4,2) node[left] {\tiny $P_0$} (7,8) node[above] {\tiny $P_1$} (9,6) node[above] {\tiny $P_2$} (10,2) node[right] {\tiny $P_3$} (6.5,-0.5) node {\tiny $Q_0$} (3,-6) node[below] {\tiny $Q_1$} (1,-4) node[below] {\tiny $Q_2$} (0,0) node[left] {\tiny $Q_3$};

 \draw (4,0) node[below] {\tiny $Z$} (6,2) node[above] {\tiny $Y$} (7,2) node[below] {\tiny $X_1$} (8.5,2.3) node {\tiny $X_2$};

 \draw (5,-8) node {\tiny Case $0 < y < x_1$};

\end{tikzpicture}

\end{minipage}

\caption{Finding simple cylinders having less intersections with $D$, case $\H(2)$.}
\label{fig:H2:polygon}
\end{figure}
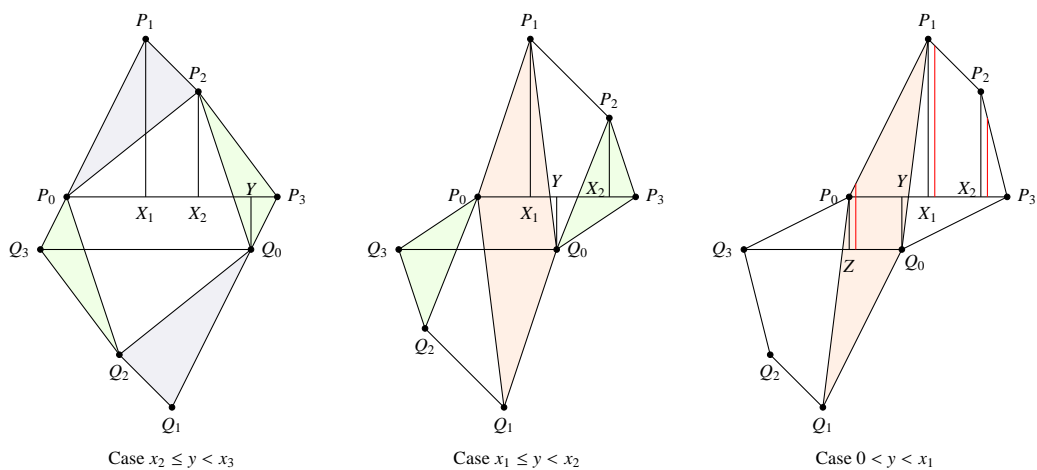

Let $X_1,X_2$, and $Y$  be respectively the vertical projections of $P_1,P_2$, and $Q_0$ on $\ol{P_0P_3}$. Let $x_1,x_2,x_3,y$ be respectively the lengths of $\ol{P_0X_1}, \ol{P_0X_2}, \ol{P_0P_3}, \ol{P_0Y}$. Clearly, we have $0\leq x_1\leq x_2 \leq x_3$ and $0\leq y \leq x_3$. Remark that by cutting and gluing, the cases $y=0$ ($Y\equiv P_0$) and $y=x_3$ ($Y \equiv P_3$) are equivalent. Therefore we can   always suppose $ 0 < y \leq x_3$.

By symmetry, we can assume that $|\ol{P_1X_1}| \geq |\ol{P_2X_2}|$ (see Figure~\ref{fig:H2:polygon}). Observe that the union of the projections of $(P_0P_1P_2)$ and $(Q_0Q_1Q_2)$ in $X$ is a cylinder $E$ which is disjoint from $C$. Similarly, the union of the projections of $(P_2P_3Q_0)$ and $(Q_2Q_3P_0)$ is also a cylinder $F$ in $X$, which is disjoint from $E$. Observe that by assumption, $E$ is always a simple cylinder, but $F$ can be a degenerate one (that is when both $\ol{P_2P_3}$ and $\ol{P_3Q_0}$ are vertical). Note that we have $\dist(C,E)=1$ and  $\dist(C,F)=2$.

Let $d$ be a core curve  of $D$ and $\hat{d}$ be the pre-image of $d$ in $\P$. Remark that $\hat{d}$ is a (finite) union of vertical segments with endpoints in the boundary of $\P$ and none of the vertices of $\P$ is contained in $\hat{d}$. We first consider the generic case, where none of the sides of $\P$ is vertical. By assumption, we have
$$
0< x_1 <x_2 < x_3 \text{ and } 0 < y < x_3.
$$
We have three possibilities:
\begin{itemize}
\item[(a)] $ x_2 \leq y < x$.  We observe that if a vertical line  intersects $\ol{P_0P_2}$ or $\ol{P_2Q_0}$ then it must intersect $\ol{P_0X_2}$ or $\ol{X_2Y}$ respectively. Thus, we have
    $$
    \#\{\hat{d}\cap\ol{P_0P_3}\}\geq \#\{\hat{d}\cap\ol{P_0P_2}\}+\#\{\hat{d}\cap\ol{P_2Q_0}\}.
    $$
    It follows that at least one of the following inequalities is true
    $$
    \left[ \begin{array}{lcl}
    \#\{\hat{d}\cap\ol{P_0P_2}\} < \#\{\hat{d}\cap\ol{P_0P_3}\} & \Rightarrow &   \iota(E,D) < \iota(C,D), \\
    \#\{\hat{d}\cap\ol{P_2Q_0}\} < \#\{\hat{d}\cap\ol{P_0P_3}\} & \Rightarrow &   \iota(F,D) < \iota(C,D).
    \end{array}
    \right.
    $$
Therefore, in this case, we can  choose $C'$ to be either $E$ or $F$.\medskip

\item[(b)] $x_1 \leq y < x_2$. Remark that in this case the parallelogram $(P_0P_1Q_0Q_1)$ is contained in $\P$, thus it projects to a simple cylinder $G$ in $X$, which is disjoint from $F$. In particular, we have $\dist(G,C) \leq 3$. We now observe that
    $$
    \#\{\hat{d}\cap \ol{X_1X_2}\}=\#\{\hat{d}\cap\ol{P_1Q_0}\}+\#\{\hat{d}\cap\ol{P_2Q_0}\} \leq \#\{\hat{d}\cap\ol{P_0P_3}\}.
    $$
    Therefore, at least one of the following inequalities is true $\iota(F,D) < \iota(C,D)$  or $\iota(G,D) < \iota(C,D)$. Hence we can choose $C'$ to be either $F$ or $G$.\medskip

\item[(c)] $0<y< x_1$. We will show that in this case $\iota(G,D) < \iota(C,D)$. Let $Z$ be the vertical projection of $P_0$ to $\ol{Q_0Q_3}$. We choose a core curve $d$ of $D$ which is  contained in the $\eps$-neighborhood of the left boundary of $D$, with $\eps >0$ small. The left boundary of $D$ is a vertical saddle connection, thus it contains (the projection of) one of   the following segments $\ol{P_0Z}, \ol{P_1X_1},\ol{P_2X_2}$. It follows that, $\hat{d}$ contains a vertical segment $\hat{d}_0$ which is $\eps$-close to one of $\ol{P_0Z}, \ol{P_1X_1},\ol{P_2X_2}$ from the right. Observe that $\hat{d}_0$ always intersects $\ol{P_0P_3}$, but when $\eps$ is chosen to be small enough, $\hat{d}_0$ does not intersect $\ol{P_1Q_0}$. Since any vertical segment in $\P$ intersecting $\ol{P_1Q_0}$ must intersect $\ol{YX_1} \subset \ol{P_0P_3}$, it follows that $\iota(G,D) < \iota(C,D)$, and we can choose $C'$ to be $G$.
\end{itemize}

It remains to show that the same arguments work in the degenerating situations, that is when one of the sides of $\P$ is vertical. First, let us suppose that $\ol{P_2P_3}$ is vertical, ({\em i.e.} $x_2=x_3$).
\begin{itemize}
\item[$\bullet$] If $y=x_3$ then $F$ becomes a degenerate cylinder. Clearly $F$ and $D$ are disjoint since they are both vertical. Therefore $\dist(C,D) \leq \dist(C,F)+1  \leq 3$, hence we can choose $C'$ to be $D$.

\item[$\bullet$] If $0<y< x_3$ then  Case (a) and Case (b) then follow from the same arguments. For Case (c),  we observe that the left boundary of $D$ is not invariant by the hyperelliptic involution, and $\ol{P_2P_3}$ projects to an invariant saddle connection. Therefore $\hat{d}_0$ is either $\eps$-close to $\ol{P_0Z}$ or $\ol{P_1X_1}$. Hence we can use the same argument to conclude that $\iota(G,D) < \iota(C,D)$ and we can choose $C'$ to be $G$.
\end{itemize}

Other degenerations are easy to deal with in similar manner, details are left for the reader.
\end{proof}

\subsection{Proof of Proposition~\ref{prop:reduce:inters}, Case $\H(1,1)$.}
\begin{proof}
 Using the notations in Proposition~\ref{prop:polygon:construct}(ii), we know that $(X,\omega)$ is obtained from a  decagon $\P:=(P_0\dots P_4Q_0\dots Q_4) \subset \R^2$.  Our arguments depend on the properties of this decagon. We have three different models for $\P$ (see Figure~\ref{fig:H11:triang:models}):  (I) both $\inter(\ol{P_0P_2})$ and $\inter(\ol{P_2P_4})$ are contained in $\inter(\P)$, (II) only one of $\inter(\ol{P_0P_2})$ and $\inter(\ol{P_2P_4})$ is contained in $\inter(\P)$, and (III) none of  $\inter(\ol{P_0P_2})$ and $\inter(\ol{P_2P_4})$ is contained in $\inter(\P)$.

Let $X_1,X_2,X_3$, and $Y$ be respectively the vertical projections of $P_1,P_2,P_3$, and $Q_0$ on $\ol{P_0P_4}$. The lengths of $\ol{P_0X_i}$, $\ol{P_0P_4}$, and $\ol{P_0Y}$ are denoted by $x_i$, $x_4$, and $y$ respectively.  As in the previous case, we have $0 \leq x_i \leq x_{i+1}, \; i=1,2,3$, and $0 < y \leq x_4$.  Let $d$ be a core curve of $D$, and $\hat{d}$ its pre-image in $\P$.

% Clearly, we have $0 \leq x_i \leq x_{i+1}, _; i=1,2,3$, and $ 0 \leq y \leq x_4$. Since the cases $y=0$ and $y=x_4$  are equivalent (by cutting and regluing), we can always assume that $0<y \leq x_4$. 

\begin{figure}[htb]
\begin{minipage}[t]{0.3\linewidth}

\centering
\begin{tikzpicture}[scale=0.35]
\fill[blue!60!yellow!10] (1,2) -- (3,7) -- (5,6) -- cycle;
\fill[blue!60!yellow!10]  (8,0) -- (6,-5) -- (4,-4) -- cycle;
\fill[green!60!yellow!10] (5,6) -- (8,7) -- (9,2) -- cycle;
\fill[green!60!yellow!10] (0,0) -- (1,-5) -- (4,-4) -- cycle;

\draw (0,0) -- (1,-5) -- (4,-4) -- (6,-5) -- (8,0) -- (9,2) -- (8,7) -- (5,6) -- (3,7) -- (1,2) -- cycle;

\draw (0,0) -- (8,0) (1,2) -- (9,2);

\draw (1,2) -- (5,6) -- (9,2) (0,0) -- (4,-4) -- (8,0);

\foreach \x in {(0,0),(1,2),(9,2),(8,0), (3,7), (5,6), (8,7), (1,-5), (4,-4),(6,-5)} \filldraw[fill=black] \x circle (3pt);

\draw (1,2) node[left] {\tiny $P_0$} (9,2) node[right] {\tiny $P_4$} (0,0) node[left] {\tiny $Q_4$} (8,0) node[right] {\tiny $Q_0$};

\draw (3,7) node[above] {\tiny $P_1$} (5,6) node[above] {\tiny $P_2$} (8,7) node[above] {\tiny $P_3$};

\draw (1,-5) node[below] {\tiny $Q_3$} (4,-4) node[below] {\tiny $Q_2$} (6,-5) node[below] {\tiny $Q_1$};

\draw (4.5,-6.5) node {\tiny Model I};
\end{tikzpicture}

\end{minipage}
\begin{minipage}[t]{0.3\linewidth}

\centering
\begin{tikzpicture}[scale=0.35]
\fill[blue!60!yellow!10] (1,2) -- (3,7) -- (5,6) -- cycle;
\fill[blue!60!yellow!10]  (8,0) -- (6,-5) -- (4,-4) -- cycle;
\fill[green!60!yellow!10] (6,4) -- (8,0)  -- (9,2) -- cycle;
\fill[green!60!yellow!10] (0,0) -- (3,-2) -- (1,2) -- cycle;

\draw (0,0) -- (3,-2) -- (4,-4) -- (6,-5) -- (8,0) -- (9,2) -- (6,4) -- (5,6) -- (3,7) -- (1,2) --  cycle;

\draw (0,0) -- (8,0) (1,2) -- (9,2);

\draw (3,-2) -- (1,2) -- (5,6) (6,4) -- (8,0) -- (4,-4) ;

\draw (3,7) -- (3,2) (5,6) -- (5,2) (6,4) -- (6,2) (8,0) -- (8,2);

\foreach \x in {(0,0),(1,2),(9,2),(8,0), (3,7), (5,6), (6,4), (3,-2), (4,-4),(6,-5)} \filldraw[fill=black] \x circle (3pt);

\draw (1,2) node[left] {\tiny $P_0$} (9,2) node[right] {\tiny $P_4$} (0,0) node[left] {\tiny $Q_4$} (8,0) node[right] {\tiny $Q_0$};

\draw (3,7) node[above] {\tiny $P_1$} (5,6) node[above] {\tiny $P_2$} (6,4) node[above] {\tiny $P_3$};

\draw (3,-2) node[below] {\tiny $Q_3$} (4,-4) node[below] {\tiny $Q_2$} (6,-5) node[below] {\tiny $Q_1$};

\draw (3,2) node[below] {\tiny $X_1$} (5,2) node[below] {\tiny $X_2$} (6,2) node[below] {\tiny $X_3$} (7.75,2) node[below] {\tiny $Y$};

\draw (4.5,-6.5) node {\tiny Model II};

\end{tikzpicture}

\end{minipage}
\begin{minipage}[t]{0.3\linewidth}

\centering
\begin{tikzpicture}[scale=0.35]
\fill[blue!60!yellow!10] (3,3) -- (5,7) -- (6,4) -- cycle;
\fill[blue!60!yellow!10] ( 3,-2) -- (6,-1) -- (4,-5) -- cycle;
\fill[green!60!yellow!10] (6,4) -- (8,0)  -- (9,2) -- cycle;
\fill[green!60!yellow!10] (0,0) -- (3,-2) -- (1,2) -- cycle;

 \draw (0,0) -- (3,-2) -- (4,-5) -- (6,-1) -- (8,0) -- (9,2) -- (6,4) -- (5,7) -- (3,3) -- (1,2) --  cycle;

 \draw (0,0) -- (8,0) (1,2) -- (9,2);

 \draw (3,3) -- (6,4) -- (8,0)  (1,2) -- (3,-2) -- (6,-1);

\draw (3,3) -- (3,2) (5,7) -- (5,2) (6,4) -- (6,2) (8,0) -- (8,2);

\foreach \x in {(0,0),(1,2),(9,2),(8,0), (3,3), (5,7), (6,4), (3,-2), (4,-5),(6,-1)} \filldraw[fill=black] \x circle (3pt);

\draw (1,2) node[left] {\tiny $P_0$} (9,2) node[right] {\tiny $P_4$} (0,0) node[left] {\tiny $Q_4$} (8,0) node[right] {\tiny $Q_0$};

\draw (3,3) node[above left] {\tiny $P_1$} (5,7) node[above] {\tiny $P_2$} (6,4) node[above right] {\tiny $P_3$};

\draw (3,-2) node[below left] {\tiny $Q_3$} (4,-5) node[below] {\tiny $Q_2$} (6,-1) node[below right] {\tiny $Q_1$};

\draw (3,2) node[below] {\tiny $X_1$} (5,2) node[below] {\tiny $X_2$} (6,2) node[below] {\tiny $X_3$} (7.75,2) node[below] {\tiny $Y$};

\draw (4.5,-6.5) node {\tiny Model III};
\end{tikzpicture}
\end{minipage}
\caption{Constructing $(X,\omega)$ from a decagon.}
\label{fig:H11:triang:models}
\end{figure}
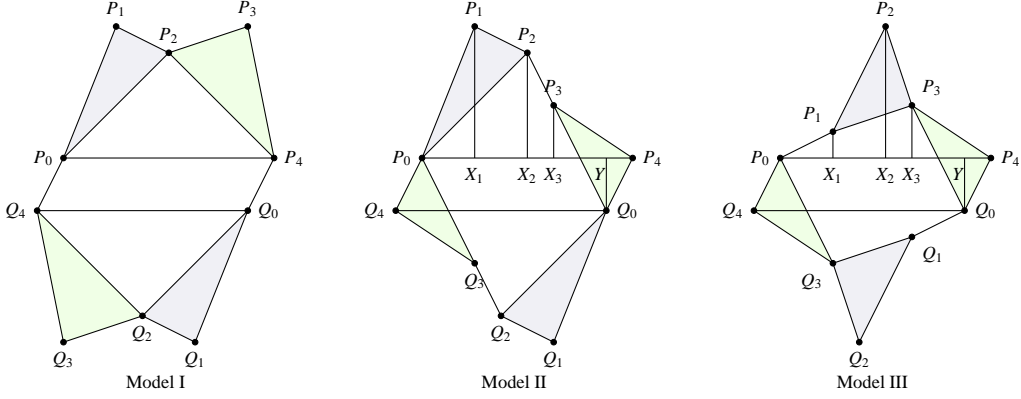

\subsubsection{Model I} In this model, the sets $(P_0P_1P_2)\cup(Q_0Q_1Q_2)$ and $(P_2P_3P_4)\cup(Q_2Q_3Q_4)$ project to two disjoint simple cylinders in $X$ which will be denoted by $E$ and $F$ respectively. Note that $\dist(C,E)=\dist(C,F)=1$.  Clearly, we have
    $$
    \#\{\hat{d}\cap\ol{P_0P_4}\}=\#\{\hat{d}\cap\ol{P_0P_2}\}+\#\{\hat{d}\cap\ol{P_2P_4}\} \Rightarrow \iota(C,D)=\iota(E,D)+\iota(F,D).
    $$
\noindent Therefore, we can pick $C'$ to be $E$ or $F$.

\subsubsection{Model II} By symmetry, we only need to consider the case $\inter(\ol{P_0P_2}) \subset \inter(\P)$, and $\inter(\ol{P_2P_4})\not\subset \inter(\P)$.  Let $E$ be  the simple cylinder on $X$ which is the projection of $(P_0P_1P_2)\cup(Q_0Q_1Q_2)$.  Let $F$ be the cylinder which is the projection of  $(P_3P_4Q_0)\cup(Q_3Q_4P_0)$. We have $\dist(C,E)=1$ and $\dist(C,F)=2$.

We first consider the generic situation, that is  $0< x_i < x_{i+1}, \, i=1,2,3,$ and $ 0< y < x_4$. Note that in this situation $F$ is a simple cylinder.  We have three cases: (a) $ x_2\leq y <x_4$, (b) $x_1\leq y < x_2$, (c)  $0<y<x_1$. In all of these cases, one can find a simple cylinder having  the desired  property by the same arguments as the case $(X,\omega) \in \H(2)$.

% \begin{itemize}
%     \item[$\bullet$] Case (a): $x_2 \leq y <x_4$. We have
%     $$
%     \#\{\hat{d}\cap\ol{P_0P_4}\} \geq \#\{\hat{d}\cap\ol{P_0P_2}\} + \#\{\hat{d}\cap\ol{P_3Q_0}\} \Rightarrow \iota(C,D) \geq \iota(E,D)+\iota(F,D).
%     $$
%     Hence we can choose $C'$ to be either $E$ or $F$.
% 
%     \item[$\bullet$] Case (b): $x_1 \leq y < x_2$. We have another simple cylinder $G$ on $X$ which is the projection of the parallelogram $(P_0P_1Q_0Q_1)$. Since $G$ is disjoint from $F$, we have $  \dist(G,C) \leq \dist(C,F)+1 =3$   and
%     $$
%     \#\{\hat{d}\cap\ol{P_0P_4}\} \geq \#\{\hat{d}\cap\ol{P_1Q_0}\} + \#\{\hat{d}\cap\ol{P_3Q_0}\} \Rightarrow \iota(C,D) \geq \iota(F,D)+\iota(G,D).
%     $$
%     Hence we can choose $C'$ to be either $F$ or $G$.
% 
%     \item[$\bullet$] Case (c): $0<y<x_1$. We will show that $\iota(G,D) < \iota(C,D)$. Let $Z$ be the vertical projection of $P_0$ on $\ol{Q_0Q_4}$. Observe that the left border of $D$ contains one of the projections of $\ol{P_0Z}$, $\ol{P_1X_1}, \ol{P_2X_2}, \ol{P_3X_3}$. We pick a closed geodesic $d$ in the $\eps$-neighborhood of the left border of $D$ with $\eps>0$ small. Then the pre-image $\hat{d}$ of $d$ in  $\P$ contains a vertical segment $\hat{d}_0$ that is $\eps$-close to one of $\ol{P_0Z}$, $\ol{P_1X_1}, \ol{P_2X_2}, \ol{P_3X_3}$ from the right. If we choose $\eps$-small enough, then $\hat{d}_0$ does not intersect $\ol{P_1Q_0}$. Since $\hat{d}_0$  always intersects $\ol{P_0P_4}$, we  conclude that $\iota(G,D) < \iota(C,D)$. Hence we can choose $C'$ to be $G$.
%     \end{itemize}

Consider now the degenerating situations: (1) $\ol{P_0P_1}$ is vertical $ \Leftrightarrow  x_1=0$, (2) $\ol{P_1P_2}$ is vertical $\Leftrightarrow x_1=x_2$, (3) $\ol{P_2P_3}$ is vertical $\Leftrightarrow x_2=x_3$, (4) $\ol{P_3P_4}$ is vertical $\Leftrightarrow x_3=x_4$, (5) $Y \equiv P_4 \Leftrightarrow y=x_4$.  If (4) or (5) does not occur then $F$ is always a simple cylinder, hence the arguments above apply. If (4) and (5) hold then $F$ is a vertical degenerate cylinder.  Since $F$ must be disjoint from $D$, we have $\dist(C,D) \leq 3$. Therefore, we can choose $C'$ to be $D$.

\subsubsection{Model III} In this case $P_2$ must be the highest point of $\P$, and $\ol{P_1P_3}$ must be contained in $\P$. Consequently, the  union $(P_1P_2P_3)\cup(Q_1Q_2Q_3)$ projects to a simple cylinder $E$ in $X$.  Let  $F$ denote the cylinder in $X$ which is the projection of $(P_3P_4Q_0)\cup(Q_3Q_4P_0)$. Remark that $\dist(C,E)=1$ and $\dist(C,F)=2$. It is not difficult to see that the same arguments as the previous cases also  allow us to get the desired conclusion.
\end{proof}

\subsection{Proof of Theorem~\ref{thm:dist:n:inter}}
\begin{proof}
By the Lemmas~\ref{lm:nsimple:cylinder},  \ref{lm:fill:hor:cyl}, we know that there exist two simple cylinders $C'$ and $D'$ such that
$$
\left\{
 \begin{array}{lcl}
\iota(C',D') & \leq & \iota(C,D) \\
\dist(C,D) & \leq & \dist(C',D')+4.
\end{array}
\right.
$$
It follows from Lemma~\ref{lm:vleaf:miss:C} and  Proposition~\ref{prop:reduce:inters}  that $\dist(C',D') \leq 3\iota(C',D')+2$. Therefore
$$
\dist(C,D) \leq 3\iota(C,D)+6.
$$
\end{proof}

\section{Infinite diameter}\label{sec:inft:diam}

In this section we  prove

\begin{Proposition}\label{prop:infty:diam}
 For any $(X,\omega)\in \H(2)\sqcup\H(1,1)$, the diameter of $\hat{\Cc}_{\rm cyl}(X,\omega,f)$ is infinite.
\end{Proposition}

The geometry of the curve complex is closely related to the Teichm\"uller space $\Tcal(S)$. Recall that given a simple closed curve $\gamma$ on $S$, for any $x \in \Tcal(S)$ the extremal length $\Extm_x(\gamma)$  of $\gamma$ is defined to be
$$
\Extm_x(\gamma)=\sup_h |\gamma^*|^2_h,
$$

\noindent where $h$ ranges over the set of Riemannian metrics of area one in the conformal class of $x$, and $|\gamma^*|_h$ is the length of the shortest curve (with respect to $h$) in the homotopy class of $\gamma$. Alternatively, one can define $\Extm_x(\gamma)$ to be the inverse of the largest modulus of an annulus  homotopic to $\gamma$ on $S$.  There is a natural coarse mapping $\Phi$ from $\Tcal(S)$ to $\Cc(S)$ defined as follows, we assign to each $x \in \Tcal(S)$ a curve of minimal $x$-extremal length on $S$. It is a well known fact (see \cite[Lem. 2.4]{MasMin99}) that there is a universal constant $c$ depending only the topology of $S$, such that the diameter of the subset of $\Cc(S)$ consisting of simple curves having minimal $x$-extremal length   is at most $c$ for any $x \in \Tcal(S)$. %\medskip

Teichm\"uller geodesics in $\Tcal(S)$ through $x$ are the projections of the lines $a_t\cdot q$, where $q$ is a holomorphic quadratic differential on $S$ equipped with the conformal structure $x$, and $a_t=\left(\begin{smallmatrix} e^t & 0 \\ 0 & e^{-t} \end{smallmatrix} \right),  \, t \in \R$.  It is proven in \cite{MasMin99} that if $L_q: \R \ra \Tcal(S)$ is a Teichm\"uller geodesic, then $\Phi(L_q(\R))$ is an un-parametrized quasi-geodesic in $\Cc(S)$. It may happen that this quasi-geodesic has finite diameter. \medskip

The curve graph $\Cc(S)$ has infinite diameter (see~\cite{MasMin99}). In \cite{Kla99}, Klarreich shows that the boundary at infinity $\partial_\infty\Cc(S)$ of $\Cc(S)$ can be  identified with the space of topological minimal foliations $\Fcal_{\rm min}(S)$ on $S$. Recall that a foliation on $S$ is minimal if it has  no leaf which is a simple closed curve, here we consider foliations up to isotopy and Whitehead moves. A characterization of sequences of curves  converging to a foliation  in $\partial_\infty \Cc(S)$ is given by Hamenst\"adt \cite{Ham06}. It follows from this result  that if the vertical  of $q$  are minimal then $\Phi\circ L_q([0,\infty))$ is a  quasigeodesic of infinite diameter in $\Cc(S)$ (see \cite{Ham07,Ham10}). \medskip

Recall that a geometric (non-degenerate) cylinder on a translation surface is modeled by $\R\times (0,h)/((x,y)\sim (x+c,y))$, where $c >0$ is its circumference and $h$ is its width. In~\cite{Vor03}, developing Smillie's ideas in~\cite{Smi00}, Vorobets showed the following

\begin{Theorem}[Smillie-Vorobets]\label{thm:cyl:large:width}
 Given any stratum $\H(\kappa)$ of translation surface,  there exists a constant $K>0$ depending on $\kappa$ such that, on every translation surface of area one in $\H(\kappa)$, one can find a geometric cylinder of width bounded below by $K$.
\end{Theorem}

Proposition~\ref{prop:infty:diam} is an easy consequence of this result and the results of Klarreich and Hamenst\"adt.

\begin{proof}[Proof of Proposition~\ref{prop:infty:diam}]
Using the action of $\GL^+(2,\R)$, we can always assume that $\Aa(X,\omega)=1$ and the vertical foliation of $(X,\omega)$ is minimal. Let $L: \R \ra \Tcal(S)$ be the Teichm\"uller geodesic defined by $q=\omega^2$. By the results of Klarreich and Hamenst\"adt, the quasi-geodesic $\Phi\circ L(\R_{>0})$ has infinite diameter.

Denote by $\distc$ the distance in $\Cc(S)$, and by $\dist$ the distance in $\hat{\Cc}_{\rm cyl}(X,\omega,f)$. For any pair $(\alpha,\beta)$ in $\hat{\Cc}_{\rm cyl}(X,\omega,f)$, we have $\distc(\alpha,\beta) \leq \dist(\alpha,\beta)$.  

For each $t \in \R$, let $(X_t,\omega_t):=a_t\cdot (X,\omega)$.  Given any $R \in \R_{>0}$ there exist $t_1,t_2 \in (0,+\infty)$ such that $\distc(\Phi\circ L(t_1), \Phi\circ L(t_2)) \geq R$. Let $\alpha_i:=\Phi\circ L(t_i)$. By Theorem~\ref{thm:cyl:large:width} we know that there is a geometric cylinder $C_i$ of width bounded below by $K$ in $(X_{t_i},\omega_{t_i})$. Let $\beta_i$ be a core curve of $C_i$.

The extremal length of $\alpha_i$ in $X_i$   is bounded  by a universal constant $e_0(S)$ (see e.g \cite[Lem. 2.1]{Min93}). Thus by definition, the length of the shortest curve $\alpha^*_i$ in the homotopy class of $\alpha_i$ with respect to $\omega_{t_i}$ is bounded by $e_0(S)$. Since the width of $C_i$ is at least $K$, we have $\#\{\alpha^*_i\cap \beta_i\} \leq e_0(S)/K$, which implies that $\iota([\alpha_i],[\beta_i]) \leq e_0(S)/K$.

It is well know that the distance in $\Cc(S)$ is bounded by a linear function of the intersection number (see for instance \cite[Lem. 2.1]{MasMin99} or \cite[Lemma 1.1]{Bow06}). Thus there is a constant $M$ depending only on $S$ such that $\distc([\alpha_i],[\beta_i])\leq M$. Therefore, we have
$$
\distc([\beta_1],[\beta_2]) \geq \distc([\alpha_1],[\alpha_2])-\distc([\alpha_1],[\beta_1])-\distc([\alpha_2],[\beta_2]) \geq R-2M
$$
Since $\dist(C_1,C_2)=\dist([\beta_1],[\beta_2]) \geq \distc([\beta_1],[\beta_2])$, the proposition follows.
\end{proof}

\section{Automorphisms of the cylinder graph}\label{sec:automorphisms}

Let $\mathrm{Aff}^+(X,\omega)$ denote the group of affine automorphisms of $(X,\omega)$. Recall that elements of $\mathrm{Aff}^+(X,\omega)$ are orientation preserving homeomorphisms of $X$ that preserve the zero set of $\omega$, and are given by affine maps in local charts of the flat metric out side of this set (see~\cite{KenSmi00, MasTab02}). Remark that the differential of such a map (in local chart associated to the flat metric) is a constant matrix in $\SL(2,\R)$. Thus we have a group homomorphism $D: \mathrm{Aff}^+(X,\omega) \rightarrow \SL(2,\R)$ which associates to each element of $\mathrm{Aff}^+(X,\omega)$ its differential (derivative). The image of $D$ in $\SL(2,\R)$  is called the {\rm Veech group} of $(X,\omega)$ and usually denoted by $\SL(X,\omega)$. Note that the kernel  of $D$ is contained in the group $\mathrm{Aut}(X)$ of automorphisms of $X$, thus must be finite. The group $\SL(X,\omega)$ can also be viewed as the stabilizer of $(X,\omega)$  for the action of $\SL(2,\R)$.

Given a point $[X,\omega,f] \in \Omega\Tcal_2$, via the marking $f:S \ra X$, one can identify $\mathrm{Aff}^+(X,\omega)$ with a subgroup of the Mapping Class Group $\MCG(S)$ of $S$ (see \cite[Section 5]{MasTab02}). An element of $\MCG(S)$ induces naturally an automorphism of the curve graph $\Cc(S)$. It is a well known fact every automorphism  of $\Cc(S)$ arises from  an element of $\MCG(S)$ (\cite{Iva97,Luo00}). Since an affine homeomorphism maps cylinders into cylinders, and saddle connections into saddle connections, it is clear that any element of $\mathrm{Aff}^+(X,\omega)$ induces an automorphism of the subgraph $\hat{\Cc}_{\rm cyl}(X,\omega,f)$. The aim of this section is to show

\begin{Proposition}\label{prop:aut:preserve:cyl}
Let $\phi$ be an element of $\MCG(S)$ which preserves the subgraph $\hat{\Cc}_{\rm cyl}(X,\omega,f)$, that is $\phi(\hat{\Cc}_{\rm cyl}(X,\omega,f)) \subset \hat{\Cc}_{\rm cyl}(X,\omega,f)$. Then $\phi$ is induced by an affine automorphism in $\mathrm{Aff}^+(X,\omega)$. In particular,  $\phi$ realizes an automorphism of $\hat{\Cc}_{\rm cyl}(X,\omega,f)$.
\end{Proposition}

\begin{Remark}\label{rk:auto:prop:equiv}
Proposition~\ref{prop:aut:preserve:cyl} is equivalent to the following statement: if $\psi: X \ra X$ is a homeomorphism satisfying the condition:  for any regular simple closed geodesic or degenerate cylinder $c$,   $\psi(c)$ is freely homotopic to the core curves of a cylinder (possibly degenerate) on  $X$, then $\psi$ is isotopic to an affine automorphism of $(X,\omega)$.
\end{Remark}

The proof of this proposition essentially follows from the arguments of \cite[Lemma 22]{DLR10}. Before getting into the proof, let us recall some basic notions of Thurston's compactification of the Teichm\"uller space. Let $\MF(S)$ denote the space of {\em measured foliations} on $S$. The space of {\em projective measured foliations} denoted by $\PMF(S)$ is naturally the quotient of $\MF(S)$ by $\R_+^*$. Thurston showed that $\PMF(S)$ can be identified with the boundary of $\Tcal(S)$. A foliation is {\em minimal} if none of its leaves is a closed curve. A (measured) foliation is {\em uniquely ergodic} if it is minimal and there exists a unique transverse measure up to scalar multiplication. %\medskip

The set of (free homotopy classes of) simple closed curves in $S$ (that is the vertex set of $\Cc(S)$) is naturally embedded in $\MF(S)$ with the transverse measure being the counting measure of intersections. The geometric intersection number $\iota(.,.)$ defined on the set of pairs of simple closed curves extends to a continuous symmetric function $\iota : \MF(S)\times \MF(S) \ra [0,+\infty)$ which satisfies $\iota(a\lambda,b\mu)=ab\iota(\lambda,\mu)$, for all $a,b \in [0,+\infty)$ and $\lambda,\mu \in \MF(S)$. It has been shown by  Thurston that the set
$$
\{(0,+\infty)\cdot \alpha,  \quad \alpha \hbox{ is a simple closed curve}\}
$$
\noindent is dense in $\MF(S)$.

Two measured foliations are {\em topologically equivalent} if the corresponding topological foliations are the same up to isotopy and Whitehead move. The following result was proved in \cite{Rees81}

\begin{Proposition}\label{prop:min:fol:inters}
 If $\lambda$ is a minimal measured foliation, and $\iota(\lambda,\mu)=0$, then $\lambda$ and  $\mu$ are topologically equivalent.
\end{Proposition}

Measured foliations are a special case of  more general objects called {\em geodesic currents} which were introduced by Bonahon (see \cite{Bon_Ann86, Bon_Invent88}). We refer to \cite{DLR10} for an introduction to this concept with  more details.  While the space of measure foliations is the completion of the set of {\em simple } closed curve, the space of geodesic currents, denoted by $\Cur(S)$, can be viewed as the completion of closed curves on $S$. In particular, we have a continuous extension of the intersection number function $\iota$ to $\Cur(S)\times \Cur(S)$. A characterization of measured foliations in the space of current geodesics was given Bonahon in \cite[Prop. 4.8]{Bon_Ann86}:

\begin{Proposition}[Bonahon]\label{prop:Bon:MF:in:Current}
$\MF(S)$ is exactly the set of geodesic current with zero self-intersection, that is
$$
\MF(S)=\{\lambda \in \Cur(S), \, \iota(\lambda,\lambda)=0\}.
$$
\end{Proposition}

Another important feature of geodesic currents we will need is the following

\begin{Proposition}[Bonahon~\cite{Bon_Invent88}, Prop. 4]\label{prop:Bon:biding:current}
 Let $\alpha$ be a geodesic current with the following property: every geodesic in $\widetilde{S}$ transversely meets another geodesic in the support of $\alpha$. Then the set $\beta \in \Cur(S)$ such that $\iota(\alpha,\beta) \leq 1$ is a compact in $\Cur(S)$.
\end{Proposition}

Remark that if $\lambda$ is a minimal foliation, then the corresponding geodesic current satisfies the hypothesis of Proposition~\ref{prop:Bon:biding:current}. \medskip

Every holomorphic one-form $(X,\omega)$ (or more generally every holomorphic quadratic differential) defines naturally two measured foliations on $X$. The leaves of these foliations are respectively vertical and horizontal geodesic lines with the transverse measures  given by $|\mathrm{Re}\omega|$ and $|\mathrm{Im}\omega|$. It is also a well-known fact that, if $\lambda$ and $\mu$ are two uniquely ergodic measured foliations jointly filling up $S$, that is for any $\nu \in \MF(S)$, we have $\iota(\nu,\lambda)+\iota(\nu,\mu) >0$, then there is a unique Teichm\"uller geodesic $g$ that joins $[\lambda]$ and $[\mu]$, where $[\lambda]$ and $[\mu]$ are the projections of $\lambda$ and $\mu$ in $\PMF(S)$. As a consequence, assume that $(X_1,\omega_1)$ and $(X_2,\omega_2)$ are two holomorphic one-forms both satisfy the following condition: the vertical  foliation of $\omega_i$ is topologically equivalent to $\lambda$, and the horizontal foliation is topologically equivalent to $\mu$. Then there exists a
diagonal matrix $A=\left(\begin{smallmatrix}e^{t} & 0 \\ 0 & e^s \end{smallmatrix} \right) \in \GL^+(2,\R)$ such that $(X_2,\omega_2)=A\cdot(X_1,\omega_1)$. %\medskip

\subsection*{Proof of Proposition~\ref{prop:aut:preserve:cyl}}
\begin{proof}
By definition, $\phi\cdot[X,\omega,f]=[X,\omega,f\circ\phi^{-1}]$. Equivalently, we can write $\phi\cdot[X,\omega,f]=[X',\omega',f']$, where $f': S \ra X'$ satisfies the following condition: there exists an isomorphism $\hat{\phi}:X' \ra X$ such that $\hat{\phi}^*\omega=\omega'$, and $f\circ\phi^{-1}$  is isotopic to $\hat{\phi}\circ f'$.  Using this identification, we have
$$
\hat{\Cc}_{\rm cyl}(X',\omega',f')=\phi(\hat{\Cc}_{\rm cyl}(X,\omega,f)).
$$
\noindent Thus, by assumption, we have $\hat{\Cc}_{\rm cyl}(X',\omega',f') \subset \hat{\Cc}_{\rm cyl}(X,\omega,f)$.

Via the maps $f:S \ra X$, $f': S \ra X'$, for any direction $\theta \in \RP^1$, we denote  by ${\nu}^\theta$ and ${\nu'}^\theta$ the measured foliations on $S$ corresponding to the vertical foliations defined by $e^{\imath\theta}\omega$ and $e^{\imath\theta}\omega'$ respectively.
The leaves of $\nu^\theta$ and ${\nu'}^\theta$ are geodesic lines in the direction of $\pm(\pi/2-\theta)$.  Observe that if $\{\theta_k\}$ is a sequence of angles converging to $\theta$, then $\nu^{\theta_k}$ converges to $\nu^\theta$, and ${\nu'}^{\theta_k}$ converges to ${\nu'}^\theta$ in $\MF(S)$. %\medskip

It follows from a classical result of Kerckhoff-Masur-Smillie~\cite{KMS_Annals86} that for almost all directions $\theta \in \RP^1$, $\nu^\theta$ (resp. ${\nu'}^\theta$) is uniquely ergodic. Set
$$
\UE(\omega):=\{[\nu^\theta] \in \PMF(S), \, \hbox{$\nu^\theta$ is uniquely ergodic, $\theta\in \RP^1$}\} \subset \PMF(S).
$$
\noindent We define $\UE(\omega')$ in the same manner.  %\medskip

We will show that $\UE(\omega')\subset\UE(\omega)$. Let $\theta$ be a direction such that ${\nu'}^\theta$ is uniquely ergodic. Without loss of generality, we can assume that $\Aa(X)=1$. For any $t \in \R$,  set
$$
({X'_t}^\theta,{\omega'_t}^\theta):=\left(\begin{smallmatrix} e^t & 0 \\ 0 & e^{-t} \end{smallmatrix} \right)\cdot(X',e^{i\theta}\omega').
$$
It follows from Theorem~\ref{thm:cyl:large:width} that there exists a constant $R>0$ such that for any $t\in \R$,  ${X'_t}^\theta$ has a cylinder $C'_t$ with  circumference bounded by $R$.  Let $c'_t$ be a core curve of $C'_t$, and consider the sequence $\{c'_k\}_{k\in \N}$.  By definition, the length of $c'_k$ with respect to ${\omega'_k}^\theta$, denoted by $\ell_{{\omega'_k}^\theta}(c'_k)$, is bounded by $R$. Thus we have
$$
\iota(e^k{\nu'}^\theta,c'_k) = e^k\iota({\nu'}^\theta,c'_k) \leq \ell_{{\omega'_k}^\theta}(c'_k) \leq R.
$$
It follows that
$$
\lim_{k\ra +\infty}\iota({\nu'}^\theta, c'_k)=0.
$$
\noindent By Proposition~\ref{prop:Bon:biding:current}, up to extracting a subsequence, we can assume that $\{c'_k\}$ converges to a geodesic current $\mu'\in \Cur(S)$.  Since $c'_k$ has zero self-intersection, it follows that $\iota(\mu',\mu')=0$, hence $\mu' \in \MF(S)$ by Proposition~\ref{prop:Bon:MF:in:Current}. By continuity of $\iota$ we have $\iota({\nu'}^\theta,\mu')=0$.   Since ${\nu'}^\theta$ is uniquely ergodic (so in particular, it is minimal), it follows from Proposition~\ref{prop:min:fol:inters} that $\mu'$ and ${\nu'}^\theta$ are topologically equivalent. Hence $\mu'$ is also uniquely ergodic. %\medskip

By definition, $\{c'_k\}_{k\in \N}$ are vertices of $\hat{\Cc}_{\rm cyl}(X',\omega',f')$. By assumption, we have $\hat{\Cc}_{\rm cyl}(X',\omega',f')\subset \hat{\Cc}_{\rm cyl}(X,\omega,f)$.  Therefore, $\{c'_k\}_{k\in \N}$ are also vertices of $\hat{\Cc}_{\rm cyl}(X,\omega,f)$, which means that  $c'_k$ is freely homotopic to either a simple closed geodesic, or a degenerate cylinder in $X$. In particular, we see that  each $c'_k$ has a well defined direction $\theta_k \in \RP^1$ with respect to $\omega$. Again, by extracting a subsequence, we can assume that  $\{\theta_k\}$ converges to $\hat{\theta}$. Thus, $\{\nu^{\theta_k}\}$ converges to $\nu^{\hat{\theta}}$.  Since we have $\iota(\nu^{\theta_k},c'_k)=0$, by continuity, it follows that $\iota(\nu^{\hat{\theta}},\mu')=0$. Since $\mu'$ is uniquely ergodic, so is $\nu^{\hat{\theta}}$, and we have $[{\nu'}^{\theta}]=[\mu']=[\nu^{\hat{\theta}}] \in \PMF(S)$. We can then conclude that $\UE(\omega')\subset \UE(\omega)$. %\medskip

Now pick a pair of projective uniquely ergodic measured foliations $([\lambda], [\mu])$ in $\UE(\omega')\subset\UE(\omega)$ that jointly fill up $S$. There exist two matrices $M$ and $M'$ such that the vertical and horizontal foliations of $M\cdot[X,\omega,f]$ (resp. of $M'\cdot[X',\omega',f']$) are topologically equivalent to $\lambda$ and $\mu$ respectively. Since there is a unique Teichm\"uller geodesic joining $[\lambda]$ and $[\mu]$, there must exist a diagonal matrix $A \in \GL^+(2,\R)$ such that $M'\cdot[X',\omega',f']=AM\cdot[X,\omega,f]$, which implies that $\phi$ is represented by an affine automorphism of $(X,\omega)$.
\end{proof}

\begin{Remark}
This proof actually works for translation surfaces in any genus with $\hat{\Cc}_{\rm cyl}$ replaced by the subgraph consisting of non-degenerate cylinders.
\end{Remark}

\section{Hyperbolicity}\label{sec:hyperbolic}

A translation surface $(X,\omega)$ is said to be {\em completely periodic} (in the sense of Calta) if the direction of any non-degenerate cylinder in $X$ is periodic, which means that whenever  we find a simple closed geodesic on $X$, the surface decomposes as union of (finitely many) cylinders in the same direction (see~\cite{Cal04, CalSmi08}). It stems out from \cite{Cal04} and \cite{McM07} that, in $\H(2)$, a surface is completely periodic if and only if it is a Veech surface. In $\H(1,1)$, a surface is completely periodic if and only if it is an eigenform for a real multiplication. In particular, there are completely periodic surfaces in $\H(1,1)$ that are not Veech surfaces.  

Let us denote by $\Ecal_D$, where $D$ is a natural number such that  $D \equiv 0$ or $1 \mod 4$,  the locus of eigenforms for the real multiplication by the quadratic order $\mathcal{O}_D$ in $\Omega\mathcal{M}_2$. Each $\Ecal_D$ is a 3 dimensional irreducible (algebraic) subvariety of $\Omega\mathcal{M}_2$ which is
invariant by the $\SL(2,\R)$-action.   The set of eigenforms in $\Omega\mathcal{M}_2$ is then (see \cite{McM07})
$$
\Ecal =\underset{D \equiv 0,1 \mod 4}{\bigcup} \Ecal_D.
$$
Even though complete periodicity is initially defined for directions of non-degenerate cylinders, it is not difficult to show that in the case of genus two, this property actually implies the periodicity for directions of degenerate cylinders (see Lemma~\ref{lm:cp:degen:cyl}). Alternatively, one can also use the argument in~\cite{Wri15} to get the same result in more general contexts (see~\cite{Wri14-survey}). In what follows, by a {\em completely periodic} surface we will mean a surface for which the direction of any cylinder (degenerate or not) is periodic. By Lemma~\ref{lm:cp:degen:cyl}, this apparently new definition agrees with the usual one by Calta. Our goal in this section is to show

\begin{Theorem}\label{thm:cper:hyperbolic}
 If $(X,\omega)\in \H(2)\sqcup\H(1,1)$ is completely periodic then $\hat{\Cc}_{\rm cyl}(X,\omega,f)$ is Gromov hyperbolic.
\end{Theorem}

To prove this theorem, we will use   the following hyperbolicity criterion by Masur-Schleimer~\cite[Theorem 3.13]{MasSch13} (see also \cite[Prop. 3.1]{Bow14}, and \cite{Gil02}), and follow Bowditch's approach in~\cite{Bow06}.

\begin{Theorem}[Masur-Schleimer]\label{thm:hyp:crit:Ma-Sc}
 Suppose that $\mathcal{X}$ is a graph with all edge lengths equal to one. Then $\mathcal{X}$ is Gromov hyperbolic if there is a constant $M \geq 0$, and for all unordered pair of vertices $x,y$ in $\mathcal{X}^0$, there is a connected subgraph $g_{x,y}$ containing $x$ and $y$ with the following properties

 \begin{itemize}
  \item[$\bullet$] (Local) If  $d_\mathcal{X}(x,y) \leq 1$ then $g_{x,y}$ has diameter at most $M$,

  \item[$\bullet$] (Slim triangle) For any $x,y,z \in \mathcal{X}^0$, the subgraph $g_{x,y}$ is contained in the $M$-neighborhood of $g_{x,z}\cup g_{z,y}$.

 \end{itemize}
\end{Theorem}

%\medskip

Let us fix $[X,\omega,f] \in \Omega\Tcal_2$, where $(X,\omega)\in \Ecal$ and $\Aa(X,\omega)=1$. We will  write $\hat{\Cc}_{\rm cyl}$ instead of  $\hat{\Cc}_{\rm cyl}(X,\omega,f)$. We know from  Corollary~\ref{cor:Cess:connect} that $\hat{\Cc}_{\rm cyl}$ is connected, and  by definition the edges of  $\hat{\Cc}_{\rm cyl}$ have length equal to one. Let $K$ be the constant in Theorem~\ref{thm:cyl:large:width}, and $C$ be a cylinder of width bounded below by $K$ in $X$. Note that the circumference of $C$ is bounded above by $1/K$. Recall that from Theorem~\ref{thm:dist:n:inter}, we know that there are two constants $K_1,K_2$ such that for any pair of cylinders $C,D$ in $X$, we have
$$
\dist(C,D) \leq K_1\iota(C,D) +K_2
$$
\noindent where $\dist$ is the distance in $\hat{\Cc}_{\rm cyl}(X,\omega,f)$, and $\iota(C,D)$ is the number of intersections of a core curve of $C$ and a core curve of $D$.

\subsection{Construction of subgraphs connecting pairs of vertices}
We will now construct for each unordered pair of cylinders $C,D$ a subgraph $\hat{\Lc}_{C,D}$ of $\hat{\Cc}_{\rm cyl}$ that satisfies the conditions of Theorem~\ref{thm:hyp:crit:Ma-Sc} with a constant $M$ which will be derived along the way.

Let us first consider the case $C$ and $D$ are parallel. If $C$ or $D$ is non-degenerate then  $\iota(C,D)=0$ hence $\dist(C,D)=1$, which means that $C$ and $D$ are connected by an edge in $\hat{\Cc}_{\rm cyl}$. We define $\hat{\Lc}_{C,D}$ to be this edge. If both $C$ and $D$ are degenerate then it may happen that $\iota(C,D) >0$. Since $(X,\omega)$ is completely periodic, there is a non-degenerate cylinder $E$ parallel to $C$ and $D$. Since  $\iota(C,E)=\iota(D,E)=0$, there are  in $\hat{\Cc}_{\rm cyl}$ two edges connecting $E$ to $C$ and to $D$. In this case, we define $\hat{\Lc}_{C,D}$ to be the union of these two edges. \medskip

Assume from now on that $C$ and $D$ are not parallel. By applying an appropriate element of $\SL(2,\R)$, we can assume that $C$ is horizontal, $D$ is vertical, and $C$ and $D$ have the same circumference. For any $t\in \R$,  set
$$
a_t=\left(\begin{array}{cc} e^t & 0 \\ 0 & e^{-t} \end{array} \right) \text { and } (X_t,\omega_t)=a_t\cdot(X,\omega).
$$
For any saddle connection $s$ in $(X,\omega)$, we will denote by $\ell_t(s)$ its Euclidean length in $(X_t,\omega_t)$. If $E$ is a cylinder in $(X,\omega)$, then $c_t(E)$ and $w_t(E)$ are respectively its circumference and width in $(X_t,\omega_t)$.

For any $R \in \R_{>0}$, let $\Lc^*_{C,D}(t,R)$ denote set of cylinders (possibly degenerate) of circumference bounded above by $R$ in $(X_t,\omega_t)$. Note that this set is finite.  Let us choose a constant $L_1$ such that

\begin{equation}\label{eq:defn:L1}
 L_1 > \max\{\frac{1}{K}, 9\},
\end{equation}

\noindent and  define
$$
\Lc^*_{C,D}(L_1)= \bigcup_{t \in \R} \Lc^*_{C,D}(t,L_1).
$$
\noindent We regard $\Lc^*_{C,D}(t,R)$ and $\Lc^*_{C,D}(L_1)$ as subsets of $\hat{\Cc}_{\rm cyl}^{(0)}$. Observe that $\Lc^*_{C,D}(t,L_1)$ contains $C$ when $t$ tends to $-\infty$, and contains  $D$ when $t$ tends to $+\infty$, therefore $\Lc^*_{C,D}$ contains $C$ and $D$. \medskip

For each $t \in \R$, consider now the set $\Lc^*_{C,D}(t,2L_1)$. From Theorem~\ref{thm:cyl:large:width}, $\Lc^*_{C,D}(t,2L_1)$ contains a vertex corresponding to a cylinder  $C_{0,t}$ of width bounded below by $K$.  Set
\begin{equation}\label{eq:def:hyp:const:M}
M_1:=\max\{2(2\frac{K_1L_1}{K}+K_2), 2\}
\end{equation}
Then we have

\begin{Lemma}~\label{lm:LCD:t:diam:bound}
 As subset of $\hat{\Cc}_{\rm cyl}$,  $\Lc^*_{C,D}(t,2L_1)$ has diameter bounded by $M_1$.
\end{Lemma}
\begin{proof}
Let $E$ be a cylinder in $\Lc^*_{C,D}(t,2L_1)$. If $\iota(E,C_{0,t})=0$, then we have $\dist(C_{0,t},E)=1$. Otherwise we have $K \iota(E,C_{0,t}) \leq \ell_t(E) \leq 2L_1$. Hence, from \eqref{eq:rel:dist:inter} we get
$$
\dist(C_{0,t},E) \leq 2\frac{K_1L_1}{K}+ K_2,
$$
\noindent and the lemma follows.
\end{proof}

Moreover, we have

\begin{Lemma}\label{lm:T:geod:infty:cyl}
Assume that the surface $(X,\omega)$ admits cylinder decompositions in both vertical and horizontal directions. Then there exists a constant $T>0$ such that
\begin{itemize}
 \item[$\bullet$] if $t>T$ then $\Lc^*_{C,D}(t,2L_1)$ only contains the vertical cylinders in $(X,\omega)$ and $\Lc^*_{C,D}(t,2L_1)$ has diameter at most $2$,

 \item[$\bullet$] if $t<-T$ then $\Lc^*_{C,D}(t,2L_1)$ only contains the horizontal cylinders in $(X,\omega)$ and $\Lc^*_{C,D}(t,2L_1)$ has diameter at most $2$.
\end{itemize}
\end{Lemma}
\begin{proof}
We only give the proof of the first assertion as the second one follows from the same argument. By assumption, $X$ decomposes into the union of some non-degenerate vertical cylinders $D_1,\dots,D_k$. Let $w_t(D_i)$ denote the width of $D_i$ in $(X_t,\omega_t)$. Let $w_t = \min \{w_t(D_i), \,  i=1,\dots,k\}$.  A non-vertical cylinder must cross one of $D_i$, thus its circumference is bounded below by $w_t$ in $(X_t,\omega_t)$. Since   we have $w_t=e^tw_0$, if $t$ is large enough any non-vertical cylinder has circumference at least $2L_1$ in $(X_t,\omega_t)$. Hence $\Lc^*_{C,D}(t,2L_1)$ only contains  the vertical cylinders.  Since any vertical cylinder is of distance one from $D_1$ in $\hat{\Cc}_{\rm cyl}$, $\Lc^*_{C,D}(t,2L_1)$ has diameter at most two.
\end{proof}

\begin{Lemma}\label{lm:a_t:action:length}
If $t-\log(2) \leq t' \leq t+\log(2)$ then $\Lc^*_{C,D}(t',R) \subset \Lc^*_{C,D}(t,2R)$ for any $R \in \R_{>0}$.   In particular $C_{0,t'} \in \Lc^*_{C,D}(t,2L_1)$.
\end{Lemma}
\begin{proof}
Let $s$ be a saddle connection or a regular geodesic in $(X_{t'},\omega_{t'})$. Let $x +\imath y$ be the period of $s$ in $(X_{t'},\omega_{t'})$. Note that $(X_t,\omega_t)=a_{t-t'}\cdot(X_{t'},\omega_{t'})$. Thus the period of $s$ in $(X_t,\omega_t)$ is $(e^{t-t'}x,e^{t'-t}y)$.  Therefore,
$$
\ell_{t}(s) =\sqrt{e^{2(t-t')}x^2 + e^{2(t'-t)}y^2} \leq  2\sqrt{x^2+y^2} =2\ell_{t'}(s).
$$
\end{proof}

Set
$$
\bar{\Lc}_{C,D}(2L_1):= \bigcup_{k \in \Z} \Lc^*_{C,D}(k\log(2),2L_1) \subset \hat{\Cc}_{\rm cyl}^{(0)}.
$$
It follows from Lemma~\ref{lm:T:geod:infty:cyl} that if $n \in \N$ is large enough then for any $m>n$, $\Lc^*_{C,D}(m,L_1)=\Lc^*_{C,D}(n,2L_1)$, and $\Lc^*_{C,D}(-m,2L_1)=\Lc^*_{C,D}(-n,2L_1)$. Therefore, the set $\bar{\Lc}_{C,D}(2L_1)$ is actually finite.  For each unordered pair $(x,y)$ of vertices in $\bar{\Lc}_{C,D}(2L_1)$, let $\Gamma(x,y)$ be a path of minimal length in $\hat{\Cc}_{\rm cyl}$ joining $x$ to $y$. Set
$$
\hat{\Lc}_{C,D}(2L_1)=\bigcup_{x,y\in \bar{\Lc}_{C,D}(2L_1)} \Gamma(x,y).
$$
As a direct consequence of Lemma~\ref{lm:a_t:action:length}, we get
\begin{Corollary}\label{cor:LCD:in:hLCD}\hfill
\begin{itemize}
 \item[a)]  If $x \in \Lc^*_{C,D}(t,2L_1)$ and $y \in \Lc^*_{C,D}(t',2L_1)$, then $\dist(x,y) \leq M_1(2+\frac{|t-t'|}{\log(2)})$.

 \item[b)]  The set $\Lc^*_{C,D}(L_1)$ is contained in $\bar{\Lc}_{C,D}(2L_1)$ and $\bar{\Lc}_{C,D}(2L_1)$ is contained in the $M_1$-neighborhood of $\Lc^*_{C,D}(L_1)$.

 \item[c)]  For any pair of vertices $(x,y) \in\Lc^*_{C,D}(L_1)\times \Lc^*_{C,D}(L_1)$, there is a path  $\Gamma(x,y)$ in $\hat{\Lc}_{C,D}(2L_1)$ from $x$ to $y$ of length equal to $\dist(x,y)$.

 \end{itemize}
\end{Corollary}

\subsection{Local property for $\hat{\Lc}_{C,D}$}
We will now show that the subgraphs $\hat{\Lc}_{C,D}(2L_1)$ constructed above satisfy the first condition of Theorem~\ref{thm:hyp:crit:Ma-Sc}.
\begin{Proposition}\label{prop:LCD:dist:1}
There exists a constant $M_2$ such that if $(X,\omega) \in \Ecal$ then for any pair of cylinders $C,D$ in  $(X,\omega)$ such that $\iota(C,D)=0$, we have $\diam \hat{\Lc}_{C,D}(2L_1) \leq M_2$.
\end{Proposition}

%This proposition is a key step to our proof of hyperbolicity of $\hat{\Cc}_{\rm cyl}$. Its importance can be highlighted in the following situation: assume that there is cylinder $D'$ parallel to $D$ and  %having a great number of intersections with $C$. Since $D'$ must be disjoint from $D$, we have $\dist(C,D')=2$. Because $D'$ has large intersection number with $C$, one may expect that the %diameter of $\hat{\Lc}_{C,D'}(2L_1)$ is large. But since the Hausdorff distance between $\hat{\Lc}_{C,D'}(2L_1)$ and $\hat{\Lc}_{C,D}(2L_1)$ is uniformly bounded (they both contain %$\Lc^*_{C,D}(L_1)$), it follows from Proposition~\ref{prop:LCD:dist:1} that the diameter of $\hat{\Lc}_{C,D'}(2L_1)$ is actually uniformly bounded.

To prove this proposition, we make use of an elementary result on slit tori (cf. Lemma~\ref{lm:s:tor:cyl:large:wd}), and the fact that if $C$ and $D$ are not parallel, then there always exists a splitting of $X$ into two subsurfaces, each of which contains one of $C$ and $D$. Those auxiliary results are proved in the appendix. The main technical difficulties arise when we have to deal with degenerate cylinders.  We  split the proof of Proposition~\ref{prop:LCD:dist:1} into two cases $(X,\omega) \in \H(2)$ and $(X,\omega) \in \H(1,1)$.  %\medskip

\subsection*{Proof of Proposition~\ref{prop:LCD:dist:1}, Case $\H(2)$}
\begin{proof}
If $C$ and $D$ are parallel then $\hat{\Lc}_{C,D}(2L_1)$ has diameter bounded by $2$ and we have nothing to prove. Suppose from now on that $C$ is horizontal, $D$ is vertical, $C$ and $D$ have the same circumference equal to $\ell$, and $\hat{\Lc}_{C,D}(2L_1)$ is the graph constructed above. Note that in this case $(X,\omega)$ is a Veech surface, thus both horizontal and vertical  directions are periodic. \medskip
%Recall that for  a surface in $\H(2)$, the complement of the closure of a non-simple cylinder is either empty or another simple cylinder in the same direction. It follows that neither  of $C,D$ is  non-simple (otherwise the other one can not exist).\\

\noindent \underline{Case 1:} one of $C$ or $D$ is non-degenerate. Assume that $C$ is non-degenerate. Let $c$ be a core curve of $C$ and $d$ a core curve of $D$. Note that $c$ is a regular simple closed geodesic. By Lemma~\ref{lm:2cyl:inters}, the condition $\iota(C,D)=0$ implies that $c\cap d=\vide$. Clearly, $C$ cannot fill $X$. If $C$ is not simple then the complement of $\ol{C}$ is a horizontal simple cylinder $C'$ whose boundary is contained in the boundary of $C$. Since $D$ is disjoint from $C$, it must be contained in $C'$. But this  is impossible since $C'$ is horizontal and $D$ is vertical. Therefore, $C$ must be a simple cylinder.

The complement of $C$ is then a slit torus with the slit corresponding to the boundary of $C$. Remark that a core curve of $D$ must be disjoint from the interior of the slit, otherwise it would cross $C$ entirely. Thus, we have in the slit torus an embedded square bounded by the boundary of $D$ and the slit (which is actually the boundary of $C$) (see Figure~\ref{fig:LCD:d1:H2}). By assumption, the length of the sides of this square is $\ell$.  Since this square has area less than one,  we must have $\ell < 1$. Therefore $C \in \Lc^*_{C,D}(t,L_1)$ for all $t \leq 0$, and $D\in \Lc^*_{C,D}(t,L_1)$ for any $t \geq 0$. Hence any $E\in \bar{\Lc}_{C,D}(2L_1)$ is of distance at most $M_1$ from $C$ or from $D$. Thus $\diam\hat{\Lc}_{C,D}(2L_1) \leq 2M_1+1$.

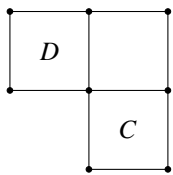
\begin{figure}[htb]
\begin{minipage}[t]{0.45\linewidth}
\centering
\begin{tikzpicture}[scale=0.35]
\draw[thin] (0,6) -- (0,3) -- (3,3) -- (3,0) -- (6,0) -- (6,6) -- cycle;
\draw[thin] (3,6) -- (3,3) -- (6,3);
\foreach \x in {(0,6),(0,3),(3,6),(3,3),(3,0),(6,6),(6,3),(6,0)} \filldraw[fill=black] \x circle (3pt);
\draw (1.5,4.5) node{\small $D$} (4.5,1.5) node {\small $C$};
\end{tikzpicture}
\end{minipage}
%\begin{minipage}[t]{0.45\linewidth}
%\centering
%\begin{tikzpicture}[scale=0.35]
%\draw[thin] (0,6) -- (0,0) -- (9,0) -- (9,6) -- cycle;
%\draw[thin] (3,6) -- (3,0);
%\draw[very thick] (3,6) -- (9,6) (3,0) -- (9,0);
%
%\foreach \x in {(0,6),(0,0),(3,6), (3,0),(5,6),(7,0),(9,6),(9,0)} \filldraw[fill=black] \x circle (3pt);
%
%\draw (1.5,3) node {\small $D$} (6,6) node[above] {\small $C$} (6,0) node[below] {\small $C$};
%\end{tikzpicture}
%\end{minipage}

\caption{Disjoint simple cylinders on surfaces in $\H(2)$}
\label{fig:LCD:d1:H2}
\end{figure}

\medskip

\noindent \underline{Case 2:} both of  $C$ and $D$ are degenerate.  From Lemma~\ref{lm:degen:cyl:deform}, for any $\eps >0$ small enough, we can deform $(X,\omega)$ into another surface $(X',\omega')$ such that
\begin{itemize}
 \item[$\bullet$] $C$ corresponds to a simple horizontal cylinder $C'$ in $X'$ of width $\eps$,

 \item[$\bullet$] $D$ corresponds to a vertical cylinder in $X'$.
\end{itemize}
Since $\iota(C',D')=\iota(C,D)=0$, it follows from Lemma~\ref{lm:2cyl:inters} that $D'$ must be disjoint from $C'$. It follows in particular that $D$ and $D'$ have the same circumference $\ell$. By construction $C'$ has the same  circumference as $C$, and $\Aa(X',\omega')=\Aa(X,\omega)+\eps\ell=1+\eps\ell$. Applying the same arguments as above to $(X',\omega')$, we see that $X'$ contains an embedded square of size $\ell$ disjoint from $C'$. Therefore we have $\ell^2 < 1+\eps\ell$.  Since $\eps$ can be chosen arbitrarily, we derive that $\ell \leq 1$. We can then conclude by the same arguments as the previous case.
\end{proof}

\subsection*{Proof of Proposition~\ref{prop:LCD:dist:1}, Case $\H(1,1)$}\hfill

\begin{proof}
Again, we only have to consider the case $C$ and $D$ are not parallel. Thus we can assume that $C$ is horizontal and $D$ is vertical. We first choose a positive real number $L > \sqrt{2}$ such that
\begin{equation}\label{eq:L1:case:H11}
L_1 \geq 3f(\sqrt{2}L)
\end{equation}
\noindent where $f(x)=\sqrt{x^2+1/x^2}$ (see Lemma~\ref{lm:s:tor:cyl:large:wd}).

\medskip

\noindent \underline{Case 1:} one of $C$ and $D$ is a simple cylinder. By Lemma~\ref{lm:2djt:sim:cyl}, we need to consider two cases  (see Figure~\ref{fig:H11:CD:sim:split})

\begin{figure}[htb]
\begin{minipage}[t]{0.4\linewidth}
\centering
\begin{tikzpicture}[scale=0.4]
\draw[thin] (0,10) -- (0,6) -- (2,6) -- (2,4) -- (6,0) -- (6,6) -- (7,7) -- (3,11) -- (2,10) -- cycle;
\draw[thin] (2,10) -- (2,6) -- (6,6) -- (2,10) (2,4) -- (6,4);

\foreach \x in {(0,10),(0,6), (2,10), (2,6), (2,4), (3,11), (6,6), (6,4), (6,0), (7,7)} \filldraw[fill=black] \x circle (3pt);
\draw (1,8) node {\tiny $D$} (4,5) node {\tiny $C$} (3,7) node {\tiny $\T$} (5,2.5) node {\tiny $\T'$} (4.5,8.5) node {\tiny $E$};
\end{tikzpicture}
\end{minipage}
\begin{minipage}[t]{0.4\linewidth}
\centering
\begin{tikzpicture}[scale=0.4]
\draw[thin] (0,8)--(0,2) -- (2,2) -- (2,0) -- (8,0) -- (8,2) -- (10,4) -- (4,4) --(4,10) -- (2,8) -- cycle;
\draw[thin] (2,8) -- (2,2) -- (8,2) (2,2) -- (4,4);

\foreach \x in {(0,8),(0,2), (2,8), (2,2), (2,0), (4,10), (4,4), (8,2), (8,0), (10,4)} \filldraw[fill=black] \x circle (3pt);
\draw (1,5) node {\tiny $D$} (5,1) node {\tiny  $C$};
\draw (3,3.5) node {\tiny $s_1$} (3,9.5) node {\tiny $s_2$} (9,2.5) node {\tiny $s_2$};
\draw (6,3) node {\tiny $\P'$} (3,6) node {\tiny $\P''$};
\end{tikzpicture}
\end{minipage}
\caption{Disjoint cylinders on surfaces in $\H(1,1)$: one of $C$ and $D$ is simple.}
\label{fig:H11:CD:sim:split}
\end{figure}
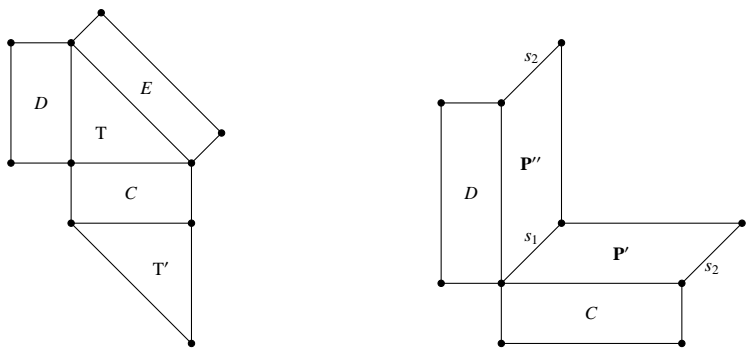

\begin{itemize}
\item[(i)] There is a simple cylinder $E$ disjoint from $C\cup D$ and the complement of $C\cup D \cup E$ is the union of two triangles $\T, \T'$ (see Figure~\ref{fig:H11:CD:sim:split} left). Since we have $\Aa(\T)+\Aa(\T')=\ell^2 < \Aa(X,\omega)=1$, it follows $\ell <1$. Hence we  can  use the same argument as in the case $(X,\omega) \in \H(2)$ to conclude that $\diam\hat{\Lc}_{C,D}(2L_1) \leq 2M_1+1$.\\

\item[(ii)] There is a pair of homologous saddle connections $s_1,s_2$ that decompose $X$ into a connected sum of two slit tori, $(X',\omega',s')$ containing $C$ and $(X'',\omega'',s'')$ containing $D$ (see Figure~\ref{fig:H11:CD:sim:split} right).

By construction, the complement of $C$ in $X'$  is an embedded parallelogram $\P'$  bounded by $s_1,s_2$  and the boundary of $C$.    Similarly, the complement of $D$ in $X''$ is also an embedded parallelogram $\P''$ bounded by $s_1,s_2$ and the boundary of $D$.     If $\ell \leq 1$ then we can conclude using the argument above. Suppose that we have $\ell \geq 1$. Let $\omega(s_i)=x + \imath y $.  Since we have $\Aa(\P')=|y|\ell$, and $\Aa(\P'')=|x|\ell$, it follows 
$$\max\{|x|,|y|\} \leq 1/\ell  \leq 1 \text{ and } |s_i|=\sqrt{x^2+y^2} \leq \sqrt{2}/\ell \leq \sqrt{2}.$$

\noindent  Set $A_1=\Aa(X',\omega'), \, A_2=\Aa(X'',\omega'')$, we have $A_1+A_2=1$.  Without loss of generality, let us suppose that $A_1 \geq 1/2$. For any $t \in \R$, the period of $s_i$ in $(X_t,\omega_t)$ is $(e^tx,e^{-t}y)$. Let $(X'_t,\omega'_t,s'_t)$ be the slit torus corresponding to $(X',\omega',s')$ in $(X_t,\omega_t)$. Recall that  we have chosen $L>\sqrt{2}$ and $L_1$ satisfies  \eqref{eq:L1:case:H11}. Let us choose a positive real number $L'\geq 1$ such that
$$
L \geq \sqrt{{L'}^2+1}.
$$
\begin{itemize}
\item[$\bullet$] For $ 0 \leq t \leq \log(\ell L')$, we have $ e^t|x| \leq L'$ and $e^{-t}|y| \leq |y| \leq 1$, thus $\ell_t(s_1) \leq \sqrt{{L'}^2+1} \leq L$. Rescaling $(X'_t,\omega'_t,s'_t)$ by $\frac{1}{\sqrt{A_1}}$, we get a torus of area one with a slit of length bounded by $\sqrt{2}L$. Using Lemma~\ref{lm:s:tor:cyl:large:wd}, we see that there exists in $\frac{1}{\sqrt{A_1}}\cdot X'_t$ a cylinder $E'_t$ disjoint from the slit of circumference bounded by $L_1$. Note that in $X'_t$, the circumference of $E'_t$ is at most $\sqrt{A_1}L_1 \leq L_1$. We have  $\dist(D,E'_t)=1$ and $E'_t\in \Lc^*_{C,D}(t,2L_1)$. Thus for any $E \in \Lc^*_{C,D}(t,2L_1)$ we have $\dist(D,E) \leq M_1+1$.

\item[$\bullet$] For $-\log(\ell L') \leq t \leq 0$, we have $e^t|x| \leq |x| \leq 1$ and $e^{-t}|y| \leq L'$, thus $ \ell_t(s_i) \leq \sqrt{{L'}^2+1} \leq L$. The same argument as the previous case then shows that $\dist(D,E) \leq M_1+1$, for any $E\in \Lc^*_{C,D}(t,2L_1)$.

\item[$\bullet$] For $t\geq \log(\ell L')$ we have $\ell_t(D)=e^{-t}\ell \leq 1/{L'}  \leq 1 \leq 2L_1$. Thus $D\in \Lc^*_{C,D}(t,2L_1)$ which implies that $\dist(D,E) \leq M_1$ for any $E \in \Lc^*_{C,D}(t,2L_1)$.

\item[$\bullet$] For $t \leq -\log(\ell L')$ we have $\ell_t(C) \leq 1/{L'} \leq 2L_1$, hence for any $E\in \Lc^*_{C,D}(t,2L_1)$,  $\dist(C,E) \leq M_1$, which implies that $\dist(D,E)\leq M_1+1$.
\end{itemize}
\end{itemize}
We can then conclude that for any $t\in \R$, and any $E\in \Lc^*_{C,D}(t,2L_1)$, we have $\dist(D,E)\leq M_1+1$. Hence $\diam\hat{\Lc}_{C,D} \leq 2(M_1+1)$.

\medskip

\noindent \underline{Case 2:} one of $C,D$ is non-degenerate and  not simple. Without loss of generality, we can assume that $C$ is neither simple nor degenerate. Lemma~\ref{lm:2cyl:inters} implies that $D$ is disjoint from $C$.  Since $C$ is not simple, the complement of $\overline{C}$ is either (a) empty, (b) a horizontal simple cylinder, (c) the union of two simple horizontal cylinders, or (d) another horizontal cylinder whose closure is a slit torus. Since there exists a vertical cylinder disjoint from $C$ (namely $D$), only (d) can occur. In this case, there are a pair of horizontal homologous saddle connection $\{s_1,s_2\}$ contained in the boundary of $C$ that decompose $(X,\omega)$ into the connected sum of two slit tori.  Let $(X',\omega',s')$ be the slit torus which is the closure of $C$, and $(X'',\omega'',s'')$ be the other one that contains $D$ (see Figure~\ref{fig:LCD:d1:H11:C:n:sim}).  
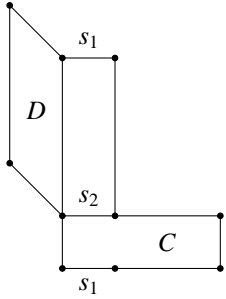
\begin{figure}[htb]
\begin{minipage}[t]{0.5\linewidth}
\centering
\begin{tikzpicture}[scale=0.35]
\draw[thin] (-2,10) -- (-2,4) -- (0,2) -- (0,0) -- (6,0) -- (6,2) -- (2,2) -- (2,8) -- (0,8) -- cycle;
\draw[thin] (0,8) -- (0,2) -- (2,2);

\foreach \x in {(-2,10), (-2,4), (0,8), (0,2), (0,0), (2,8), (2,2), (2,0), (6,2), (6,0)} \filldraw[fill=black] \x circle (3pt);
\draw (-1,6) node {\small $D$} (4,1) node {\small $C$} (1,8) node[above] {\small $s_1$} (1,2) node[above] {\small $s_2$} (1,0) node[below] {\small $s_1$};
\end{tikzpicture}
\end{minipage}
\caption{Disjoint cylinders on surfaces in $\H(1,1)$: $C$ is not simple nor degenerate.}
\label{fig:LCD:d1:H11:C:n:sim}
\end{figure}

Let $x=|s_1|=|s_2|$. Observe that $X''$ contains a rectangle bounded by $s_1,s_2$ and the saddle connections bordering $D$. Therefore we have $ x\ell \leq 1 \Leftrightarrow 0 \leq x \leq 1/\ell$. By the same arguments as the previous case, we also get $\diam\hat{\Lc}_{C,D}\leq 2(M+1)$.  \medskip

\noindent \underline{Case 3:} one of $C$ and $D$ is degenerate.   Let us assume that $C$ is degenerate. Using Lemma~\ref{lm:degen:cyl:deform}, we can find a family $(X_t,\omega_t), t \in [0,\eps)$, of surfaces in $\H(1,1)$ that are deformations of $(X,\omega)$, such that $C$ corresponds to a simple horizontal cylinder $C_t$ on $X_t$, for $t>0$, which has the same circumference. Note that the width of $C_t$ is $t$. Therefore $\Aa(X_t,\omega_t)=\Aa(X,\omega)+t\ell$.

By construction, $D$ corresponds to a cylinder $D_t$ on $X_t$ which is disjoint from $C_t$ (since we have $\iota(C_t,D_t)=\iota(C,D)=0$). By Lemma~\ref{lm:2djt:sim:cyl} we know that either (i) $(X_t,\omega_t)$ contains two embedded triangles $\T,\T'$ disjoint from $C_t$ and $D_t$, or (ii) there is a splitting of $(X_t,\omega_t)$ into  two slit tori $(X'_t,\omega'_t,s'_t)$ and $(X''_t,\omega''_t,s''_t)$ such that $C_t \subset X'_t$ and $D_t\subset X''_t$. 

If (i) occurs, then we have $\Aa(\T)=\Aa(T')=\ell^2/2 \leq 1/2$, which implies that $\ell \leq 1$. If (ii) occurs, then since the slits ($s'$ and $s''$) are disjoint from $C_t$, they persist  as we collapse $C_t$ to get back $(X,\omega)$. Thus, we have  the same splitting on $(X,\omega)$. In conclusion, we can use the same arguments as in Case 1 to handle this  case.  The proof of  Proposition~\ref{prop:LCD:dist:1} is now complete.
\end{proof}

\subsection{Slim triangle property for $\hat{\Lc}_{C,D}$}

We now prove  that the subgraphs $\hat{\Lc}_{C,D}(2L_1)$ satisfy the second property of Theorem~\ref{thm:hyp:crit:Ma-Sc}. The idea of the proof can found in \cite[Lemma 4.4]{Bow06}. To alleviate the notations, in what follows we will write $\hat{\Lc}_{C,D}$ instead of $\hat{\Lc}_{C,D}(2L_1)$.

\begin{Proposition}\label{prop:LCD:sim:tria}
There exists a constant $M_3$ such that for any triple of cylinders $\{C,D,E\}$ in $(X,\omega)$, we have $\hat{\Lc}_{C,D}$ is contained in the $M_3$-neighborhood of $\hat{\Lc}_{C,E}\cup\hat{\Lc}_{E,D}$ in $\hat{\Cc}_{\rm cyl}(X,\omega,f)$.
\end{Proposition}
\begin{proof}

If $C$ and $D$ are parallel then $\hat{\Lc}_{C,D}$ is contained in the $2$-neighborhood of $\hat{\Lc}_{C,E}\cup\hat{\Lc}_{D,E}$. From now on we assume that $C$ and $D$ are not parallel.\medskip

By Corollary~\ref{cor:LCD:in:hLCD}, we only need to show that $\Lc^*_{C,D}(L_1)$ is contained in the $M_3$-neighborhood of $\Lc^*_{C,E}(L_1)\cup\Lc^*_{E,D}(L_1)$. Remark that to define $\bar{\Lc}_{C,D}(2L_1)$ and $\hat{\Lc}_{C,D}(2L_1)$ one needs to specify an origin for the time $t$ by the condition that the circumferences of $C$ and $D$ are equal. On the other hand to define $\Lc^*_{C,D}(L_1)$ this normalization is not required. If $E$ is parallel to $C$ then $\Lc^*_{C,D}(L_1)=\Lc^*_{E,D}(L_1)$, and if $E$ is parallel to $D$ then $\Lc^*_{C,D}(L_1)=\Lc^*_{C,E}(L_1)$. In both of these cases we have nothing to prove.

Let us now assume that $E$ is neither parallel to $C$ nor to $D$. We can then renormalize (using $\SL(2,\R)$) such that $C$ is horizontal, $D$ is vertical, and $E$ has slope equal to one. Recall that for any $t\in \R, \,  (X_t,\omega_t)=a_t\cdot(X,\omega)$, $C_{0,t}$ is a cylinder of width bounded below by $K$ in $(X_t,\omega_t)$, and the constant $L_1$ is chosen so that $L_1 > 1/K$ (see  \eqref{eq:defn:L1}). \medskip

\noindent {\em Claim:} if $t \leq 0$ then $C_{0,t}$ is contained in the $M_1$-neighborhood of $\Lc^*_{C,E}(L_1)$.
\begin{proof}[Proof of the claim]
Since $(X,\omega)$ is completely periodic, it  decomposes into cylinders in both directions of $C$ and $E$.  Let us denote by $C=C_1,\dots,C_m$ the horizontal cylinders, and by $E=E_1,\dots,E_n$ the cylinders in the direction of $E$. As usual we denote by $\ell_t(C_i)$ (resp. $\ell_t(E_j)$) the circumference of $C_i$ (resp. of $E_j$) in $(X_t,\omega_t)$. Let $u_i(t)$ be the width of $C_i$, and $v_j(t)$ be the width of $E_j$ in $(X_t,\omega_t)$. Remark that
$$
\ell_t(C_i)=e^t\ell(C_i), \quad u_i(t)=e^{-t}u_i, \quad \ell_t(E_j)= \sqrt{\cosh(2t)}\ell(E_j), \quad v_j(t)=\frac{v_j}{\sqrt{\cosh(2t)}}
$$
Since $(X,\omega)$ has area one we have
\begin{equation}\label{eq:area:1}
1=\sum u_i\ell(C_i)=\sum v_j\ell(E_j).
\end{equation}
Let $x_j$ (resp. $y_i$) be the intersection number of a core curve of $C_{0,t}$ and a core curve of $E_j$ (resp. of $C_i$). Since the circumference of $C_{0,t}$ is bounded by $1/K  < L_1$, we have
\begin{equation}\label{eq:C0:bound:1}
\sum y_iu_i(t)=e^{-t}\sum y_iu_i \leq \ell(C_{0,t}) \leq L_1 \Rightarrow \sum y_iu_i \leq e^tL_1.
\end{equation}

Since the width of $C_{0,t}$ is bounded below by $K$, we have $x_jK \leq \ell_t(E_j)=\sqrt{\cosh(2t)}\ell(E_j)$. Since $t\leq 0$,  it follows
\begin{equation}\label{eq:C0:bound:2}
  x_j \leq \frac{\sqrt{\cosh(2t)}}{K}\ell(E_j) \leq \frac{e^{-t}}{K}\ell(E_j).
\end{equation}

Let $(X',\omega'):= U\cdot (X,\omega)$, where $U=\left(\begin{smallmatrix} 1 & -1 \\ 0 & 1 \end{smallmatrix} \right)$. Let $\ell'(C_i)$ and $u'_i$ (resp. $\ell'(E_j)$ and $v'_j$) be the circumference and the width of $C_i$ (resp. of $E_j$) in $(X',\omega')$. Note that $C_i$ is horizontal, and $E_j$ is vertical in $(X',\omega')$. Thus, $\ell'(C_i))=\ell(C_i), \, u'_i=u_i$, and $\ell'(E_j)=\ell(E_j)/\sqrt{2}, \, v'_j=\sqrt{2}v_j$.

For any $s \in \R$, let $(X'_s,\omega'_s):=a_s\cdot(X',\omega')$. Let $\ell'_s(C_i)$ and $u'_i(s)$ (resp. $\ell'_s(E_j)$ and $v'_j(s)$) be the circumference and the width of $C_i$ (resp. of $E_j$) in $(X'_s,\omega'_s)$.

Let $x+\imath y$ be the period of the core curves of $C_{0,t}$ in $(X'_s,\omega'_s)$. From \eqref{eq:C0:bound:2} we get
\begin{equation}\label{eq:C0:x:bound}
|x|=\sum x_jv'_j(s)=e^s\sum x_jv'_j=e^s\frac{\sqrt{2}e^{-t}}{K}\sum\ell(E_j)v_j = \frac{\sqrt{2}e^{s-t}}{K}
\end{equation}
From \eqref{eq:C0:bound:1}, we get
\begin{equation}\label{eq:C0:y:bound}
|y|=\sum y_iu'_i(s) =e^{-s}\sum y_iu_i \leq e^{t-s}L_1.
\end{equation}

\noindent Thus for $s=t$, the circumference of $C_{0,t}$ in $(X'_s,\omega'_s)$ is at most $\sqrt{3}L_1< 2L_1$. Let ${C'}_{0,s}$ be a cylinder of width bounded below by $K$ in $(X'_s,\omega'_s)$. We have $\dist({C'}_{0,s},C_{0,t}) \leq M_1$ by Lemma~\ref{lm:LCD:t:diam:bound}, which means that $C_{0,t}$ is contained in the $M_1$-neighborhood of $\Lc^*_{C,E}(L_1)$.
\end{proof}

It follows immediately from the claim that $\Lc^*_{C,D}(t,L_1)$ is contained in the $2M_1$-neighborhood of $\Lc^*_{C,E}(L_1)$ if $t \leq 0$. By similar arguments, one can also show that $\Lc^*_{C,D}(t,L_1)$ is contained in the $2M_1$-neighborhood of $\Lc^*_{E,D}(L_1)$ if $t\geq 0$. Therefore, we can conclude that $\Lc^*_{C,D}(L_1) =\cup_{t\in \R}\Lc^*_{C,D}(t,L_1)$ is contained in the $2M_1$-neighborhood of $\Lc^*_{C,E}(L_1)\cup\Lc^*_{E,D}(L_1)$, which implies that $\hat{\Lc}_{C,D}$ is contained in the $3M_1$-neighborhood of $\Lc^*_{C,E}(L_1)\cup\Lc^*_{E,D}(L_1)$.
\end{proof}

\subsection{Proof of Theorem~\ref{thm:cper:hyperbolic}}
\begin{proof}
From Proposition~\ref{prop:LCD:dist:1}, and Proposition~\ref{prop:LCD:sim:tria}, we see that $\hat{\Cc}_{\rm cyl}(X,\omega,f)$ with the family of subgraphs $\hat{\Lc}_{C,D}$ satisfies the two conditions of Theorem~\ref{thm:hyp:crit:Ma-Sc} with $M=\max\{M_2,M_3\}$. Therefore,  $\hat{\Cc}_{\rm cyl}(X,\omega,f)$ is Gromov hyperbolic.
\end{proof}

\section{Quotient by affine automorphisms}\label{sec:quotient}
In this section we investigate the quotient of $\hat{\Cc}_{\rm cyl}(X,\omega,f)$ by the group $\mathrm{Aff}^+(X,\omega)$. Our main focus is the case where $(X,\omega)$ is a Veech surface, that is when $\SL(X,\omega)$ is a lattice in $\SL(2,\R)$. Throughout this section $(X,\omega)$ is a fixed translation surface in $\H(2)\sqcup \H(1,1)$, and $\hat{\Cc}_{\rm cyl}$ is the cylinder graph of $(X,\omega)$ with some marking map. We denote by $\Gp$ the quotient graph  $\hat{\Cc}_{\rm cyl}\diagup\Aff^+(X,\omega)$, and by $\Vp$ and $\Ep$ the sets of vertices and edges of $\Gp$ respectively. Notice that an edge may join a vertex to itself (we then have a loop), and there may be more than one edges with the same endpoints. We use the notations $|\Vp|$ and $|\Ep|$ to designate the cardinalities of $\Ep$ and $\Vp$. We will show

\begin{Theorem}\label{thm:quotient:fin:ver}
Let $(X, \omega)$ be a surface in $\H(2)\sqcup\H(1,1)$. Then $(X,\omega)$ a Veech surface if and only if $|\Vp|$ is  finite.
\end{Theorem}

Theorem~\ref{thm:quotient:fin:ver} does not mean, when $(X,\omega)$ is a Veech surface,  that the quotient graph $\Gp$ is a finite graph, as we have

\begin{Proposition}\label{prop:Veech:q:infin:H11}
If $(X,\omega)$ is Veech surface in $\H(2)$ then $\Gp$ is a finite graph,  that is $|\Vp|$ and $|\Ep|$ are both finite. There  exist Veech surfaces in $\H(1,1)$ such that $|\Vp| < \infty$ but $|\Ep| =\infty$.
\end{Proposition}

\subsection{Proof of Theorem~\ref{thm:quotient:fin:ver}}
Recall that the $\SL(2,\R)$-orbit of a Veech surface $(X,\omega)$ projects to an algebraic curve in $\Mm_2$ isomorphic to $\Xcal:=\mathbb{H}\diagup\SL(X,\omega)$, this curve is called a {\em Teichm\"uller curve}. The direction of any saddle connection on $X$ is periodic, that is $X$ is decomposed into finitely many cylinders in this  direction. Moreover, there is a parabolic element in $\SL(X,\omega)$ that fixes this direction. Thus each cylinder in $X$ corresponds to a cusp in $\Xcal$.

Let $\theta$ be a periodic direction for $X$. Let $C_1,\dots,C_k$ be the cylinders of $X$ in the direction $\theta$, and $T_i$ be the Dehn twist about the core curves of $C_i$. Let $\gamma$ be the generator of the parabolic subgroup of $\SL(X,\omega)$ that fixes $\theta$. Then there exist some integers $m_1,\dots,m_k$ such that $\gamma$ is the differential of an element of $\Aff^+(X,\omega)$ isotopic to $T_1^{m_1}\circ\dots\circ T_k^{m_k}$.

%\medskip
\subsubsection{Proof that $(X,\omega)$  is Veech implies that $\Vp$ is finite}\hfill
\begin{proof}
If $(X,\omega)\in \H(2)$, then $X$ has one or two cylinders in the direction $\theta$. In the first case, we have three more degenerate ones, and in the second case there is no degenerate cylinder. Thus the total  number of cylinders (degenerate or not) in a periodic direction is at most $4$. If $(X,\omega) \in \H(1,1)$, then by similar arguments, we see that $X$ has at most $5$ cylinders in the direction $\theta$. We have seen that $\theta$ corresponds to  a cusp of $\Xcal$. Since $\Xcal$ has finitely many cusps, it follows that $X$ has finitely many cylinders up to action of $\Aff^+(X,\omega)$. Therefore, $\Vp$ is finite.
\end{proof}

\subsubsection{Proof that $\Vp$ is finite implies $(X,\omega)$ is Veech.}\hfill

In what follows, by  an {\em embedded triangle} in $X$, we mean the image of a triangle $\Tb$ in the plane by a map $\varphi: \Tb \rightarrow X$ which is locally isometric, injective in the interior of $\Tb$, and sending the vertices of $\Tb$ to the singularities of $X$. Note that $\varphi$ maps a side of $\Tb$ to a concatenation of some saddle connections. By a slight abuse of notation, we will also denote by $\Tb$ the image of $\varphi$ in $X$. To show that $(X,\omega)$  is a Veech surface, we will use the following characterization of Veech surfaces by Smillie-Weiss~\cite{SmiWei10}.

\begin{Theorem}[Smillie-Weiss]\label{thm:V:no:small:triang}
$(X,\omega)$ is a Veech surface if and only if there exists an $\eps >0$ such that the area of any embedded triangle $\Tb$ in $X$ is at least $\eps$.
\end{Theorem}

We now assume that $|\Vp|$ is finite.   If $v$ is a vertex of $\hat{\Cc}_{\rm cyl}$, we denote by $\bar{v}$ its equivalence class in $\Vp$.   Clearly, the group $\Aff^+(X,\omega)$ preserves the areas of the cylinders in $X$. Therefore, each element of $\Vp$ has a well-defined area (a degenerate cylinder has zero area).  Since $\Vp$ is finite, we can write $\Vp=\{\bar{v}_1, \dots,\bar{v}_n\}$, where $n=|\Vp|$.  Using $\GL^+(2,\R)$, we can normalize so that $\Aa(X,\omega)=1$. Let $a_i=\Aa(v_i)$, and define
\begin{eqnarray*}
\Ap_1 & = & \{a_1,\dots,a_n\}, \\
\Ap_2 & = & \{|a_i-a_j|, a_i \neq a_j\}, \\
\Ap_3 & = & \{1-(a_i+a_j), \, a_i+a_j < 1\}, \\
\Ap_4 & = & \{1-(a_i+a_j+a_k), \, a_i+a_j+a_k < 1\}.
\end{eqnarray*}
Set $\eps=\min\{ \Ap_1\cup\Ap_2\cup\Ap_3\cup \Ap_4\}$. We will need the following lemma on slit tori.

\begin{Lemma}\label{lm:s:tor:n:per:area}
 Let $(\hat{X},\hat{\omega},\hat{s})$ be a slit torus.  By a {\em cylinder} in $\hat{X}$, we will mean a connected component  of $X$ that is cut out by a pair of parallel simple closed geodesics passing through the endpoints of $\hat{s}$.

 Assume that $\hat{s}$ is not parallel to  any simple closed geodesic of $\hat{X}$. Then there exists a sequence of cylinders $\{\hat{C}_k\}_{k\in \N}$ such that $\hat{C}_k$ is disjoint from the slit $\hat{s}$ for all $k\in \N$, and $\Aa(\hat{C}_k) \overset{k\ra +\infty}{\lra} \Aa(\hat{X})$.
\end{Lemma}
\begin{proof}
Using $\GL^+(2,\R)$, we can normalize so that $(\hat{X},\hat{\omega})=(\C/(\Z\oplus\imath\Z),dz)$. The slit $\hat{s}$ is then represented  by a segment $[0,(1+\imath\alpha)t]$, with $t \in (0,\infty)$ and $\alpha \in \R \setminus \Q$. In this setting, each simple closed geodesic $c$ of $\hat{X}$ corresponds to a vector $p+\imath q$ with $p,q \in \Z$ and $\gcd(p,q)=1$.  Let $c_1$ and $c_2$ be the simple geodesics parallel to $c$ which pass through the endpoints of $\hat{s}$. Note that $c_1,c_2$ cut $\hat{X}$ into two cylinders. By \cite[Lemma 4.1]{Ng11}, we know that one of the two cylinders is disjoint from $\hat{s}$ if and only if
$$
t|\det\left( \begin{smallmatrix} p & 1 \\ q & \alpha \end{smallmatrix} \right)|=t|p\alpha-q| < 1.
$$
\noindent Note that the quantity $t|p\alpha-q|$ is precisely the area of the cylinder that contains $\hat{s}$. Since $\alpha$ is an irrational number, one can find a sequence $\{(p_k,q_k)\}_{k\in \N}$ such that
$$
\gcd(p_k,q_k)=1, \, t|\alpha p_k-q_k| < 1, \text{ and } \lim_{k\ra \infty}|\alpha p_k-q_k|=0.
$$
For each $(p_k,q_k)$ in this sequence,  we have  a cylinder $\hat{C}_k$ in direction of $p_k+\imath q_k$ disjoint from $\hat{s}$ such that
$$
\Aa(\hat{C}_k) =1-t|\alpha p_k-q_k|.
$$
In particular, we have $\lim_{k\ra \infty}\Aa(\hat{C}_k)=1$, which proves the lemma.
\end{proof}

As a consequence of this lemma, we get
\begin{Corollary}\label{cor:split:periodic}
 Let $(s_1,s_2)$ be a pair of homologous saddle connections in $X$ that are exchanged by the hyperelliptic involution $\hinv$. If one of the connected components cut out by $(s_1,s_2)$ is a slit torus, then the direction of $s_1,s_2$ is periodic.
\end{Corollary}
\begin{proof}
 If $(X,\omega) \in \H(2)$ then $X$ is  decomposed by $(s_1,s_2)$  into a simple cylinder and a slit torus, if $(X,\omega) \in \H(1,1)$ then $X$ is decomposed into two slit tori. Thus, it suffices to show that $s_i$ is  parallel to a closed geodesic in each slit torus. If this is not the case, then by Lemma~\ref{lm:s:tor:n:per:area}, we can find in this slit torus a sequence of cylinders disjoint from the slit whose area converges to the area of the torus. Note that such cylinders are also cylinders of $X$. Thus their areas belong to $\Ap_1$. Since $\Ap_1$ is finite, it cannot contain a non-constant converging sequence. Therefore, we can conclude that the direction of $(s_1,s_2)$ is periodic.
\end{proof}

Let $\Tb$ be an embedded triangle in $X$. We will show that $\Aa(\Tb) > \eps/2$.  We first remark that it suffices to consider the case where each side of $\Tb$ is a saddle connection, since otherwise there is another embedded triangle contained in $\Tb$ with this property.  Let $\hinv$ denote the hyperelliptic involution of $X$, and $\Tb'=\hinv(\Tb)$. Let  $s_1,s_2,s_3$ be the sides of $\Tb$ and $s'_i$ be the image of $s_i$ by $\hinv$. The proof is naturally splits into $2$ cases depending on the stratum of $(X,\omega)$.

\medskip

\begin{proof}[Case $(X,\omega)\in \H(2)$:] we need to consider the following two situations:

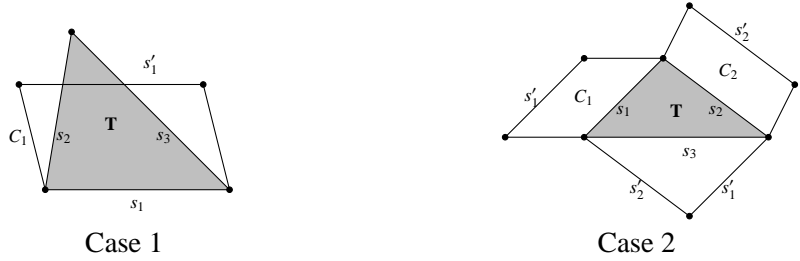
\begin{figure}[htb]
\begin{minipage}[t]{0.45\linewidth}
 \centering
 \begin{tikzpicture}[scale=0.35]
 \fill[gray!50] (1,0) -- (2,6) -- (8,0) -- cycle;
\draw (0,4) -- (1,0) -- (8,0) -- (7,4) -- cycle;
\draw (1,0) -- (2,6) -- (8,0);
\foreach \x in {(0,4),(1,0),(2,6),(7,4),(8,0)} \filldraw[fill=black] \x circle (3pt);

\draw (4.5,0) node[below] {\tiny $s_1$} (1.7,2) node {\tiny $s_2$} (5.5,2) node {\tiny $s_3$} (5,4) node[above] {\tiny $s'_1$} (3.5,2.5) node {\tiny $\Tb$} (0,2) node {\tiny $C_1$};

\draw (4,-2) node {Case 1};
 \end{tikzpicture}
\end{minipage}
\begin{minipage}[t]{0.45\linewidth}
 \centering
 \begin{tikzpicture}[scale=0.35]
 \fill[gray!50] (3,3) --(6,6) --(10,3) -- cycle;
  \draw (0,3) -- (3,3) -- (7,0) -- (10,3) -- (11,5) -- (7,8) --(6,6) -- (3,6) -- cycle;
  \draw (3,3) --(6,6) -- (10,3) -- cycle;
  \foreach \x in {(0,3),(3,6),(3,3),(6,6),( 7,8),(7,0),(10,3),(11,5)} \filldraw[fill=black] \x circle (3pt);

  \draw (4.5,4) node {\tiny $s_1$} (8,4) node {\tiny $s_2$} (7,3) node[below] {\tiny $s_3$} (1,4.5) node {\tiny $s'_1$} (8.5,1) node {\tiny $s'_1$} (9,7) node {\tiny $s'_2$} (5,1) node {\tiny $s'_2$} (3,4.5) node {\tiny $C_1$} (8.5,5.5) node {\tiny $C_2$} (6.5,4) node {\tiny $\Tb$};

  \draw (5,-1) node {Case 2};
 \end{tikzpicture}
\end{minipage}

\caption{Embedded triangles in a surface in $\H(2)$.}
\label{fig:embed:triang:H2}
\end{figure}

\begin{itemize}
\item[$\bullet$] \underline{Case 1:} none of the sides of $\Tb$ is invariant by $\hinv$.  From Lemma~\ref{lm:sc:no:inv:types}, $s_i$ and $s'_i$ bound a simple cylinder denoted by $C_i$. Let $h_i$ be length of the perpendicular segment from  the opposite  vertex  of $s_i$  in $\Tb$ to $s_i$. If $\inter(\Tb)\cap\inter(C_1) \neq \vide$, then both $s_2$ and $s_3$ cross $C_1$ entirely, which implies that the width of $C_1$ is is at most $h_1$ (see Figure~\ref{fig:embed:triang:H2} left). It follows that $\Aa(\Tb) \geq 1/2\Aa(C_1) >\min\Ap_1/2$. The same arguments apply in the cases $\inter(\Tb)$ intersects $\inter(C_2)$ or $\inter(C_3)$. 
If $\inter(\Tb)$ is disjoint from $\inter(C_i), \, i=1,2,3$, then we have three disjoint cylinders in $X$ (if $\inter(C_i)\cap\inter(C_j)\neq \varnothing$ then $s_i$ must cross $C_j$ entirely hence $\inter(\Tb)\cap\inter(C_j) \neq \vide$). Since $(X,\omega) \in \H(2)$, this situation cannot occur (see Theorem~\ref{thm:exist:mult:curv}). Hence, we can conclude that $\Aa(\Tb) \geq \eps/2$ in this case.

\item[$\bullet$]\underline{Case 2:} one of the sides of $\Tb$ is invariant by $\hinv$. In this case, the union of $\Tb$ and its image by $\hinv$ is an embedded parallelogram (see Lemma~\ref{lm:embd:par}). This means that there is a parallelogram $\Pb$ in the plane such that $\Tb$ is one of the two triangles cut out by a diagonal of $\Pb$, and there is a map $\varphi: \Pb \ra X$ locally isometric, injective in $\inter(\Tb)$, mapping the vertices of $\Pb$ to the singularity of $X$. Remark that all the sides of $\Tb$ cannot be invariant by $\hinv$ because this would imply that $X=\varphi(\Pb)$ is a torus.  If there are two sides of $\Tb$ that are invariant by $\hinv$, then $\varphi(\Pb)$ is a simple cylinder in $X$, hence $\Aa(\Tb) \geq \min\Ap_1/2$. If there is only one side invariant by $\hinv$, then the complement of $\varphi(\Pb)$ is the union of two disjoint simple cylinders $C_1,C_2$ (see Figure~\ref{fig:embed:triang:H2} right), which implies 
$\Aa(\Pb)=1-(\Aa(C_1)+\Aa(C_2))$. Therefore, we have $\Aa(\Tb) > \min\Ap_3/2 \geq \eps/2$.  This completes the  of Theorem~\ref{thm:quotient:fin:ver} for the case  $(X,\omega) \in \H(2)$.
\end{itemize}
\end{proof}

\begin{proof}[Case $(X,\omega) \in \H(1,1)$:] We consider the following situations:
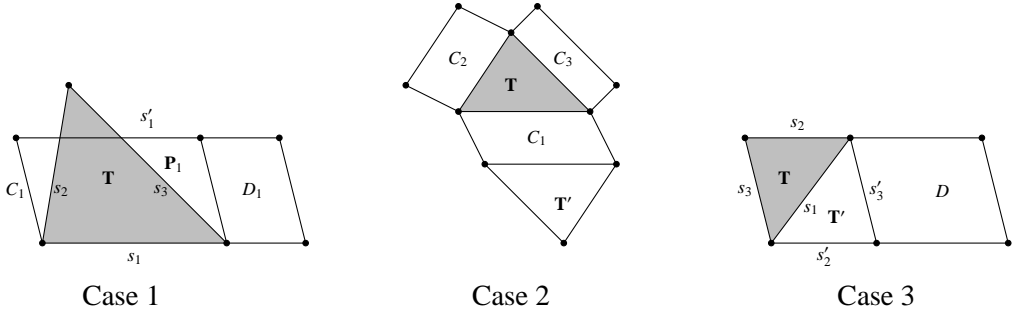
\begin{figure}[htb]
\begin{minipage}[t]{0.3\linewidth}
 \centering
 \begin{tikzpicture}[scale=0.35]
 \fill[gray!50] (1,0) -- (2,6) -- (8,0) -- cycle;
\draw (0,4) -- (1,0) -- (11,0) -- (10,4) -- cycle;
\draw (1,0) -- (2,6) -- (8,0) -- (7,4);
\foreach \x in {(0,4),(1,0),(2,6),(7,4),(8,0),(10,4),(11,0)} \filldraw[fill=black] \x circle (3pt);

\draw (4.5,0) node[below] {\tiny $s_1$} (1.7,2) node {\tiny $s_2$} (5.5,2) node {\tiny $s_3$} (5,4) node[above] {\tiny $s'_1$} (3.5,2.5) node {\tiny $\Tb$} (0,2) node {\tiny $C_1$} (9,2) node {\tiny $D_1$} (6,3) node {\tiny $\Pb_1$};

\draw (4,-2) node {Case 1};
 \end{tikzpicture}
\end{minipage}
\begin{minipage}[t]{0.3\linewidth}
 \centering
 \begin{tikzpicture}[scale=0.35]
 \fill[gray!50] (2,5) -- (4,8) -- (7,5) --  cycle;
 \draw (0,6) -- (2,5) -- (3,3) -- (6,0) -- (8,3) -- (7,5) -- (8,6) -- (5,9) -- (4,8) -- (2,9) -- cycle;
 \draw (2,5) -- (4,8) -- (7,5) -- (2,5) (3,3) -- (8,3);
 \foreach \x in {(0,6),(2,9),(2,5),(3,3),(4,8),(5,9),(6,0),(7,5),(8,6),(8,3)} \filldraw[fill=black] \x circle (3pt);
 \draw (5,4) node {\tiny $C_1$} (2,7) node {\tiny $C_2$}  (6,7) node {\tiny $C_3$} (4,6) node {\tiny $\Tb$} (6,1.5) node {\tiny $\Tb'$};
 \draw (4,-2) node {Case 2};
 \end{tikzpicture}
\end{minipage}
\begin{minipage}[t]{0.3\linewidth}
 \centering
 \begin{tikzpicture}[scale=0.35]
  \fill[gray!50] (0,4) -- (1,0) -- (4,4) -- cycle;
  \draw (0,4) -- (1,0) -- (10,0) -- (9,4) -- cycle;
  \draw (1,0) -- (4,4) -- (5,0);
  \foreach \x in {(0,4),(1,0),(4,4),(5,0),(9,4),(10,0)} \filldraw[fill=black] \x circle (3pt);
  \draw (2.5,1.3) node {\tiny $s_1$} (2,4.5) node {\tiny $s_2$} (0,2) node {\tiny $s_3$} (3,-0.5) node {\tiny $s'_2$} (5,2) node {\tiny $s'_3$} (7.5,2) node {\tiny $D$} (1.5,2.5) node {\tiny $\Tb$} (3.5,1) node {\tiny $\Tb'$};

  \draw (5,-2) node {Case 3};
 \end{tikzpicture}
\end{minipage}

\caption{Embedded triangles in a  surface in $\H(1,1)$}
\label{fig:embed:triang:H11}
\end{figure}

\begin{itemize}
 \item[$\bullet$] \underline{Case 1:}  there exists $i$ such that  $s'_i$ intersects $\inter(\Tb)$. Note that we must have $s'_i\neq s_i$.  Let us assume that $i=1$. Recall that $s_1$ and $s'_1$ either bound a simple cylinder, or decompose $X$ into two tori. In the first case, the same argument as in the case $(X,\omega)\in \H(2)$ shows that $\Aa(\Tb)\geq \min\Ap_1/2$. For the second case, observe that the intersection of $\Tb$ with one of the slit tori consists of a domain bounded by $s_1$ and  some subsegments of $s_2,s_3$ and $s'_1$ (see Figure~\ref{fig:embed:triang:H11}). Let $(X_1,\omega_1, \tilde{s}_1)$ denote this slit torus.

 We can assume that $s_1$ is horizontal. By Corollary~\ref{cor:split:periodic} we know that the horizontal direction is periodic for $X_1$, thus $X_1$ is the closure of a horizontal cylinder $C_1$.
 Remark that $X_1$ contains a transverse simple cylinder $D_1$ disjoint from $s_1\cup s'_1$, whose core curves cross  $C_1$ once. The complement of $D_1$ in $X_1$ is an embedded parallelogram $\Pb_1$ bounded by $s_1,s'_1$ and the boundary of $D_1$.  Clearly, we have $\Aa(\Tb) \geq \Aa(\Pb_1)/2$. By definition, we have
 $$
 \Aa(\Pb_1) =\Aa(C_1)-\Aa(D_1) \geq \min\Ap_2.
 $$
 Thus we have $\Aa(\Tb) \geq \eps/2$.
 \item[$\bullet$] \underline{Case 2:} none of $s'_i$ intersects $\inter(\Tb)$, and   $s'_i \neq s_i, \, i=1,2,3$.  It is not difficult to show that that this case only happens when $s_i$ and $s'_i$ bound a simple cylinder $C_i$ disjoint from $\inter(\Tb)\cup\inter(\Tb')$. Therefore, $X$ is decomposed into the union  of three cylinders  $C_1,C_2,C_3$, and $\Tb\cup \Tb'$ (see Figure~\ref{fig:embed:triang:H11}). Thus in this case, we have
%  Assume that $(s_1,s'_1)$ decompose $X$ into two slit tori, then $\Tb$ is contained in one of the tori since $\inter(\Tb)$ does not intersect $s_1\cup s'_1$.  Since both slit tori are invariant by $\hinv$, both $\Tb$ and $\hinv(\Tb)$ are contained in the same slit torus. But by assumption,  $\Tb$ and $\hinv(\Tb)$ do not have any common side, hence they cannot be contained in the same slit torus, and we get a contradiction.
%  

 $$
 \Aa(\Tb)=\frac{1}{2}\left(1-(\Aa(C_1)+\Aa(C_2)+\Aa(C_3))\right) \geq \min \Ap_4/2 \geq \eps/2.
 $$
\item[$\bullet$] \underline{Case 3:} none of $s'_i$ intersects $\inter(\Tb)$ and one of $s_1,s_2,s_3$ is invariant by $\hinv$. Let us assume that $s'_1=s_1$. It follows that $\Tb\cup\Tb'$ is an embedded parallelogram $\Pb$. If both $(s_2,s'_2)$ and $(s_3,s'_3)$ are the boundaries of some simple cylinders $C_2$ and $C_3$ respectively, then $C_2$ and $C_3$ are disjoint, and $C_2\cup C_3$ is disjoint from $\Pb$. By construction we must have $X=\ol{\Pb}\cup\ol{C}_2\cup \ol{C}_3$, which is impossible since $(X,\omega) \in \H(1,1)$. Therefore, we can assume that $(s_2,s'_2)$ decompose $X$ into two slit tori. Let $X_1$  be the slit torus that contains $\Pb$. By Corollary~\ref{cor:split:periodic}, we know that the direction of $(s_2,s'_2)$ is periodic, which means that $X_1$ is the closure of a cylinder $C$. Observe that the complement of $\Pb$ in $X_1$ must be a cylinder $D$ bounded by $(s_3,s'_3)$ (see Figure~\ref{fig:embed:triang:H11}). Therefore
$$
 \Aa(\Tb)=\frac{1}{2}\Aa(\Pb)=\frac{1}{2}(\Aa(C)-\Aa(D)) \geq \frac{1}{2}\min\Ap_2 \geq \eps/2.
$$

 \item[$\bullet$] \underline{Case 4:} none of $s'_i$ intersects $\inter(\Tb)$ and two of $s_1,s_2,s_3$ are invariant by $\hinv$. In this case $\Tb\cup\Tb'$ is a simple cylinder. Therefore $\Aa(\Tb) \geq \min\Ap_1/2 \geq \eps/2$.
\end{itemize}

Since in all cases we have $\Aa(\Tb) \geq \eps/2$, it follows from Theorem~\ref{thm:V:no:small:triang} that $(X,\omega)$ is a Veech surface.
\end{proof}

\subsection{Proof of Proposition~\ref{prop:Veech:q:infin:H11}}\hfill
\subsubsection{Case $(X,\omega) \in \H(2)$} \hfill
\begin{proof}
We have shown that $\Vp$ is finite, it remains to show that $\Ep$ is also finite. Let $v$ be a vertex of $\hat{\Cc}_{\rm cyl}$, and $C$ be the corresponding cylinder in $X$. We denote by $\bar{v}$ the equivalence class of $v$ in $\Gp$. Using $\SL(2,\R)$, we can suppose that $C$ is horizontal.

If $C$ is a non-degenerate cylinder, then we have three cases: (a) $C$ is the unique horizontal cylinder, (b) $X$ has two horizontal cylinders and $C$ is not simple, (c) $C$ is a simple cylinder. In case (a), there are $3$ edges in $\hat{\Cc}_{\rm cyl}$ that have $v$ as an endpoint, those edges connect $v$ to three degenerate cylinders contained in the boundary of $C$. In case (b), there is only one edge in $\hat{\Cc}_{\rm cyl}$ having $v$ as an endpoint, this edge connects $C$ to the other horizontal simple cylinder. Thus in case (a) and case (b), there are only finitely many edges having $\bar{v}$ as an endpoint.

Assume now that we are in case (c). Let $D$ be the other horizontal cylinder of $X$. Observe that the closure of $D$ is a slit torus $(X',\omega',s')$ where $s'$ corresponds to the boundary of $C$.  Let $d$ be a core curve of $D$, and $e$  be a simple closed geodesic in $X'$ disjoint from the slit $s$' and crossing $d$ once. We consider $\{d,e\}$ as a basis of $H_1(X',\Z)$. If $C'$ is a cylinder in $X$ disjoint from $C$, then $C'$ must be entirely contained in $\ol{D}$. Thus the core curves of $C'$ are determined by a unique element of $H_1(X',\Z)$, and we can write $C'= md+ne$ with $m,n \in \Z$.

By assumption, a core curve $c'$ of $C'$ cannot cross the slit $s'$. The necessary and sufficient condition for this is that $|\omega'(c')\wedge\omega'(s')| \leq \Aa(X')=\Aa(D)$ (see~\cite[Lem. 4.1]{Ng11}). But $|\omega'(c')\wedge\omega'(s')|=|n||\omega'(e)\wedge\omega'(s')|$. Thus we can conclude that $|n|$ is bounded by  some constant $n_0$.

We have seen that $\Aff^+(X,\omega)$ contains an element $\phi=T_1^{m_1}\circ T_2^{m_2}$, where $T_1$ and $T_2$ are the Dehn twists about the core curves of $C$ and $D$ respectively. Observe that $\phi$ fixes the vertices of $\hat{\Cc}_{\rm cyl}$ corresponding to $C$ and $D$. The action of $\phi$ on the curves contained in $\ol{D}$ is given by
$$
\phi(md+ne)=(m \pm m_2n) d +ne.
$$
\noindent Thus up to action of $\{\phi^k\}_{k\in \Z}$,  any cylinder $C'$ contained in $\ol{D}$ belongs to the equivalence class of  a cylinder  $C''$ also contained in $\ol{D}$ whose core curves are represented by $md+nc$ with $|n| \leq |n_0|$ and $|m| \leq |m_2n|\leq |m_2||n_0|$. We can then conclude that there are finitely many edges in $\Ep$ which contains $\bar{v}$ as an endpoint.  

It remains to consider the case $C$ is degenerate. In this case $X$ has a unique non-degenerate cylinder in the horizontal direction, which contains  $C$ in its boundary. Remark that the complement of $C$ in $X$ can be isometrically identified with a flat torus with an embedded geodesic segment removed. Therefore, the arguments above also hold in this case. Since we have proved that  the set of vertices of $\Gp$ is finite, it follows that the set of edges of $\Gp$  is also  finite.
\end{proof}

\subsubsection{Case $(X,\omega)\in \H(1,1)$}\hfill
\begin{proof}
Let $(X,\omega)$ the surface constructed from $6$ squares as shown in Figure~\ref{fig:6square:H11}. This surface has $3$ horizontal cylinders denoted by  $C_1,C_2,C_3$, where $C_i$ is the cylinder with $i$ squares. It has two vertical cylinders denoted by $D_1$ and $D_2$, where the core curves of $D_1$ cross $C_1$ and $C_3$. Let $v$ be the vertex of $\hat{\Cc}_{\rm cyl}$ corresponding to $C_1$, and $w$ be the vertex corresponding to $C_2$. The fact that $\Gp$ has finitely many  vertices follows from Theorem~\ref{thm:quotient:fin:ver}. We will show that $\Gp$ has infinitely many edges.

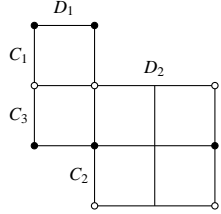
\begin{figure}[htb]
\begin{tikzpicture}[scale=0.4]
 \draw (0,6) --(0,2) -- (2,2) -- (2,0) -- (6,0) -- (6,4) -- (2,4) -- (2,6) -- cycle;
 \draw (0,4) -- (2,4) (2,2) -- (6,2);
 \draw (2,4) -- (2,2) (4,4) --(4,0);

 \foreach \x in {(0,6),(0,2),(2,6),(2,2),(6,2)} \filldraw[fill=black] \x circle (3pt);
 \foreach \x in {(0,4),(2,4),(2,0),(6,4),(6,0)} \filldraw[fill=white] \x circle (3pt);

 \draw (-0.5,5)  node {\tiny $C_1$} (-0.5,3) node {\tiny $C_3$} (1.5,1) node {\tiny $C_2$} (1,6) node[above] {\tiny $D_1$} (4,4) node[above] {\tiny $D_2$} ;
 \end{tikzpicture}

\caption{Example of square-tiled surface in $\H(1,1)$}
\label{fig:6square:H11}
\end{figure}

\noindent Given  a cylinder $C$  on $X$, we denote by $T_C$ the Dehn twist about the core curves of $C$.  Observe that $f=T^6_{C_1}\circ T^3_{C_2}\circ T^2_{C_3}$  and $g=T_{D_1}\circ T^2_{D_2}$ are two elements of $\Aff^+(X,\omega)$, their differentials are $\left(\begin{smallmatrix} 1 & 6 \\ 0 & 1\end{smallmatrix}\right)$ and $\left(\begin{smallmatrix} 1 & 0 \\ 2 & 1\end{smallmatrix}\right)$ respectively. If $h$ is an element of $\Aff^+(X,\omega)$ that preserves the horizontal direction, then $h$ must map a horizontal cylinder to a horizontal cylinder. Since $C_1,C_2,C_3$ have different circumferences, $h$ must preserve each of them, which implies that $h=f^k, \, k \in \Z$. We derive in particular that there is no affine homeomorphism that maps $C_2$ to $C_1$.

For any $n \in \N$, let $E_n$ be the image of $C_2$ by $g^n$. Remark that $E_n=T^{2n}_{D_2}(C_2)$, hence $E_n$ is contained in the closure $\ol{D}_2$ of $D_2$. In particular, $E_n$ is disjoint from $C_1$. Thus, there is an edge $\mathbf{e}_n$ in $\hat{\Cc}_{\rm cyl}$ connecting $v$ to the vertex $w_n$ corresponding to $E_n$. By definition, all the vertices $w_n$ belong to the equivalence class $\bar{w}$ of $w$ in $\Gp$. We will show that the edges $\{\mathbf{e}_n\}_{n\in \N}$ are all distinct up to action of $\Aff^+(X,\omega)$, which means that there are infinitely many edges in $\Ep$ between  $\bar{v}$ and $\bar{w}$.

Assume that there is an affine automorphism  $h\in \Aff^+(X,\omega)$ such that $h(\mathbf{e}_{n_1})=\mathbf{e}_{n_2}$, for some $n_1, n_2 \in \N$.  If $h(w_{n_i})=v$, then there is an element of $\Aff^+(X,\omega)$ that sends $w$ to $v$, or equivalently  $C_2$ to $C_1$. But we have already seen that such an element does not exist, thus this case cannot occur. Therefore, we must have $h(v)=v$ and $h(w_{n_1})=w_{n_2}$. Since any element of $\Aff^+(X,\omega)$ preserving $C_1$ belongs to the subgroup generated by $f$, we derive that $h$ also preserves $C_2$ and $C_3$. Observe that a core curve of $E_{n_i}$ crosses $C_2$ $2n_i$ times. Therefore, if $n_1\neq n_2$, then $h$ cannot exist. We can then conclude  that the projections of  all the edges $\mathbf{e}_n$ are distinct in $\Gp$, which proves the proposition.
\end{proof}

\section{Quotient graphs and McMullen's prototypes}\label{sec:prototype}
By the works of McMullen~\cite{McM07, McM_spin}, we know that closed $\SL(2,\R)$-orbits in $\H(2)$ are indexed by the discriminant $D$, that is a natural number $D \in\N$ such that $D \equiv 0,1 \mod 4$, together with the parity of the spin structure when $D \equiv 1 \mod 8$. 

\begin{figure}[htb]
\begin{minipage}[t]{0.3\linewidth}
\centering
\begin{tikzpicture}[scale=0.25, inner sep=0.2mm, vertex/.style={circle, draw=black, fill=blue!20, minimum size=1mm},
head/.style={rectangle, draw=black, minimum size=5mm},
>= stealth]
 \draw(0,8) -- (0,0) -- (5,0) -- (5,5) -- (3,5) -- (3,8) -- cycle;
 \draw (0,5) -- (3,5);
 \foreach \x in {(0,8),(0,5),(0,0),(3,8),(3,5),(3,0),(5,5),(5,0)}\filldraw[fill=black] \x circle (3pt);

 \node (title) at (5,11) [head] {$D=5$};
 \node (above) at (10,7) [vertex] {\tiny $C_1$};
 \node (center) at (10,4) [vertex] {\tiny $C_2$};

 \draw (above) to (center);
 \draw (center) .. controls (8,0) and (12,0) .. (center);

 \draw[<->] (-0.5,8) -- (-0.5,5);
 \draw[<->] (-0.5,5) -- (-0.5,0);
 \draw[<->] (0,-0.5) -- (5,-0.5);

 \draw(-0.5,6.5) node[left] {\tiny $\frac{\sqrt{5}-1}{2}$} (-0.5,2.5) node[left] {\tiny $1$} (2.5,-0.5) node[below] {\tiny $1$};
 \draw (1.5,6.5) node {\tiny $C_2$} (3,2.5) node {\tiny $C_1$};
 \draw  (5,-2.5) node {\tiny $(0,1,1,-1),\lbd=\frac{\sqrt{5}-1}{2}$};
\end{tikzpicture}
\end{minipage}
\begin{minipage}[t]{0.65\linewidth}
\centering
\begin{tikzpicture}[scale=0.25, inner sep=0.2mm, vertex/.style={circle, draw=black, fill=blue!20, minimum size=1mm},
head/.style={rectangle, draw=black, minimum size=5mm},
>= stealth]

\node (title) at (9,10) [head] {$D=8$};
\draw (0,7) -- (0,0) -- (6,0) -- (6,3) -- (4,3) -- (4,7) -- cycle;
\draw (0,3) -- (4,3);

\draw[<->] (-0.5, 7) -- (-0.5,3);
\draw[<->] (-0.5,3) -- (-0.5,0);
\draw[<->] (0,-0.5) -- (6,-0.5);

\foreach \x in {(0,7),(0,3),(0,0), (4,7),(4,3),(4,0),(6,3),(6,0)} \filldraw[fill=black] \x circle (3pt);

\draw (-0.5,5) node[left] {\tiny $\sqrt{2}$} (-0.5,1.5) node[left] {\tiny $1$} (3,-0.5) node[below] {\tiny $2$};

\draw (3,1.5) node {\tiny $C_1$} (2,5) node {\tiny $C_2$};

\draw (3,-2.5) node {\tiny $(0,2,1,0)$};
%*****************************************************************************
\draw (12,5) -- (12,7) -- (10,7) -- (10,0) -- (15,0) -- (15,5) -- (10,5);

\foreach \x in {(10,7),(10,5),(10,0), (12,7),(12,5),(12,0),(15,5),(15,0)} \filldraw[fill=black] \x circle (3pt);
\draw[<->] (9.5,7) -- (9.5,5);
\draw[<->] (9.5,5) -- (9.5,0);
\draw[<->] (10,-0.5) -- (15,-0.5);

\draw (9.5,6) node[left] {\tiny $\sqrt{2}-1$} (9.5,3) node[left] {\tiny $1$} (12.5,-0.5) node[below] {\tiny $1$};
\draw (12.5,2.5) node {\tiny $C_4$} (11,6) node {\tiny $C_3$};
\draw (12.5,-2.5) node {\tiny $(0,1,1,-2)$};
%*******************************************************************************
\node (top) at (19,9) [vertex] {\tiny $C_1$};
\node (center1) at (19,6) [vertex] {\tiny $C_2$};
\node (center2) at (19,2) [vertex] {\tiny $C_3$};
\node  (bottom) at (19,-1) [vertex] {\tiny $C_4$};

\draw (top) to (center1);
\draw (center1) to (center2);
\draw (center2) to (bottom);
\draw (center2) .. controls (23,0) and (23,4) .. (center2);
\end{tikzpicture}
\end{minipage}

\bigskip

\begin{minipage}[t]{0.75\linewidth}
\centering
\begin{tikzpicture}[scale=0.25, inner sep=0.2mm, vertex/.style={circle, draw=black, fill=blue!20, minimum size=1mm},
head/.style={rectangle, draw=black, minimum size=5mm},
>= stealth]
\node (title) at (12,9) [head] {$D=9$};

\draw (3,3) -- (3,6) -- (0,6) -- (0,0) -- (6,0) -- (6,3) -- (0,3);
\foreach \x in {(0,6), (0,3), (0,0), (3,6), (3,3), (3,0), (6,3), (6,0)} \filldraw[fill=black] \x circle (3pt);
\draw[<->] (-0.5,6) -- (-0.5,3);
\draw[<->] (-0.5,3) -- (-0.5,0);
\draw[<->] (0,-0.5) -- (6,-0.5);
\draw (-0.5,4.5) node[left] {\tiny $1$} (-0.5,1.5) node[left] {\tiny $1$} (3,-0.5) node[below] {\tiny $2$};
\draw (3,1.5) node {\tiny $C_1$} (1.5,4.5) node {\tiny $C_2$};
\draw (3,-2.5) node {\tiny $(0,2,1,-1)$};
%*****************************************************************************

\draw (10,3) -- (10,0) -- (19,0) -- (19,3) -- cycle;
\foreach \x in {(10,3), (10,0), (13,3), (13,0), (16,3), (16,0), (19,3), (19,0)} \filldraw[fill=black] \x circle (3pt);
\draw[<->] (9.5,3) -- (9.5,0);
\draw[<->] (10,-0.5) -- (19,-0.5);
\draw (9.5,1.5) node[left] {\tiny $1$} (14.5,-0.5) node[below] {\tiny $3$};

\draw[<->, dashed, >=angle 45] (10,3) .. controls (11,4) and (15,4) .. (16,3);
\draw (13,4) node[above] {\tiny $C_3$} (16,1.5) node {\tiny $C_4$};

\draw (11.5,3) +(0,-0.2) -- +(0,0.2);
\draw (14.5,3) +(-0.1,-0.2) -- +(-0.1,0.2) +(0.1,-0.2) -- +(0.1,0.2);
\draw (17.5,3) +(-0.2,-0.2) -- +(-0.2,0.2) +(0,-0.2) -- +(0,0.2) +(0.2,-0.2) -- +(0.2,0.2);

\draw (14.5,0) +(0,-0.2) -- +(0,0.2);
\draw (11.5,0) +(-0.1,-0.2) -- +(-0.1,0.2) +(0.1,-0.2) -- +(0.1,0.2);
\draw (17.5,0) +(-0.2,-0.2) -- +(-0.2,0.2) +(0,-0.2) -- +(0,0.2) +(0.2,-0.2) -- +(0.2,0.2);

% \draw (11.5,3) node[below] {\tiny $s_1$} (14.5,3) node[below] {\tiny $s_2$} (17.5,3) node[below] {\tiny $s_3$};
% \draw (11.5,0) node[above] {\tiny $s_2$} (14.5,0) node[above] {\tiny $s_1$} (17.5,0) node[above] {\tiny $s_3$};
\draw (14.5,-2.5) node {\tiny $C_3$ is degenerate};
%******************************************************************************

\node (top) at (23,9) [vertex] {\tiny $C_1$};
\node (center1) at (23,6) [vertex] {\tiny $C_2$};
\node (center2) at (23,3) [vertex] {\tiny $C_3$};
\node (bottom) at (23,0) [vertex] {\tiny $C_4$};

\draw (top) to (center1);
\draw (center1) to (center2);
\draw (center2) to (bottom);
\draw (center1) .. controls  (27,4) and (27,8) .. (center1);
\end{tikzpicture}
\end{minipage}

\caption{Examples of $\Gp$ for some small values of $D$. For each two-cylinder decomposition, we provide the corresponding prototype $(a,b,c,e)$. A loop at some vertex represents a Butterfly move that does not change the prototype.}
 \label{fig:fin:gr:ex:H2}
\end{figure}
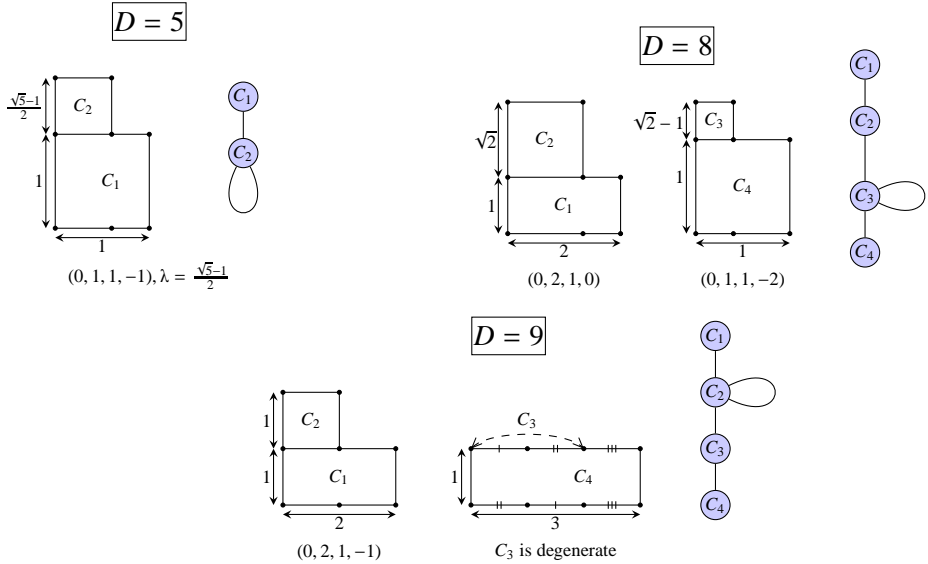

Let $(X,\omega)$ be an eigenform in $\Ecal_D\cap\H(2)$ for some fixed $D$. Following McMullen~\cite{McM_spin}, every two-cylinder decomposition of $X$ is encoded by a quadruple of integers $(a,b,c,e) \in \Z^4$ called {\em prototype} satisfying the following conditions
$$
(\mathcal{P}_D) \quad \left\{
\begin{array}{lll}
 b >0, & c>0, & \gcd(b,c) >a \geq 0,\\
 D=e^2+4bc, & b>c+e, & \gcd(a,b,c,e)=1.
\end{array}
\right.
$$
Set $\lbd=(e+\sqrt{D})/2$. Up to action of $\GL^+(2,\R)$, the decomposition of $X$ consists of two horizontal cylinders. The first one is simple and represented by a square of size $\lbd$. The other one  is non-simple and represented by a parallelogram constructed from the vectors $(b,0)$ and $(a,c)$. Note that we always have $b >\lbd$.

The quotient graph $\Gp$ turns out to be closely related to the set of McMullen's prototypes. Namely, each prototype corresponds to  a cluster of two vertices of $\Gp$ which represent the cylinders in the corresponding cylinder decomposition. Let $C_1,C_2$ be the cylinders in this decomposition, where $C_1$ is the simple one. Then the vertex corresponding to $C_2$ is only adjacent to the one corresponding to $C_1$ in $\Gp$. This is because any other cylinder of $X$ must cross $C_2$.

On the other hand, if there is an edge in $\Gp$ between two vertices representing two simple cylinders which are not parallel, then the two cylinder decompositions are related by a ``Butterfly move'' (see \cite[Sect. 7]{McM_spin} for the precise definitions). In other words, $\Gp$ can be viewed as a geometric object representing $\mathcal{P}_D$, each prototype is represented by two  vertices connected by an edge, and all the other edges of $\Gp$ represent Butterfly moves.

There is nevertheless a slight difference between the two notions.  The set $\mathcal{P}_D$ only parametrizes two-cylinder decompositions of $X$, while in  $\Gp$ we also have one-cylinder decompositions.  If $\sqrt{D}\not\in \N$, then any cylinder in $X$ is contained in a two-cylinder decomposition. Thus, the set of prototypes exhausts all the equivalence classes of cylinders in $X$ (hence it provides the complete list of cusps of the corresponding Teichm\"uller curve). But when $D$ is a square ({\em e.g.} $D=9$),  we need to take into account one-cylinder decompositions as well as degenerate cylinders. In Figure~\ref{fig:fin:gr:ex:H2}, we draw the quotient cylinder graphs  of surfaces corresponding to some small values of $D$.

%******************************************************************************************************************************
%******************************************************************************************************************************
%******************************************************************************************************************************

\appendix
\section{Triangulations}\label{sec:triangulation}
In this section we construct triangulations of $(X,\omega)$ that are invariant by the hyperelliptic involution. The aim of these triangulations is to provide a ``preferred'' way to  represent $(X,\omega)$ as a polygon in $\R^2$ when we have a horizontal simple cylinder on $X$. The results of this section are certainly not new and known to most people in the field (see {\em e.g.} \cite{Veech92}). We present them here only for the sake of completeness.

In what follows, for any saddle connection $s$, we will denote by $\hor(s)$ the length  of the horizontal component of $s$, that is $\hor(s)=|\mathrm{Re}(\omega(s))|$. If $\Del$ is a triangle bounded by the saddle connections $s_1,s_2,s_3$, we define $\hor(\Del)=\max\{\hor(s_i), \; i=1,2,3\}$. Our main result in this section is the following

%As a direct consequence of Proposition~\ref{prop:triang:adapt}, we draw

\begin{Proposition}\label{prop:polygon:construct}%\label{cor:polygon:construct}
 Let $(X,\omega)$  be a translation surface in $\H(2)\sqcup\H(1,1)$ having a simple horizontal cylinder $C$. Assume that every regular leaf of  the vertical foliation of $(X,\omega)$ crosses $C$.
 \begin{itemize}
  \item[(i)] If $(X,\omega) \in \H(2)$, then $(X,\omega)$ can be obtained by identifying the pairs of opposite sides of an octagon $\P=(P_0\dots P_3Q_0\dots Q_3) \subset \R^2$ (see Figure~\ref{fig:H2H11:hor:cyl:triang}), where the vertices are labelled  clockwise, such that
  \begin{itemize}
  \item[$\bullet$] $\overrightarrow{P_iP_{i+1}}=-\overrightarrow{Q_iQ_{i+1}}, \; i=0,1,2$, and $\overrightarrow{P_3Q_0}=-\overrightarrow{Q_3P_0}$,

  \item[$\bullet$] the diagonals $\ol{P_0P_3}$ and $\ol{Q_0Q_3}$ are horizontal, the parallelogram  $(P_0P_3Q_0Q_3)$ is contained in $\P$ and projects to $C \subset X$,

  \item[$\bullet$] for $i=1,2$, the vertical line through  $P_i$ (resp. $Q_i)$ intersects $\ol{P_0P_3}$ (resp. $\ol{Q_0Q_3}$), and the vertical segment from $P_i$ (resp. from $Q_i$) to the intersection is contained in $\P$.
 \end{itemize}
 
% \begin{figure}[htb]
% \centering
% \begin{tikzpicture}[scale=0.35]
% \fill[blue!40!yellow!30] (0,0) -- (6,0) -- (7,2) -- (1,2) -- cycle;
% 
% \draw (0,0) -- ( 1,-5) -- (4,-4) -- (6,0) -- (7,2) -- (6,7) -- (3,6) -- cycle;
% 
% \draw (3,6) -- (7,2)  (0,0) -- (4,-4); 
% 
% \draw[red] (1,2) -- (7,2) (0,0) -- (6,0);
% 
% %\draw (5,1.2) node {\tiny $\Del^+_0$} (2,0.5) node {\tiny $\Del^-_0$} (3.5,4) node {\tiny $\Del^+_1$} (3.5,-2) node {\tiny $\Del^-_1$} (5,5.5) node {\tiny $\Del^+_2$} (2,-3.5) node {\tiny $\Del^-_2$};
% 
% \draw[red] (4,2.3) node {\tiny $s^+$} (3,-0.3) node {\tiny $s^-$};
% 
% \draw (0.5,2) node {\tiny $P_0$} (3,6) node[above] {\tiny $P_1$} (6,7) node[above] {\tiny $P_2$} (7.5,2) node {\tiny $P_3$} (6.5,-0.3) node {\tiny $Q_0$} (4,-4) node[below] {\tiny $Q_1$} (1,-5) node[below] {\tiny $Q_2$} (0,0) node[left] {\tiny $Q_3$};
% \end{tikzpicture}
% \caption{Representation of a surface in $\H(2)$ by a symmetric polygon. The simple horizontal cylinder is represented by the colored region.}
% \label{fig:H2:hor:cyl:triang}
% \end{figure}
% 

\item[(ii)] If $(X,\omega) \in \H(1,1)$, then $(X,\omega)$ can be obtained by identifying the pairs of opposite sides of a decagon $\P=(P_0\dots P_4 Q_0\dots Q_4)$ (see Figure~\ref{fig:H2H11:hor:cyl:triang}), where the vertices are labelled clockwise, such that
 \begin{itemize}
  \item[$\bullet$] $\overrightarrow{P_iP_{i+1}}=-\overrightarrow{Q_iQ_{i+1}}, \; i=0,\dots,3$, and $\overrightarrow{P_4Q_0}=-\overrightarrow{Q_4P_0}$,

  \item[$\bullet$] the diagonals $\ol{P_0P_4}$ and $\ol{Q_0Q_4}$ are horizontal, the parallelogram  $(P_0P_4Q_0Q_4)$ is contained in $\P$ and projects to $C \subset X$,

  \item[$\bullet$] for $i=1,2,3$, the vertical line through $P_i$ (resp. $Q_i$) intersects $\ol{P_0P_4}$ (resp. $\ol{Q_0Q_4}$), and the vertical segment from $P_i$ (resp. from $Q_i$) to the intersection is contained in $\P$.
 \end{itemize}
\end{itemize}
\end{Proposition}

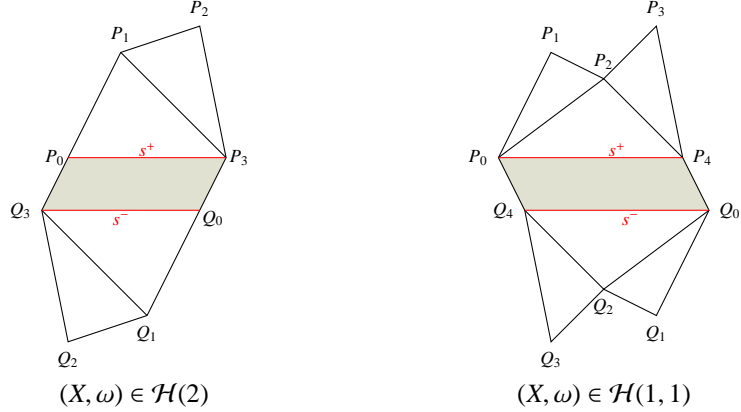
\begin{figure}[htb]
\begin{minipage}[t]{0.4\linewidth}
\centering
\begin{tikzpicture}[scale=0.35]
\fill[blue!40!yellow!30] (0,0) -- (6,0) -- (7,2) -- (1,2) -- cycle;

\draw (0,0) -- ( 1,-5) -- (4,-4) -- (6,0) -- (7,2) -- (6,7) -- (3,6) -- cycle;

\draw (3,6) -- (7,2)  (0,0) -- (4,-4); 

\draw[red] (1,2) -- (7,2) (0,0) -- (6,0);

%\draw (5,1.2) node {\tiny $\Del^+_0$} (2,0.5) node {\tiny $\Del^-_0$} (3.5,4) node {\tiny $\Del^+_1$} (3.5,-2) node {\tiny $\Del^-_1$} (5,5.5) node {\tiny $\Del^+_2$} (2,-3.5) node {\tiny $\Del^-_2$};

\draw[red] (4,2.3) node {\tiny $s^+$} (3,-0.3) node {\tiny $s^-$};

\draw (0.5,2) node {\tiny $P_0$} (3,6) node[above] {\tiny $P_1$} (6,7) node[above] {\tiny $P_2$} (7.5,2) node {\tiny $P_3$} (6.5,-0.3) node {\tiny $Q_0$} (4,-4) node[below] {\tiny $Q_1$} (1,-5) node[below] {\tiny $Q_2$} (0,0) node[left] {\tiny $Q_3$};

\draw (3.5,-7) node {\small $(X,\omega)\in \H(2)$};
\end{tikzpicture}
\end{minipage}
\begin{minipage}[t]{0.4\linewidth}
 \centering
 \begin{tikzpicture}[scale=0.35]
 \fill[blue!40!yellow!30] (0,2) -- (1,0) -- (8,0) -- (7,2) -- cycle;
 \draw (0,2) -- (1,0) -- (2,-5) -- (4,-3) -- (6,-4) -- (8,0) -- (7,2) -- (6,7) -- (4,5) -- (2,6) -- cycle;
 \draw (0,2) -- (4,5) -- (7,2)  (1,0) -- (4,-3) -- (8,0);
 \draw[red] (0,2) -- (7,2) (1,0) -- (8,0);
 
 %\draw (2,1.2) node {\tiny $\Del^+_0$} (6,0.7) node {\tiny $\Del^-_0$} (3,3) node {\tiny $\Del^+_1$} (3.5,-1) node {\tiny $\Del^-_1$} (2,4.5) node {\tiny $\Del^+_2$} (5.5,-3) node {\tiny $\Del^-_2$} (5.5,5) node {\tiny $\Del^+_3$}  (2.5,-3) node {\tiny $\Del^-_3$};

 \draw[red] (4.5,2.3) node {\tiny $s^+$} (5,-0.3) node {\tiny $s^-$};
 
 \draw (0,2) node[left] {\tiny $P_0$} (2,6) node[above] {\tiny $P_1$} (4,5) node[above] {\tiny $P_2$} (6,7) node[above] {\tiny $P_3$} (7,2) node[right] {\tiny $P_4$} (8,0) node[right] {\tiny $Q_0$} (6,-4) node[below] {\tiny $Q_1$} (4,-3) node[below] {\tiny $Q_2$} (2,-5) node[below] {\tiny $Q_3$} (1,0) node[left] {\tiny $Q_4$};
 
 \draw (4,-7) node {\small $(X,\omega)\in \H(1,1)$};
 \end{tikzpicture}
\end{minipage}
% \begin{minipage}[t]{0.4\linewidth}
%  \centering
%  \begin{tikzpicture}[scale=0.35]
%  \fill[blue!40!yellow!30] (0,2) -- (1,0) -- (8,0) -- (7,2) -- cycle;
%  \draw (0,2) -- (2,-2) -- (5,-5) -- (7,-3) -- (8,0) -- (6,4) -- (3,7) -- (1,5) -- cycle;
%  \draw (1,5) -- (6,4) -- (0,2)  (8,0) -- (2,-2) -- (7,-3);
%  \draw[red] (0,2) -- (7,2) (1,0) -- ( 8,0);
% 
%  %\draw (2,1.2) node {\tiny $\Del^+_0$} (6,0.7) node {\tiny $\Del^-_0$} (5,3) node {\tiny $\Del^+_1$} (2.5,-1) node {\tiny $\Del^-_1$} (2,4) node {\tiny $\Del^+_2$} (6,-1.5) node {\tiny $\Del^-_2$} (3,5.5) node {\tiny $\Del^+_3$} (5,-3.5) node {\tiny $\Del^-_3$};
%  
%  \draw[red] (3,2.4) node {\tiny $s^+$} (4.5,-0.3) node {\tiny $s^-$};
%  
%  \draw (0,2) node[left] {\tiny $P_0$} (1,5) node[left] {\tiny $P_1$} (3,7) node[above] {\tiny $P_2$} (6,4) node[right] {\tiny $P_3$} (7,2) node[right] {\tiny $P_4$} (8,0) node[right] {\tiny $Q_0$}
%  (7,-3) node[right] {\tiny $Q_1$} (5,-5) node[below] {\tiny $Q_2$} (2,-2) node[left] {\tiny $Q_3$} (1,0) node[left] {\tiny $Q_4$};
% \end{tikzpicture}
% \end{minipage}

\caption{Representations of surfaces in $\H(2)$ and $\H(1,1)$  symmetric polygons. The simple horizontal cylinder is represented by the colored parallelogram.}
\label{fig:H2H11:hor:cyl:triang}
\end{figure}

\begin{proof}
Cut off $C$ from $X$, and identify the geodesic segments in the boundary of the resulting surface, we then obtain either a slit torus (if $(X,\omega)\in \H(2)$) or a surface in $\H(2)$ with a marked saddle connection (if $(X,\omega)\in \H(1,1)$). Let $(X',\omega')$ denote the new surface, and $s'$ the marked saddle connection. If $(X',\omega')$ is a slit torus, then there is a unique involution of $X'$ acting by $-\Id$ on $H_1(X',\Z)$ and exchanges the endpoints of $s'$. By a slight abuse of notation, we will call this involution the hyperelliptic involution of $X'$. Thus, in both cases, $s'$ is invariant by the hyperelliptic involution. 

By assumption all the regular vertical leaves of  $X'$ intersect $s'$. Let $\{\Del_i^{\pm},\; i=1,\dots,k\}$ be the triangulation of $X'$ provided by Lemma~\ref{lm:slit:tor:triang} and Lemma~\ref{lm:H2:triang} ($k=2$ if $(X',\omega') \in \H(0,0)$, $k=3$ if $(X',\omega')\in \H(2)$). We can represent $C$ by a parallelogram in $\R^2$. The polygon $\P$ is obtained from this parallelogram by gluing successively the triangles $\Del_1^+,\dots,\Del_k^+$, then $\Del_1^-,\dots,\Del_k^-$.
\end{proof}

\begin{Lemma}\label{lm:slit:tor:triang}
Let $(X,\omega,s)$ be a slit torus. Let $\hinv$ be the elliptic involution of $X$ that exchanges the endpoints $P_1,P_2$ of $s$. Assume that all the leaves of vertical foliation meet $s$. Then there exists a unique triangulation of $X$ into $4$ triangles $\Del_1^\pm, \Del_2^\pm$ with vertices in $\{P_1,P_2\}$, such that
\begin{itemize}
\item[$\bullet$] $\Del_i^+$ and $\Del_i^-$ are exchanged by $\hinv$,

\item $s$ is contained in both $\Del^+_1$ and $\Del_1^-$,

\item for $i=1,2$, the union $\Del_i^+\cup \Del_i^-$ is a cylinder in $X$,

\item $\Del_1^+$ is adjacent to $\Del_1^-$ and $\Del_ 2^+$, $\Del_1^-$ is adjacent to $\Del_1^+$ and $\Del_2^-$,

\item[$\bullet$] $\hor(\Del_1^\pm)=\hor(s)$,  and $\hor(\Del_2^\pm)=\hor(c^+)$, where $c^+$ is the unique  common side of $\Del_2^+$ and  $\Del_1^+$.
\end{itemize}
There are two possible configurations for this triangulation which are shown in Figure~\ref{fig:sl:tor:triang}.
\begin{figure}[htb]
\begin{minipage}[t]{0.45\linewidth}
 \centering
 \begin{tikzpicture}[scale=0.35]

 \draw (0,5) -- (1,0) -- (4,1) -- (6,5) -- (5,10) -- (2,9) -- cycle;
 \draw (2,9) -- (6,5) (0,5) -- (4,1);
 \draw[red] (0,5) -- (6,5);

 \foreach \x in {(0,5),(1,0),(2,9),(4,1),(5,10), (6,5)} \filldraw[fill=black] \x circle (3pt);

 \draw[red] (3,5) node[above] {\tiny $s$};
 \draw (2,7) node {\tiny $\Delta^+_1$} (4,3) node {\tiny $\Delta^-_1$} (4.5,8) node {\tiny $\Delta^+_2$} (2,1.5) node {\tiny $\Delta^-_2$};
 \draw (0,5) node[left] {\tiny $P_1$} (6,5) node[right] {\tiny $P_2$} (2,9) node[left] {\tiny $P_2$} (4,1) node[right] {\tiny $P_1$} (5,10) node[above] {\tiny $P_1$} (1,0) node[below] {\tiny $P_2$};
 \end{tikzpicture}
\end{minipage}
\begin{minipage}[t]{0.45\linewidth}
 \centering
 \begin{tikzpicture}[scale=0.35]

 \draw (0,5) -- (-1,0) -- (-4,1) -- (-6,5) -- (-5,10) -- (-2,9) -- cycle;
 \draw (-2,9) -- (-6,5) (0,5) -- (-4,1);
 \draw[red] (0,5) -- (-6,5);

 \foreach \x in {(0,5),(-1,0),(-2,9),(-4,1),(-5,10), (-6,5)} \filldraw[fill=black] \x circle (3pt);

 \draw[red] (-3,5) node[above] {\tiny $s$};
 \draw (-2,7) node {\tiny $\Delta^+_1$} (-4,3) node {\tiny $\Delta^-_1$} (-4.5,8) node {\tiny $\Delta^+_2$} (-2,1.5) node {\tiny $\Delta^-_2$};
 \draw (-6,5) node[left] {\tiny $P_1$} (0,5) node[right] {\tiny $P_2$} (-2,9) node[right] {\tiny $P_1$} (-4,1) node[left] {\tiny $P_2$} (-5,10) node[above] {\tiny $P_2$} (-1,0) node[below] {\tiny $P_1$};
 \end{tikzpicture}
\end{minipage}

\caption{Triangulation of a slit torus.}
\label{fig:sl:tor:triang}
\end{figure}
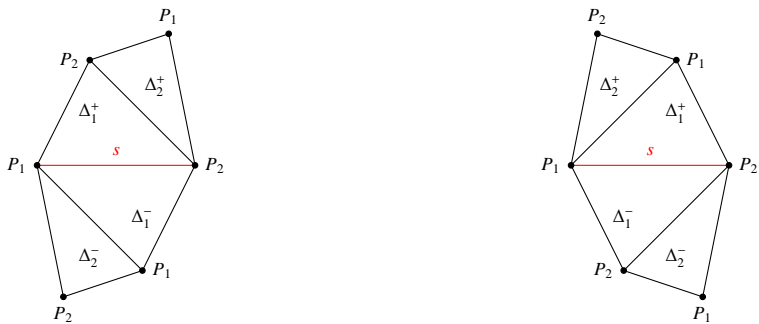
\end{Lemma}
\begin{proof}
By Lemma~\ref{lm:slit:tor:decomp}, we know that there exists a pair of simple closed geodesics $c^+,c^-$ passing through the endpoints of $s$  that cut $X$ into 2 cylinders satisfying $\hor(c^\pm) \leq \hor(s)$. One of the cylinders cut out by $c^\pm$ contains $s$, we denote it by $C_1$, the other one is denoted by $C_2$.  Note that we must have $\hor(c^\pm) >0$, otherwise there are vertical leaves that do not meet $s$. It is easy to see that we get the desired triangulation by adding some geodesic segments in  $C_1$ and $C_2$ joining the endpoints of $s$.
\end{proof}

% There exists a unique geodesic segment $d$ in $C_1$ joining $P_1$ and $P_2$ such that $\inter(d)\cap\inter(s)=\vide$ and $\hor(d)+\hor(c^\pm)=\hor(s)$.  Observe that $d$ and $s$ decompose $C_1$ into two triangles denoted by $\Del_1^\pm$, where $\Del_1^+$ is bounded by $d,s$ and $c^+$.
% 
% Since $\hor(c^+)>0$, there exist two geodesic segments $e,f$ in $C_2$ joining $P_1$ and $P_2$ such that $\hor(d)+\hor(e)=\hor(c^+)$ which decompose $C_2$ into 2 triangles. Let us denote the two triangles in $C_2$ by $\Del_2^\pm$, where $\Del_2^+$ is the one that contains $c^+$. By construction, $\Del_i^\pm$ are permuted by the elliptic involution, and clearly we have
% $$
% \hor(\Del_1^\pm)=\hor(s), \text{ and } \hor(\Del_2^\pm)=\hor(c^+).
% $$
% 
% \noindent It is now easy to check that all the properties listed in the statement are satisfied.
% \end{proof}

%In what follows,  $(X,\omega)$ will be a translation surface of genus two. If $a$ and $b$ are saddle connections or simple closed geodesics in $X$ that are not parallel, we denote by $\iota(a,b)$ the number of intersection points of $a$ and $b$ except the zeros of $\omega$.

%The following lemma provides a canonical way to triangulate a surface in $\H(2)$ with a horizontal saddle connection invariant by the hyperelliptic involution.

\begin{Lemma}\label{lm:H2:triang}
Let $(X,\omega)$ be a surface in $\H(2)$ and $s$ be a saddle connection on $X$ invariant by the hyperelliptic involution $\hinv$. We assume that $s$ is horizontal and all the leaves of the vertical foliation meet $s$. Then we can triangulate $X$ into 6 triangles $\Del_i^\pm, \;  i=1,2,3$, whose sides are saddle connections, satisfying the following
\begin{itemize}
\item[$\bullet$] $\hinv(\Del_i^+)=\Del_i^-, \; i=1,2,3$,

\item[$\bullet$] $\Del_1^+$ and $\Del_1^-$ contain $s$, and $\hor(\Del_1^\pm)=\hor(s)$,

\item[$\bullet$] $\Del_2^+$ has a unique common side with $\Del_1^+$ which will be denoted by $a^+$, and $\hor(\Del_2^+)=\hor(a^+)$.

\item[$\bullet$] $\Del_3^+$ either has a unique common side $b^+$ with $\Del_1^+$ and $\hor(\Del_3^+)=\hor(b^+)$, or $\Del_3^+$ has a unique common side $c^+$ with $\Del_2^+$ and $\hor(\Del_3^+)=\hor(c^+)$.
\end{itemize}
This triangulation is unique. The configurations of the triangles $\Del_i^\pm, \; i=1,2,3$,  are shown in Figure~\ref{fig:H2:triang}.
\begin{figure}[htb]
\begin{minipage}[t]{0.3\linewidth}
 \centering
 \begin{tikzpicture}[scale=0.35]
  \draw (0,0) -- (1,-5) -- (3,-3) -- (5,-4) -- (7,0) -- (6,5) -- (4,3) -- (2,4)  -- cycle;
  \draw (0,0) -- (3,-3) -- (7,0) -- (4,3) -- cycle;
  \draw[red] (0,0) -- (7,0);

  \draw (4,1) node {\tiny $\Del^+_1$} (3,-1) node {\tiny $\Del^-_1$} (2,2.5) node {\tiny $\Del^+_2$} (5,-2.5) node {\tiny $\Del^-_2$} (5.5,3) node {\tiny $\Del^+_3$} (1.5,-3) node {\tiny $\Del^-_3$};
  \draw[red] (3,0) node[above] {\tiny $s$};
 \end{tikzpicture}
\end{minipage}
\begin{minipage}[t]{0.3\linewidth}
 \centering
 \begin{tikzpicture}[scale=0.35]
 \draw (0,0) -- (1,-3) -- (3,-5) -- (6,-2) -- (7,0) -- (6,3) -- (4,5) -- (1,2) -- cycle;
 \draw (6,3) -- (1,2) -- (7,0) (0,0) -- (6,-2) --  (1,-3);
 \draw[red] (0,0) -- (7,0);

 \draw (2,1) node {\tiny $\Del^+_1$} (5,-1) node {\tiny $\Del^-_1$} (5,1.5) node {\tiny $\Del^+_2$} (1.5,-2) node {\tiny $\Del^-_2$} (4,3.5) node {\tiny $\Del^+_3$} (3,-3.5) node {\tiny $\Del^-_3$};
 \draw[red] (3.5,0) node[above] {\tiny $s$};
 \end{tikzpicture}
\end{minipage}
\begin{minipage}[t]{0.3\linewidth}
 \centering
 \begin{tikzpicture}[scale=0.35]
 \draw (0,0) -- (-1,-3) -- (-3,-5) -- (-6,-2) -- (-7,0) -- (-6,3) -- (-4,5) -- (-1,2) -- cycle;
 \draw (-6,3) -- (-1,2) -- (-7,0) (0,0) -- (-6,-2) --  (-1,-3);
 \draw[red] (0,0) -- (-7,0);

 \draw (-2,1) node {\tiny $\Del^+_1$} (-5.5,-1) node {\tiny $\Del^-_1$} (-5,2) node {\tiny $\Del^+_2$} (-2,-2) node {\tiny $\Del^-_2$} (-4,3.5) node {\tiny $\Del^+_3$} (-3,-3.5) node {\tiny $\Del^-_3$};
 \draw[red] (-3.5,0) node[above] {\tiny $s$};
 \end{tikzpicture}
\end{minipage}

\caption{Triangulation of surfaces in $\H(2)$.}
\label{fig:H2:triang}
\end{figure}
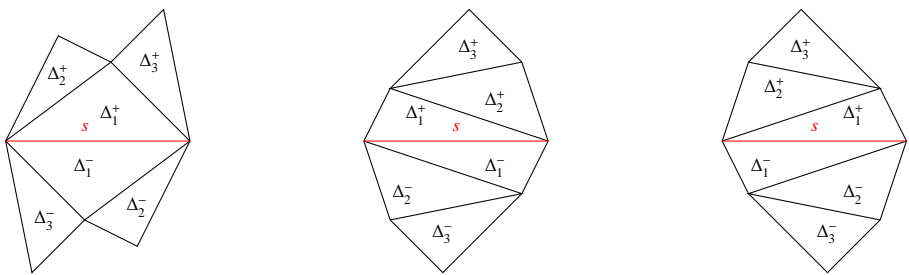
\end{Lemma}
\begin{proof}
From Lemma~\ref{lm:embd:par}, we see that there exist a parallelogram $\P \subset \R^2$ and a locally isometric map $\varphi :\P \ra X$ that maps a diagonal of $\P$ to $s$. By construction, $\varphi(\P)$ is decomposed into two embedded triangles $\Del_1^\pm$, where $\Del_1^+$ is the one above $s$,  both of which satisfy $\hor(\Del_1^\pm)=\hor(s)=|s|$. Note also that $\hinv(\Del_1^+)=\Del_1^-$.

Let us denote the non-horizontal sides of $\Del_1^+$ by $a^+$ and $b^+$, and their images by $\hinv$ by $a^-$ and $b^-$ respectively. If both of $a^+$ and $b^+$ are invariant by $\hinv$ then $X=\varphi(\P)$, which implies that $X$ is a torus, and we have a contradiction. Therefore,  we only have two cases:
\begin{itemize}
\item[a)] None of $a^+,b^+$ is invariant by $\hinv$. In this cases, by Lemma~\ref{lm:sc:no:inv:types} the complement of $\varphi(\P)$ is the disjoint union of two cylinders bounded by $a^\pm$ and $b^\pm$ respectively. Note that none of  $a^+$ and $b^+$ is vertical, otherwise there would be vertical leaves that do not meet $s$. We can then triangulate the cylinders bounded by $a^\pm$ and $b^\pm$ in the same way as in Lemma~\ref{lm:slit:tor:triang}.

\item[b)] One of $a^+,b^+$ is invariant by $\hinv$. We can assume that $b^+$ is invariant by $\hinv$. In this case $\varphi(\P)$ is actually a simple cylinder bounded by $a^\pm$. The complement of $\varphi(\P)$ is then a slit torus $(X_1,\omega_1,s_1)$, where $s_1$ is the identification of $a^\pm$. From the assumption that all the vertical leaves meet $s$, we derive that $a^\pm$ are not vertical. Thus we can follow the same argument as in Lemma~\ref{lm:slit:tor:triang} to get the desired triangulation.
\end{itemize}
\end{proof}

%The following proposition, which relies on Lemma~\ref{lm:slit:tor:triang} and Lemma~\ref{lm:H2:triang},  gives us a canonical way to triangulate translation surfaces in genus two with a simple %horizontal cylinder. Those triangulations will be  used in Section~\ref{sec:reduce:inters:nb}.

\section{Cylinders and decompositions}\label{sec:aux:lemmas}
In this section, we give the proofs of some lemmas which are used in Section~\ref{sec:hyperbolic}.
\begin{Lemma}\label{lm:cp:degen:cyl}
Let  $(X,\omega)\in \H(2)\sqcup\H(1,1)$ be a completely periodic surface in the sense of Calta. If $C$ is a  degenerate cylinder in $X$, then the direction of $C$ is periodic, that is $X$ is decomposed into cylinders in the direction of $C$.
\end{Lemma}
\begin{proof}
If $(X,\omega) \in \H(2)$ then $(X,\omega)$ is a Veech surface, thus the direction of any saddle connection is periodic and we are done. Assume now that  $(X,\omega) \in \H(1,1)$. In $\H(1,1)$, we have a local action of $\C$ which  only changes the relative periods and leaves the absolute periods invariant. Orbits of this local action are leaves of the kernel foliation. It is well known that the any eigenfom locus is invariant by this local action.

 %Let $c$ be any path from $x_1$ to $x_2$ %and $v$ be a vector in $\R \simeq \C$ with $|v|$ small enough. Then there exists a unique surface  $(X',\omega')$ close to $(X,\omega)$ such that $(X',\omega')$ %has the same absolute periods as $(X,%\omega)$ and $\omega'(c)=\omega(c)+v$. We will write $(X',\omega')=(X,\omega)+v$.\medskip

Let us  label the zeros of $\omega$ by $x_1,x_2$ and define the orientation of any path connecting  $x_1$ and $x_2$ to be from $x_1$ to $x_2$. Using this local action of $\C$, we can collapse the two zeros of $\omega$ as follows,  let $s$ be a saddle connection invariant by the hyperelliptic involution satisfying the following condition:
\begin{flushleft}
$(\mathcal{S})$ if there exists another saddle connection $s'$  joining $x_1$ and $x_2$ such that $\omega(s')=\lambda\omega(s)$ with $\lambda \in (0;+\infty)$,  then  we  have $\lambda >1$.
\end{flushleft}

We can then  reduce the length of $s$ to zero by moving in the kernel foliation leaf of $(X,\omega)$, the resulting surface is an eigenform in $\H(2)$ having the same absolute periods as $(X,\omega)$. The condition on $s$ implies that $x_1$ and $x_2$ do not collide before $s$ is reduced to a point. For a proof of this fact, we refer to \cite{KZ03, LN14}. Remark that the new surface in $\H(2)$ is a Veech surface. \medskip

Without loss of generality, we can assume that $C$ is horizontal. By definition, $C$ is the union of two saddle connections $s_1,s_2$ both invariant by the hyperelliptic involution, and up to a renumbering we have $\omega(s_1) \in \R_{>0}, \omega(s_2) \in \R_{<0}$.

Assume that  neither of $s_1,s_2$ satisfies $(\mathcal{S})$, then there exist two other saddle  connections $s'_1,s'_2$ such that $\omega(s'_i)=\lambda_i\omega(s_i)$, with $\lambda_i \in (0;1)$. This implies that there are four horizontal saddle connections on $X$. Since $(X,\omega) \in \H(1,1)$, there are at most 4 saddle connections in a fixed direction, and this maximal number is realized if and only if the direction is periodic. Thus, in this case we can conclude that $X$ is horizontally periodic.

Let us now assume that one of $s_1,s_2$, say $s_1$, satisfies the condition $(\mathcal{S})$.  We can then collapse $x_1,x_2$ along $s_1$ to get a Veech surface $(X_0,\omega_0)\in \H(2)$.  Since $\omega(s_2)-\omega(s_1)$ is an absolute period, it stays unchanged along the collapsing procedure. Therefore, $s_2$ persists in $X_0$, and we have $\omega_0(s_2)=\omega(s_2)-\omega(s_1) \in \R$.  In particular, $(X_0,\omega_0)$ has a horizontal saddle connection, and because $(X_0,\omega_0)$ is a Veech surface, it must be horizontally periodic. It follows that $(X,\omega)$ is also horizontally periodic. This completes the proof of the lemma.
\end{proof}

\begin{Lemma}\label{lm:2djt:sim:cyl}
Let  $(X,\omega) \in \H(1,1)$. Let $C$ be a horizontal (possibly degenerate) cylinder in $X$, and $D$ be a vertical simple cylinder disjoint from $C$. Then either
\begin{itemize}
\item[(a)] there is another  simple cylinder $E$ disjoint from $C\cup D$ such that the complement of $C\cup D \cup E$ is the union of two embedded triangles, or

\item[(b)] there exist a pair of homologous saddle connections $s_1,s_2$ that decompose $X$ into two two slit tori $(X',\omega',s')$ and $(X'',\omega'',s'')$ such that $C$ is contained in $X'$ and $D$ is contained in $X''$.
\end{itemize}

\end{Lemma}
\begin{proof}
We first consider the case $C$ is not degenerate. In this case, the complement of $\ol{C}$ in $X$ is either (1) empty, (2) a horizontal simple cylinder, (3) the disjoint union of two horizontal simple cylinders, (4) a torus with a horizontal slit, or (5) a surface $(\hat{X},\hat{\omega})\in \H(2)$ with a marked horizontal saddle connection $s$. Since we have a vertical simple cylinder disjoint from $C$, only (4) and (5) can occur. In case (4), we automatically have two slit tori, one of which is the closure of $C$, and the other one must contain $D$. Therefore we get case (b) of the statement of the lemma.

Let us now assume that we are in case (5).  In this case $C$ must be a simple horizontal cylinder, and the saddle connection $s$ in $\hat{X}$ corresponds to the boundary of $C$. Note that $s$ is invariant by the hyperelliptic involution $\hat{\hinv}$ of $\hat{X}$. Let $\varphi: \P \ra \hat{X}$ be the embedded parallelogram associated to $s$. Let $a^\pm$ and $b^\pm$ be the images by $\varphi$ of the sides of $\P$ such that $\hat{\hinv}(a^+)=a^-$ and $\hat{\hinv}(b^+)=b^-$. Remark that $D$ must be disjoint from $\varphi(\P)$ since any vertical geodesic intersecting $\varphi(\inter(\P))$ must intersect $\inter(s)$, and hence $C$, but we have assumed that $D$ is  disjoint from $C$.

If  $a^+=a^-$ and $b^+=b^-$ then $\hat{X}$ must be a torus, and we have a contradiction. Therefore, we only have two cases:
\begin{itemize}
\item[$\bullet$]  $a^+\neq a^-$ and $b^+\neq b^-$. In this case, the complement of $\varphi(\P)$ is the disjoint union of two simple cylinders. Since $D$ is contained in this union, $D$ must be one of the two. Let us denote the other one by $E$. To obtain $(X,\omega)$ from $(\hat{X},\hat{\omega})$, we need to slit open $s$ and glue back $C$. Consequently, we see that $(X,\omega)$ has three disjoint simple cylinders $C,D,E$. The complement of $C\cup D \cup E$ is the union of two embedded triangles, which are the images of the triangles in $\P$ cut out by $s$. Thus, we get case (a) of the statement of the lemma.

\item[$\bullet$] $a^+=a^-$ and $b^+\neq b^-$. In this case, $\varphi(\P)$ is a simple cylinder bounded by $b^\pm$. The complement of  $\varphi(\P)$ is then a slit torus $(X'',\omega'',s'')$ with the slit $s''$ corresponding to $b^\pm$.  We can view $(X'',\omega'',s'')$  as  a subsurface of $X$. Observe that  $D$ must be contained in $(X'',\omega'')$ and disjoint from  the slit $s''$, since otherwise a core curve of $D$ must cross $C$.  The complement of $(X'',\omega'',s'')$ is another slit torus $(X',\omega',s')$  which is obtained by slitting $\varphi(\P)$ along $s$ and gluing back $C$.   Therefore, we get case (b) of the statement of the lemma.
\end{itemize}

\medskip

%Consider now the case $C$ is degenerate. By Lemma~\ref{lm:degen:cyl:H2:H11}, we know that either (1) $C$ is contained in a slit torus $(X',\omega',s')$ such that any vertical leaf crossing $s'$ intersects $C$, or (2) there are two simple cylinders $C_1,C_2$  disjoint from $C$ such that any vertical leaf crossing $C_1$ or $C_2$ intersects $C$. In case (1), $D$ must be contained in the complement of $(X',\omega',s')$ which is another slit torus. Hence we get case (b) of the statement of the lemma. In case (2), observe that the complement of $C_1\cup C_2$ is the union of two embedded parallelograms which cannot contain a vertical cylinder. Since $D$ is disjoint from $C$, up to a renumbering we must have $D=C_1$. Let $E=C_2$. By cutting and pasting, it is not difficult to see that the complement of $C\cup D \cup E$ is the union of two embedded triangles. Thus we get case (a) of the statement of the lemma.

Assume now $C$ is degenerate. By Lemma~\ref{lm:degen:cyl:deform}, there exist deformations $(X_t,\omega_t), \, t \in [0,\eps)$, of $(X,\omega)$ such that $C$ corresponds to a simple horizontal cylinder $C_t$ in $X_t$, which has the same circumference as $C$ and height equal to $t$. By construction, $D$ corresponds to a simple vertical cylinder $D_t$ in $X_t$ which is disjoint from $C_t$. Observe that $C_t$ and $D_t$ satisfy case (5) above. Therefore,  by the preceding arguments, the conclusion of the lemma is true for $C_t$ and $D_t$. In either case, the corresponding decomposition of $X_t$ persists as $t \ra 0$, which implies that we have the same decomposition on $(X,\omega)$.
\end{proof}

%\section{Cylinders with uniformly bounded width}
In what follows if $u=(u_1,u_2)$ and $v=(v_1,v_2)$ are two vectors in $\R^2$, we denote $ u\wedge v:= \det\left(\begin{smallmatrix} u_1 & v_1 \\ u_2 & v_2 \end{smallmatrix}\right)$, and $|u|,|v|$ are the Euclidean norms of $u$ and $v$ respectively.

\begin{Lemma}\label{lm:s:tor:cyl:large:wd}
Given a constant $L>0$, let
\begin{equation}\label{eq:def:L1:b}
L_1:=3\max\{f(L),f(2\del)\}
\end{equation}
\noindent where $f(x)=\sqrt{x^2+1/x^2}$, and $\del:= (3/4)^{\frac{1}{4}}$. Then for any slit torus $(X,\omega,s)$ with $\Aa(X,\omega)=1$, and $|s| < L$, there exists in $X$ a cylinder disjoint from $s$  with  area at least $1/2$ and circumference bounded above by $L_1$.
\end{Lemma}

\begin{proof}
Let $\Lambda$ be the lattice in $\C$ such that $(X,\omega)$ can be identified with $(\C/\Lambda, dz)$.   Since $\Lambda$ has co-volume one, there exists a vector $v \in \Lambda$ such that $|v| \leq \del$.  Set $u=\omega(s)\in \C \simeq \R^2$. \medskip

Let us first consider the case $|u| \leq \frac{1}{2\del}$. We then have
$$
|u\wedge v|\leq  |u||v| \leq 1/2.
$$
The vector $v$ corresponds to  a simple closed geodesic $c$ on $X$. The inequality above implies that there exist a pair of simple closed geodesics parallel to $c$ cutting $X$ into two cylinders, one of which contains $s$ denoted by $C$, the other one denoted by $C'$ consists of closed geodesics parallel to $c$ that do not intersect $s$ (see \cite[Lem. 4.1]{Ng11} or \cite[Th. 7.2]{McM_spin}).  Note that the circumferences of both $C$ and $C'$ are $|v| \leq \del$. Since $\Aa(C)=|u\wedge v| \leq 1/2$, we have $\Aa(C') \geq 1/2$. Thus $C'$ has the required  properties. \medskip

We can now turn into the case $\frac{1}{2\del} \leq |s| \leq L$. By definition, we have $f(|s|) \leq L_1/3$. By multiplying $\omega$ by a complex number of module $1$, which does not change the area of $X$ and the length of $s$, we can assume that $s$ is horizontal. From Lemma~\ref{lm:embd:par}, we know that there exists a local isometry $\varphi$ from a parallelogram $\P \subset \R^2$ into $X$ such that a horizontal diagonal  of $\P$ is mapped to $s$. Since $X$ is a torus, $C:=\varphi(\P)$ is actually a cylinder in $X$.  Let $\eta$ be the distance from the highest point of $\P$ to its horizontal diagonal. By construction, we have $\Aa(C)=\Aa(\P)=\eta|s| \leq \Aa(X,\omega)=1$. Thus $\eta \leq 1/|s|$. Remark that the boundary components of $C$ are the images by $\varphi$ of two opposite sides of $\P$. Hence the circumference of $C$ is bounded by $\sqrt{|s|^2+\eta^2} \leq f(|s|) \leq L_1/3$.

Observe that the complement of $C$ is another cylinder $C'$ in $X$ sharing the same boundary with $C$. If $\Aa(C') \geq 1/2$ then we are done. Let us consider the case $\Aa(C') < 1/2$, which means that $\Aa(C) > 1/2 > \Aa(C')$. By cutting and pasting, we can also realize $C$ as a parallelogram ${\bf Q}=(P_1P_2P_3P_4)$ with two horizontal sides $\ol{P_1P_2}$ and $\ol{P_4P_3}$ identified with $s$.  Note that the distance between $\ol{P_1P_2}$ and $\ol{P_4P_3}$ is $\eta$. We can then realize $C'$ as a parallelogram ${\bf Q}'=(P_2P_3P_5P_6)$ adjacent to ${\bf Q}$ where $P_5$ is contained in the horizontal stripe bounded by the lines supporting $\ol{P_1P_2}$ and $\ol{P_4P_3}$ (see Figure~\ref{fig:slit:tor:cyl}). Let $P'_6$ and $P'_5$ be the intersections of the line supporting $\ol{P_5P_6}$ and the lines supporting $\ol{P_1P_2}$ and $\ol{P_4P_3}$ respectively.
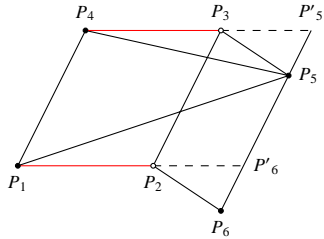
\begin{figure}[htb]
\begin{tikzpicture}[scale=0.3]
\draw  ( 6,0) -- (9,-2) -- (13,6)  (0,0) --  (3,6) -- (12,4) -- (0,0)  (6,0) -- (9,6) -- (12,4);
\draw[red] (0,0) -- (6,0) (3,6) -- (9,6);
\draw[dashed] (9,6) -- (13,6) (6,0) -- (10,0);

\foreach \x in{(0,0), (3,6), (9,-2), (12,4)} \filldraw[fill=black] \x circle (3pt);
\foreach \x in {(6,0), (9,6)} \filldraw[fill=white] \x circle (3pt);

\draw (0,0) node[below] {\tiny $P_1$} (6,0) node[below] {\tiny $P_2$} (9,-2) node[below] {\tiny $P_6$} (10,0) node[right] {\tiny ${P'}_6$} (12,4) node[right] {\tiny $P_5$} (13,6) node[above] {\tiny ${P'}_5$} (9,6) node[above] {\tiny $P_3$} (3,6) node[above] {\tiny $P_4$};
\end{tikzpicture}
\caption{Cylinder with bounded circumference and area $\geq 1/2$ in a slit torus.}
\label{fig:slit:tor:cyl}
\end{figure}

Clearly we have $\Aa(C')=\Aa({\bf Q}')=\Aa((P_2P_3{P'}_5{P'}_6))$.  Since $\Aa(C') < \Aa(C)$, we have $|\ol{P_2{P'}_6}| < |\ol{P_1P_2}|$, and $|\ol{P_1{P'}_6}|< 2|\ol{P_1P_2}|\leq 2L$. If ${P'_6}\equiv P_6$, then $X$ has a horizontal cylinder $C_0$ with circumference equal $|\ol{P_1{P'}_6}|$ and area equal $1$. Clearly the core curves of $C_0$ do not intersect $s$, therefore $C_0$ has the required properties. If $P_6\neq {P'}_6$, then by construction, $\ol{P_1P_5}$ and $\ol{P_4P_5}$ project to two simple closed geodesics in $X$, denoted by $c_1$ and $c_2$ respectively. These closed geodesics  meet $s$ only at one of its endpoints. Let $d_1$ and $d_2$ be respectively the simple closed geodesics parallel to $c_1$ and $c_2$ passing through the  other endpoint of $s$. Observe that $c_1$ and $d_1$ (resp. $c_2$ and $d_2$) cut $X$ into 2 cylinders, one of which contains $s$  will be denoted by $C_1$ (resp. $C_2$), the other one is denoted by $C'_1$ (resp. $C'_2$). Now, remark that
$$
\Aa(C_1)=|\overrightarrow{P_1P_5}\wedge\overrightarrow{P_1P_2}|, \text{ and } \Aa(C_2)=|\overrightarrow{P_4P_5}\wedge\overrightarrow{P_4P_3}|.
$$
\noindent Since
$$
|\overrightarrow{P_1P_5}\wedge\overrightarrow{P_1P_2}|+|\overrightarrow{P_4P_5}\wedge\overrightarrow{P_4P_3}| =|\overrightarrow{P_1P_2}\wedge\overrightarrow{P_1P_4}|=\Aa(C) \leq 1,
$$
\noindent we have either $\Aa(C_1) \leq 1/2$, or $\Aa(C_2) \leq 1/2$. Assume that $\Aa(C_1) \leq 1/2$, then $\Aa(C'_1) \geq 1/2$. Remark that
$$
|c_1|=|\ol{P_1P_5}| \leq |\ol{P_1{P'}_6}|+|\ol{{P'}_6P_5}| \leq 2L_1/3+L_1/3=L_1.
$$
\noindent Thus we can conclude that $C'_1$ satisfies the statement of the lemma. In the case $\Aa(C_2) \leq 1/2$, the same argument shows that the complement $C'_2$ of $C_2$ has the required properties. The proof of the lemma is now complete.
\end{proof}

\end{document}